\documentclass[a4paper]{article}%
\usepackage{amsfonts}
\usepackage{amsmath}
\usepackage{amssymb}
\usepackage{amstext}
\usepackage{graphicx}
\usepackage[a4paper]{geometry}
\usepackage[page]{appendix}
\usepackage{colortbl}%
\providecommand{\U}[1]{\protect\rule{.1in}{.1in}}
\newtheorem{theorem}{Theorem}[section]

\newtheorem{corollary}[theorem]{Corollary}

\newtheorem{lemma}[theorem]{Lemma}

\newtheorem{remark}[theorem]{Remark}

\newenvironment{proof}[1][Proof]{\noindent\textbf{#1.} }{\ \rule{0.5em}{0.5em}}
\numberwithin{equation}{section}
\begin{document}

\title{On the Origin of Minnaert Resonances}
\author{Andrea Mantile\\Laboratoire de Math\'{e}matiques de Reims, UMR9008 CNRS,\\Universit\'{e} de Reims Champagne-Ardenne, Moulin de la Housse\\BP 1039, 51687 Reims, France
\and Andrea Posilicano\\DiSAT, Sezione di Matematica, Universit\`a dell'Insubria, via Valleggio 11,
I-22100 Como, Italy
\and Mourad Sini\\RICAM, Austrian Academy of Sciences, Altenbergerstrasse 69, A-4040, Linz, Austria}
\date{}
\maketitle

\begin{abstract}
It is well known that the presence, in a homogeneous acoustic medium, of a
small inhomogeneity (of size $\varepsilon$), enjoying a high contrast of both
its mass density and bulk modulus, amplifies the generated total fields. This
amplification is more pronounced when the incident frequency is close to the
Minnaert frequency $\omega_{M}$. Here we explain the origin of such a
phenomenon: at first we show that the scattering of an incident wave of
frequency $\omega$ is described by a self-adjoint $\omega$-dependent
Schr\"{o}dinger operator with a singular $\delta$-like potential supported at
the inhomogeneity interface. Then we show that, in the low energy regime
(corresponding in our setting to $\varepsilon\ll1$) such an operator has a
non-trivial limit (i.e., it asymptotically differs from the Laplacian) if and
only if $\omega=\omega_{M}$. The limit operator describing the non-trivial
scattering process is explicitly determined and belongs to the class of point
perturbations of the Laplacian. When the frequency of the incident wave
approaches $\omega_{M}$, the scattering process undergoes a transition between
an asymptotically trivial behaviour and a non-trivial one.

\end{abstract}

\section{Introduction}

Models related to the wave propagation in the presence of small scaled but
highly contrasted inhomogeneities appear in different areas of applied
sciences as in acoustics, electromagnetism and elasticity where these
inhomogeneities model micro-bubbles, nano-particles and micro- or
nano-inclusions, respectively. There is a critical ratio between the size and
the contrast of the inhomogeneities under which the generated fields can be
drastically enhanced. This enhancement has tremendous applications in imaging,
in the broad sense, and material sciences, to cite a few. It has been observed
and quantified in the stationary regimes for certain values of the incident
frequencies (see, e.g., \cite{Dev}) which are referred to as resonances. The
purpose of our work is understanding the origin of such particular
frequencies, enlightening the mechanism which leads to their emergence.

To study this question, we consider the stationary acoustic wave propagation
in the presence of micro-bubbles. We deal with a linear model described by the
mass density and the bulk modulus see \cite{CaMiPaTi 85}-\cite{CaMiPaTi 86}.
When the background medium is homogeneous, with constant mass density
$\rho_{0}$ and bulk modulus $k_{0}$, and the micro-bubble has shape
$\Omega^{\varepsilon}$, with diameter $\varepsilon$ of about few tens of
micrometers, mass density $\rho_{\varepsilon}$ and bulk modulus
$k_{\varepsilon}$, then the resonant frequencies are expected to appear in the
following asymptotic regimes:

\begin{itemize}
\item \emph{Low density / low bulk bubble}, characterized by the
small-$\varepsilon$ behaviour: $\rho_{\varepsilon}/\rho_{0}\sim k_{\varepsilon
}/k_{0}\sim\varepsilon^{-r}$ with $r>0$. In this regime the relative speed of
propagation: $c_{\varepsilon}^{2}/c_{0}^{2}:=\frac{\rho_{\varepsilon}%
}{k_{\varepsilon}}\,\frac{\rho_{0}}{k_{0}}\sim1$ is moderate, but the contrast
of the transmission coefficient is large as $\varepsilon\ll1$.

\item \emph{Moderate density / low bulk bubble,} defined by $\rho
_{\varepsilon}/\rho_{0}\sim1$ and $k_{\varepsilon}/k_{0}\sim\varepsilon^{-r}$,
$r>0$, as $\varepsilon\ll1$. These properties mean that the relative speed of
propagation is small. But the contrast of the transmission coefficient is
moderate. Such bubbles are not known to exist in nature but they might be
designed, see \cite{Z-F-2018}.
\end{itemize}

These configurations give rise to different types of resonant frequencies. We
classify them as follows:

\begin{itemize}
\item The\emph{\ Minnaert resonance}, which corresponds to a surface-mode for
the low density / low bulk bubbles.

\item The \emph{body resonances}, which correspond to body-modes for the
moderate density / low bulk bubbles.
\end{itemize}

The size of such resonances depends on the value of $r$. In particular, they
are very large when $r<2$ and very small when $r>2$, in terms of the relative
diameter $\varepsilon$, $\varepsilon\ll1$. However, when $r=2$, they are
moderate and their dominant parts are independent of $\varepsilon$. In what
follows we present the general setting of our problem and provide a
qualitative argument showing how these resonances indeed appear.

\subsection{Resonant frequencies generated by a micro-bubble}

Let $\Omega\subset\mathbb{R}^{3}$ be an open bounded and connected domain with
a smooth boundary\footnote{In most of the computations, the Lipschitz
regularity is enough, and all the results here presented hold with a boundary
of class $\mathcal{C}^{1,1}$. However, to avoid too many technicalities, we
prefer to work with a smooth boundary.
} $\Gamma:=\partial\Omega$. We define the contracted domain
\begin{equation}
\Omega^{\varepsilon}:=\left\{  x:x=y_{0}+\varepsilon\,\left(  y-y_{0}\right)
\,,\ y\in\Omega\right\}  \label{Intro_Omega_eps}%
\end{equation}
and denote with $\Gamma^{\varepsilon}:=\partial\Omega^{\varepsilon}$ its
boundary. The acoustic medium is defined by the density $\rho$ and the bulk
$k$ both discontinuous across $\Gamma^{\varepsilon}$
\begin{equation}
\rho:=%
\begin{cases}
\rho_{\varepsilon}\quad\text{inside }\Omega^{\varepsilon}\,,\\
\rho_{0}\quad\text{outside }\Omega^{\varepsilon}\,,
\end{cases}
\quad\text{and}\qquad k:=%
\begin{cases}
k_{\varepsilon}\quad\text{inside }\Omega^{\varepsilon}\,,\\
k_{0}\quad\text{outside }\Omega^{\varepsilon}\,.
\end{cases}
\label{acoustic_coeff}%
\end{equation}
Let $u^{\mathrm{in}}(x,\omega,\theta):=e^{i\omega\sqrt{\rho_{0}/k_{0}}%
\;x\cdot\theta}$ be the incident plane wave, propagating in the direction
$\theta$, and $\nu$ be the exterior unit normal to $\Gamma^{\varepsilon}$. The
scattering of $u^{\mathrm{in}}$ by the medium perturbation introduced in
(\ref{acoustic_coeff}) is described by the boundary value problem (see
\cite{AmmChaChouSi}, \cite{Amm 1})%
\begin{equation}%
\begin{cases}
\left(  \nabla\cdot\frac{1}{\rho}\nabla+\omega^{2}\frac{1}{k}\right)
u=0\,,\quad\text{in }\mathbb{R}^{3}\backslash\Gamma^{\varepsilon}\,,\\
\\
\left.  u\right\vert _{\mathrm{in}}=\left.  u\right\vert _{\mathrm{ex}%
}\,,\quad\left.  \frac{1}{\rho_{\varepsilon}}\nu\cdot\nabla u\right\vert
_{\mathrm{in}}=\left.  \frac{1}{\rho_{0}}\nu\cdot\nabla u\right\vert
_{\mathrm{ex}}\,,\quad\text{at }\Gamma^{\varepsilon}\,,\\
\\
u^{\mathrm{sc}}:=u-u^{\mathrm{in}}\,,\quad\frac{\partial u^{\mathrm{sc}}%
}{\partial|x|}-i\omega\sqrt{\frac{\rho_{0}}{k_{0}}}\;u^{\mathrm{sc}}=o\left(
|x|^{-1}\right)  \,,\quad\text{as }|x|\rightarrow+\infty\,,
\end{cases}
\label{acoustic_scattering_eq_0}%
\end{equation}
where $\left.  f\right\vert _{\mathrm{in}/\mathrm{ex}}$ denote the lateral
traces on $\Gamma^{\varepsilon}$. It is worth noticing that these interface
conditions provide with the regularity of the total field $u=u^{\mathrm{sc}%
}+u^{\mathrm{in}}$ across the boundary: we further refer to \cite{CaMiPaTi
85}-\cite{CaMiPaTi 86} for the physical setting justifying
(\ref{acoustic_scattering_eq_0}). Denoting with $\mathcal{K}_{\omega}^{0}$ the
Green's function of the background medium $(\rho_{0},k_{0})$ satisfying the
outgoing Sommerfeld radiation conditions and with%
\begin{equation}
\alpha:=\frac{1}{\rho_{\varepsilon}}-\frac{1}{\rho_{0}}\,,\qquad\beta
:=\frac{1}{k_{\varepsilon}}-\frac{1}{k_{0}}\,, \label{the contrasts}%
\end{equation}
the contrasts between the inner and the outer acoustic coefficients, by the
Lippmann-Schwinger representation of the total acoustic field $u$ we have
\begin{equation}
u(x)-\alpha\,\nabla_{x}\cdot\int_{\Omega^{\varepsilon}}\mathcal{G}_{\omega
}^{0}(x-y)\nabla u(y)dy-\beta\omega^{2}\int_{\Omega^{\varepsilon}}%
\mathcal{G}_{\omega}^{0}(x-y)u(y)dy=u^{\mathrm{in}}(x)\,.
\end{equation}
An integration by parts allows to transform this integro-differential equation
into a solely integral equation (see detailed computations in \cite[Sec.
3]{DaGhanSi})
\begin{equation}
u(x)-\left(  \beta-\alpha\rho_{\varepsilon}/k_{\varepsilon}\right)  \omega
^{2}\int_{\Omega^{\varepsilon}}\mathcal{G}_{\omega}^{0}(x-y)u(y)dy+\alpha
\int_{\Gamma^{\varepsilon}}\mathcal{G}_{\omega}^{0}(x-y)\frac{\partial
u}{\partial\nu}(y)dy=u^{\mathrm{in}}(x)\,. \label{int_eq}%
\end{equation}
We next rephrase this problem using the Newtonian (volume-type) operator
\begin{equation}
N_{\omega}(\varepsilon):L^{2}({\Omega^{\varepsilon}})\rightarrow L^{2}%
({\Omega^{\varepsilon}}),\qquad(N_{\omega}(\varepsilon)u)(x):=\int%
_{\Omega^{\varepsilon}}\mathcal{G}_{\omega}^{0}(x-y)u(y)dy\,,
\end{equation}
with image in $H^{2}({\Omega^{\varepsilon}})$, and the surface-type
operator\footnote{Here $p.v$ refers to the Cauchy principal value.}%
\begin{equation}
K_{\omega}^{\ast}(\varepsilon):H^{-1/2}(\Gamma^{\varepsilon})\rightarrow
H^{-1/2}(\Gamma^{\varepsilon}),\qquad(K_{\omega}^{\ast}(\varepsilon
)\varphi)(x):=p.v.\int_{\Gamma^{\varepsilon}}\frac{\partial}{\partial\nu_{x}%
}\mathcal{G}_{\omega}^{0}(x-y)\varphi(y)dy\,.
\end{equation}
The notation adopted is justified by the fact that $K_{\omega}^{\ast
}(\varepsilon)$ identifies with the $L^{2}(\Gamma^{\varepsilon})$-adjoint of
the well known Neumann-Poincar\'{e} operator (see Subsection \ref{Sec_K}).
Taking the normal derivative (here simply denoted with $\partial_{\nu}$) and
trace on $\Gamma^{\varepsilon}$, from (\ref{int_eq}) we obtain the surface
integral equation%
\begin{equation}
\left(  1+\frac{\alpha}{2}\right)  \partial_{\nu}u-\left(  \beta-\alpha
\rho_{\varepsilon}/k_{\varepsilon}\right)  \omega^{2}\partial_{\nu}N_{\omega
}(\varepsilon)u+\alpha\,K_{\omega}^{\ast}(\varepsilon)\left(  \partial_{\nu
}u\right)  =\partial_{\nu}u^{\mathrm{in}}\,,\quad\text{at }\Gamma
^{\varepsilon}\,.
\end{equation}
Hence, the total acoustic field in the exterior of the bubble $\mathbb{R}%
^{3}\backslash\overline{{\Omega^{\varepsilon}}}$ is fully computable from the
values $u|\Omega^{\varepsilon}$ and $\partial_{\nu}u|_{\mathrm{in}}$ which are
solutions of the following closed-form system of integral equations%
\begin{equation}
\left[  I-\left(  \beta-\alpha\rho_{\varepsilon}/k_{\varepsilon}\right)
\omega^{2}N_{\omega}(\varepsilon)\right]  u+\alpha\int_{\Gamma^{\varepsilon}%
}\mathcal{G}_{\omega}^{0}(x-y)\frac{\partial u}{\partial\nu}(y)\,d\sigma
(y)=u^{\mathrm{in}}(x)\,,\qquad\text{in }{\Omega^{\varepsilon}\,,}
\label{LS_eq_1}%
\end{equation}%
\begin{equation}
\left[  \frac{1}{\alpha}+\frac{1}{2}+K_{\omega}^{\ast}\left(  \varepsilon
\right)  \right]  \partial_{\nu}u-\frac{\left(  \beta-\alpha\rho_{\varepsilon
}/k_{\varepsilon}\right)  }{\alpha}\,\omega^{2}\partial_{\nu}N_{\omega
}(\varepsilon)u=\frac{1}{\alpha}\,\partial_{\nu}u^{\mathrm{in}}\,,\qquad
\text{at }\Gamma^{\varepsilon}\,. \label{LS_eq_2}%
\end{equation}
The Newtonian and the Neumann-Poincar\'{e} operators appearing above can be
similarly defined on the dilated domain $\Omega$; let us simply denote them as
$N_{\omega}$ and $K_{\omega}^{\ast}$ in this case. When $\omega=0$, each of
the operators $N_{0}$ and $K_{0}^{\ast}$ generates discrete sequences of
eigenvalues%
\begin{equation}
\sigma_{p}\left(  N_{0}\right)  =\left\{  \lambda_{m}\right\}  _{m\in
\mathbb{N}}\subset\mathbb{R}\,,\quad\text{and}\quad\sigma_{p}\left(
K_{0}^{\ast}\right)  \subset\left[  -\frac{1}{2},\frac{1}{2}\right)  \,.
\label{eigenvalues}%
\end{equation}
These are the key properties in estimating the resonances. Since $N_{\omega
}(\varepsilon)$ and $K_{\omega}^{\ast}(\varepsilon)$ scale as: $N_{\omega
}(\varepsilon)\sim\varepsilon^{2}N_{0}$ and $K_{\omega}^{\ast}(\varepsilon
)\sim-1/2+\varepsilon^{2}\left(  \partial_{\omega}K_{{\omega}}^{\ast}\right)
\left(  0\right)  $ as $\varepsilon\rightarrow0_{+}$, these operators may
excite the eigenvalues of $N_{0}$ or $K_{0}^{\ast}$ and create a singularity
in (\ref{LS_eq_1})-(\ref{LS_eq_2}) depending on the scales defining our micro-bubbles.

\begin{itemize}
\item For low density / low bulk bubbles, we have $\left(  \beta-\alpha
\rho_{\varepsilon}/k_{\varepsilon}\right)  \sim1$ and then $\left(
\beta-\alpha\rho_{\varepsilon}/k_{\varepsilon}\right)  \omega^{2}N_{\omega
}(\varepsilon)$ $\ll1$ as $\varepsilon\ll1$. Hence, there is no singularity
coming from (\ref{LS_eq_1}). Nevertheless, if $\alpha\sim\varepsilon^{-2}$ as
$\varepsilon\ll1$ (which corresponds to the assumption of a low density and
bulk regime with $r=2$) then a suitable choice of $\omega$ allows to excite
the eigenvalue $-1/2$ of $K_{0}^{\ast}$ and create a singularity in
(\ref{LS_eq_2}). In this case, we have the Minnaert resonance with
surface-modes. This resonance was first observed in \cite{Amm 1} based on
indirect integral equation methods. This result was extended to more general
families of micro-bubbles in \cite{AmmChaChouSi} and
\cite{ACCS-effective-media} (see also the more recent \cite{AmmFe 1} and
\cite{AmmFe 2}).

\item For moderate density / low bulk bubbles, we have $\alpha\sim1$ and then
we keep away from the full spectrum of $K_{0}^{\ast}$. Hence there is no
singularity coming from (\ref{LS_eq_2}). But as $\left(  \beta-\alpha
\rho_{\varepsilon}/k_{\varepsilon}\right)  \sim\varepsilon^{-2}\gg1$, suitable
choices of $\omega$ allow to excite the eigenvalues of the Newtonian operators
$N_{0}$ and create singularities in (\ref{LS_eq_1}). This gives us a sequence
of resonances with volumetric-modes which were observed in \cite{A.D-F-M-S}
and \cite{M-M-S}.

\item Observe that if $\alpha$ is negative (i.e. \emph{negative mass
densities}, similar to the Drude model for electromagnetism for instance) then
we could excite the other sequence of eigenvalues of $K_{0}^{\ast}$. This
gives us another sequence of resonances (i.e. corresponding to the sequence of
plasmonics in electromagnetics).
\end{itemize}

When the incident frequency $\omega$ is close to the ones generating the
singularities, the total field inside the bubble, solution of the system
(\ref{LS_eq_1})-(\ref{LS_eq_2}) becomes large. This implies an enhancement of
the scattered and far-fields, and motivates the definition of \emph{resonant
frequency}, which became widely used in a somehow generic sense. The limit
behaviour of the scale-dependent scattering problem (\ref{int_eq}) has been
described in \cite{Amm 1},\cite{DaGhanSi}, where the asymptotic analysis is
developed using layer potential techniques and the Gohberg-Sigal theory. Here
we develop a different approach based on the resolvent analysis of a
frequency-dependent Schr\"{o}dinger operator. The advances provided by this
approach are presented in Section \ref{Sec_Results}.\hspace{0.05cm}

\subsection{\label{Sec_Model}The equivalent frequency-dependent
Schr\"{o}dinger operator}

The asymptotic framework is next realized by contrasting an homogeneous
acoustic background with a small homogeneous inclusion, supported on
$\Omega^{\varepsilon}$, whose acoustic density and bulk are both defined by
the piecewise constant field $1_{\mathbb{R}^{3}\backslash\Omega_{\varepsilon}%
}+\varepsilon^{-2}1_{\Omega^{\varepsilon}}$. Following the notation introduced
in (\ref{Intro_Omega_eps})-(\ref{acoustic_coeff}), we assume
\begin{equation}
\rho=k:=%
\begin{cases}
1/\varepsilon^{2}\quad\text{inside }\Omega^{\varepsilon}\,,\\
1\quad\text{outside }\Omega^{\varepsilon}\,.
\end{cases}
\label{acoustic_coeff_1}%
\end{equation}
Since both contrasts are of size $\varepsilon^{-2}$, this regime defines a low
density / low bulk bubble and -- according to our previous discussion --
generates an asymptotically bounded Minnaert resonance with dominant part
independent of $\varepsilon$. Furthermore, in this particular scaling we have
$\left(  \beta-\alpha\rho_{\varepsilon}/k_{\varepsilon}\right)  =0$ (compare
with (\ref{acoustic_coeff})) which cancels the body-potential contribution in
(\ref{int_eq}). In what follows we incorporate the assumption
(\ref{acoustic_coeff_1}) in our scattering problem and provide with a
frequency-dependent auxiliary operator allowing to rephrase
(\ref{acoustic_scattering_eq_0}) in terms of a generalized eigenfunction
problem. This approach requires a large use of layer mappings, potentials and
integral operators which naturally appear in the modeling of scattering from
interfaces and obstacles. The precise definitions, the related mapping
properties and the common notation are recalled in the Appendix. When these
operators refer to the contracted boundary $\Gamma^{\varepsilon}$ we adopt
appropriate notation which are next introduced. Let $\gamma_{0}^{\mathrm{in}%
/\mathrm{ex}}(\varepsilon)$ and $\gamma_{1}^{\mathrm{in}/\mathrm{ex}%
}(\varepsilon)$ denote the lateral traces and normal-traces operators on
$\Gamma^{\varepsilon}$; the corresponding mean-traces and jumps are
$\gamma_{0}(\varepsilon)$, $\gamma_{1}(\varepsilon)$ and $[\gamma
_{0}(\varepsilon)]$, $[\gamma_{1}(\varepsilon)]$ respectively. Omitting the
radiation condition, which is next tacitly assumed, the acoustic scattering
equation (\ref{acoustic_scattering_eq_0}) writes as%

\begin{equation}%
\begin{cases}
\left(  \nabla\left(  1_{\mathbb{R}^{3}\backslash\Omega^{\varepsilon}%
}+\varepsilon^{-2}1_{\Omega^{\varepsilon}}\right)  \nabla+\omega^{2}\left(
1_{\mathbb{R}^{3}\backslash\Omega^{\varepsilon}}+\varepsilon^{-2}%
1_{\Omega^{\varepsilon}}\right)  \right)  u=0\,, \qquad\text{in }%
\mathbb{R}^{3}\backslash\Gamma^{\varepsilon}\,,\\
\\
\left[  \gamma_{0}(\varepsilon) \right]  u=0\,, \qquad\left(  \gamma
^{\mathrm{ex}}_{1}(\varepsilon) -\varepsilon^{-2}\gamma^{\mathrm{in}}%
_{1}(\varepsilon) \right)  u=0\,, \qquad\text{on }\Gamma^{\varepsilon}\,,
\end{cases}
\label{acoustic_eq_1}%
\end{equation}
Since the co-normal jump condition implies%
\begin{equation}
[\gamma_{1}(\varepsilon)] u=(\varepsilon^{-2}-1)\gamma^{\mathrm{in}}%
_{1}(\varepsilon) u\,,
\end{equation}
this problem rephrases as%
\begin{equation}%
\begin{cases}
\left(  \Delta+\omega^{2}\right)  u=0\,, \qquad\text{in }\mathbb{R}%
^{3}\backslash\Gamma^{\varepsilon}\,,\\
\\
\left[  \gamma_{0}(\varepsilon)\right]  u=0\,, \qquad\left[  \gamma_{1}\left(
\varepsilon\right)  \right]  u=\left(  \varepsilon^{-2}-1\right)
\gamma^{\mathrm{in}}_{1}( \varepsilon) u\,, \qquad\text{on }\Gamma
^{\varepsilon}\,.
\end{cases}
\label{acoustic_eq_1_1}%
\end{equation}

The Dirichlet-to-Neumann operator for the domain $\Omega^{\varepsilon}$, next
denoted with $D\!N_{z}(\varepsilon)$, is defined by%
\begin{equation}
D\!N_{z}(\varepsilon)\varphi:=\gamma_{1}^{\mathrm{in}}\left(  \varepsilon
\right)  u\,,\qquad%
\begin{cases}
\left(  \Delta+z^{2}\right)  u=0\,,\qquad\text{in }\Omega^{\varepsilon}\,,\\
\\
\gamma_{0}^{\mathrm{in}}(\varepsilon)u=\varphi\qquad\text{on }\Gamma
^{\varepsilon}\,.
\end{cases}
\label{DN_k_eps_def}%
\end{equation}
Such a definition is well-posed whenever $z^{2}\notin\sigma(-\Delta
_{\Omega^{\varepsilon}}^{D})\}$, where $\Delta_{\Omega^{\varepsilon}}^{D}$ is
the Dirichlet Laplacian in $L^{2}(\Omega^{\varepsilon})$. Since $\lambda
_{\Omega}:=\inf\sigma(-\Delta_{\Omega}^{{D}})>0$, by%
\begin{equation}
z^{2}\in\sigma\left(  -\Delta_{\Omega^{\varepsilon}}^{{D}}\right)  \quad
\iff\quad\varepsilon^{2}z^{2}\in\sigma\left(  -\Delta_{\Omega}^{{D}}\right)
\,,
\end{equation}
there follows that for each $z\in\mathbb{C}$ there exists $\varepsilon_{0}>0$
small enough (depending on $z$) such that $D\!N_{z}(\varepsilon)$ exists for all
$0<\varepsilon<\varepsilon_{0}$. Assuming $\varepsilon^{2}\omega^{2}%
\notin\sigma(\Delta_{\Omega}^{D})$, the solution of (\ref{acoustic_eq_1_1})
solves the homogeneous problem in (\ref{DN_k_eps_def}) with $z^{2}=\omega^{2}$
and with boundary datum $\gamma_{0}^{\mathrm{in}}(\varepsilon)u$. By
definition, we have%
\begin{equation}
\gamma_{1}^{\mathrm{in}}\left(  \varepsilon\right)  u=D\!N_{\omega}\left(
\varepsilon\right)  \gamma_{0}^{\mathrm{in}}\left(  \varepsilon\right)  u
\end{equation}
and (\ref{acoustic_eq_1_1}) recasts to
\begin{equation}%
\begin{cases}
\left(  \Delta+\omega^{2}\right)  u=0\,,\qquad\text{in }\mathbb{R}%
^{3}\backslash\Gamma^{\varepsilon}\,,\\
\\
\left[  \gamma_{0}\left(  \varepsilon\right)  \right]  u=0\,,\qquad\left[
\gamma_{1}\left(  \varepsilon\right)  \right]  u=\left(  \varepsilon
^{-2}-1\right)  D\!N_{\omega}\left(  \varepsilon\right)  \gamma_{0}%
(\varepsilon)u\,,\qquad\text{on }\Gamma^{\varepsilon}\,.
\end{cases}
\label{acoustic_eq_2}%
\end{equation}

Let us recall from \cite{MaPoSi SE} and \cite{MaPo SM} that, given 
$s\in\left(  0,1/2\right)  $ and $\Theta\in\mathsf{B}\left(  H^{s}\left(
\Gamma^{\varepsilon}\right)  ,H^{-s}\left(  \Gamma^{\varepsilon}\right)
\right)  $ self-adjoint (in the sense of the duality) and defining
$\Theta\,\delta_{\Gamma^{\varepsilon}}\in\mathcal{D}^{\prime}\left(
\mathbb{R}^{3}\right)  $ as%
\begin{equation}
\Theta\,\delta_{\Gamma^{\varepsilon}}u:=\int_{\Gamma^{\varepsilon}}%
d\sigma\,\Theta\left(  \gamma\left(  \varepsilon\right)  u\right)  \,,
\end{equation}
a self-adjoint (in $L^{2}(\mathbb{R}^{3})$) realization of $-\Delta +\Theta\delta_{\Gamma^{\varepsilon}}$ is
provided by the restriction of $\left(  \left.  \Delta\right\vert
\,\ker\left(  \gamma_{0}\left(  \varepsilon\right)  \right)  \right)  ^{\ast}$
to functions fulfilling interface conditions of the kind
\begin{equation}
\left[  \gamma_{0}\left(  \varepsilon\right)  \right]  u=0\,,\qquad\left[
\gamma_{1}\left(  \varepsilon\right)  \right]  =\Theta\gamma_{0}\left(
\varepsilon\right)  u\,,
\end{equation}
(see also \cite{BraEx} and \cite{BeLaLo} for previous references to this
topic). This suggests a formal analogy between our problem and the generalized
eigenfunction equation for singular perturbations of the Laplacian with
$\delta$-type transmission conditions. As a further support to this remark, we
also notice that the integral form of (\ref{acoustic_eq_2}) simply reads as
(compare with (\ref{int_eq}))%
\begin{equation}
u=u^{\mathrm{in}}-\left(  \varepsilon^{-2}-1\right)  SL_{\omega}\left(
\varepsilon\right)  \gamma_{1}^{\mathrm{in}}\left(  \varepsilon\right)  u\,,
\label{acoustic_eq_1_int}%
\end{equation}
where $u^{\mathrm{in}}$ is an incoming wave (in this case a solution of
$\left(  \Delta+\omega^{2}\right)  u^{\mathrm{in}}=0$ in $\mathbb{R}^{3}$) and
$S\!L_{\omega}(\varepsilon)$ is the single-layer operator related to
$\Gamma^{\varepsilon}$. Hence, the scattered field is represented in terms of
a single-layer potential, which, as it has been shown in \cite{MaPoSi LAP},
corresponds to the solution form of the scattering problem for $\delta$-type
singular perturbations of the free Laplacian.

A specific feature of classical scattering problems consists in the fact that
the total field identifies with a generalized eigenfunction of an auxiliary
Schr\"odinger-type operator which usually depends on the frequency. This is
quite evident when one considers the simpler stationary problem for classical
waves propagating in a medium with a local perturbation of the bulk. Assume
for instance to have a piecewise constant bulk described by $b_{0}1_{\Omega
}+1_{\mathbb{R}^{3}\backslash\Omega}$; then, a stationary wave with frequency
$\omega>0$ solves the equation%
\[
\left(  \Delta-\omega^{2}\left(  1-b_{0}\right)  1_{\Omega}+\omega^{2}\right)
u=0\,,
\]
corresponding to the generalized eigenfunction problem at energy $\omega^{2}$
for the Schr\"{o}dinger operator $-\Delta+\omega^{2}\left(  1-b_{0}\right)
1_{\Omega}$. In the attempt of adapting this construction to the more complex
framework considered in (\ref{acoustic_eq_2}), which involves a discontinuity
on $\Gamma^{\varepsilon}$ both for the acoustic bulk and density, we push
further the analogy and consider $H_{\omega}(\varepsilon)$ of the form
\begin{equation}
\Delta+\left(  \varepsilon^{-2}-1\right)  D\!N_{\omega}\left(  \varepsilon
\right)  \delta_{\Gamma^{\varepsilon}}\,, \label{H_eps_omega}%
\end{equation}
as a candidate for the frequency-dependent operator to identify the solutions
of (\ref{acoustic_eq_2}) in terms of generalized eigenfunctions of $H_{\omega
}(\varepsilon)$ at energy $\omega^{2}$.

\subsection{\label{Sec_Results}The main results}

The precise definition of $H_{\omega}(\varepsilon)$ is%
\begin{equation}%
\begin{array}
[c]{ccc}%
H_{\omega}\left(  \varepsilon\right)  u:=\Delta u\,, &  & \text{in }%
\mathbb{R}^{3}\backslash\Gamma^{\varepsilon}\,,
\end{array}
\label{Intro_tilda_eps_omega_def}%
\end{equation}
for any $u$ in the domain
\begin{equation}
\mathrm{dom}\left(  H_{\omega}\left(  \varepsilon\right)  \right)  :=\left\{
u\in H_{\Delta}^{0}(\mathbb{R}^{3}\backslash\Gamma^{\varepsilon})\cap
H^{1}(\mathbb{R}^{3}):[\gamma_{1}\left(  \varepsilon\right)  ]u=\left(
\varepsilon^{-2}-1\right)  D\!N_{\omega}\left(  \varepsilon\right)  \gamma
_{0}\left(  \varepsilon\right)  u\right\}
\,,\label{Intro_tilda_eps_omega_dom_def}%
\end{equation}
where $H_{\Delta}^{0}(\mathbb{R}^{3}\backslash\Gamma^{\varepsilon})$ is the
set of the functions $u\in L^{2}(\mathbb{R}^{3})$ such that%
\begin{equation}%
\begin{array}
[c]{ccc}%
\Delta_{\Omega^{\varepsilon}}u\in L^{2}(\Omega^{\varepsilon})\,, &  &
\Delta_{\mathbb{R}^{3}\backslash\overline{\Omega^{\varepsilon}}}\,u\in
L^{2}(\mathbb{R}^{3}\backslash\overline{\Omega^{\varepsilon}})\,.
\end{array}
\label{Laplacian_weak}%
\end{equation}
Notice that the jump condition $[\gamma_{0}\left(  \varepsilon\right)  ]u=0$
is incorporated into $\mathrm{dom}(H_{\omega}(\varepsilon))\subseteq
H^{1}(\mathbb{R}^{3})$. The properties of $H_{\omega}(\varepsilon)$ are
investigated in Section \ref{Sec_H_dil}. We next resume the main features of
this model. Let $S_{0}$ denote the single layer operator of the Laplacian in
the whole space (see the definition in Subsection \ref{Sec_SL}). The
\emph{capacitance} of $\Omega$ is defined by%
\begin{equation}
c_{\Omega}:=\int_{\Gamma}(S_{0}^{-1}1)(x)\,d\sigma(x)\,,\label{C_Omega}%
\end{equation}
and the related \emph{Minnaert frequency} is%
\begin{equation}
\omega_{M}:=\sqrt{\frac{c_{\Omega}}{|\Omega|}}\,,\label{omega_M_def}%
\end{equation}
where $|\Omega|$ denotes the volume of $\Omega$. It is worth recalling that
the positiveness of $S_{0}$ implies $c_{\Omega}>0$. According to Theorem
\ref{heo} and definitions in Subsection \ref{mod-op}, for each $\omega>0$ and
$\varepsilon>0$ sufficiently small, \eqref{Intro_tilda_eps_omega_dom_def} and
\eqref{Intro_tilda_eps_omega_def} define a self-adjoint operator in
$L^{2}\left(  \mathbb{R}^{3}\right)  $. The corresponding resolvent equation%
\begin{equation}
(H_{\omega}\left(  \varepsilon\right)  +z^{2}%
)u=f\,,\label{Intro_h_eps_omega_res_eq}%
\end{equation}
is nothing but%
\begin{equation}
\left\{
\begin{array}
[c]{l}%
\begin{array}
[c]{ccc}%
\left(  \Delta+z^{2}\right)  u=f\,, &  & \text{in }\mathbb{R}^{3}%
\backslash\Gamma^{\varepsilon}\,,
\end{array}
\\
\\%
\begin{array}
[c]{ccc}%
\left[  \gamma_{0}\left(  \varepsilon\right)  \right]  u=0\,, &  & \left[
\gamma_{1}\left(  \varepsilon\right)  \right]  u=\left(  \varepsilon
^{-2}-1\right)  D\!N_{\omega}\left(  \varepsilon\right)  \gamma_{0}\left(
\varepsilon\right)  u\,.
\end{array}
\end{array}
\right.  \label{Intro_r_k_eps_omega_eq}%
\end{equation}
In this work we provide the asymptotic analysis of the problems
(\ref{acoustic_eq_2}) and (\ref{Intro_h_eps_omega_res_eq}) by using a
resolvent operator approach. Our aims are the resolvent analysis of resonances
and global asymptotic expansions of the scattered field. The main
contributions and the novelties with respect to previous works are summarized
in the following.

\begin{quotation}
1. Resolvent's asymptotics and acoustic resonances
\end{quotation}

Since $H_{\omega}\left(  \varepsilon\right)  $ depends on the physical
frequency, the dominating term of its resolvent $\left(  -H_{\omega
}(\varepsilon)-z^{2}\right)  ^{-1}$ may still depend on the values of $\omega$
as $\varepsilon\rightarrow0$. Our main result consists in showing that the
resonant frequency $\omega_{M}$ is the unique value of $\omega$ at which the
resolvent has a non-trivial limit, converging toward a point perturbation of
the Laplacian. After introducing the Sobolev spaces%
\begin{align*}
&  \left.  \dot{H}_{y_{0}}^{2}(\mathbb{R}^{3}):=\left\{  u\in\mathcal{C}%
_{b}(\mathbb{R}^{3}):|\nabla u|\in L^{2}(\mathbb{R}^{3})\,,\ \Delta u\in
L^{2}(\mathbb{R}^{3})\,,\ u(y_{0})=0\right\}  \,,\right.  \\
& \\
&  \left.  H_{y_{0}}^{2}(\mathbb{R}^{3}):=\dot{H}_{y_{0}}^{2}(\mathbb{R}%
^{3})\cap L^{2}(\mathbb{R}^{3})=\{u\in H^{2}(\mathbb{R}^{3}):u(y_{0}%
)=0\}\,,\right.
\end{align*}
such a limit operator is defined as%
\begin{align}
&  \left.  \Delta_{y_{0}}:\mathrm{dom}(\Delta_{y_{0}})\subset L^{2}%
(\mathbb{R}^{3})\rightarrow L^{2}(\mathbb{R}^{3})\,,\qquad\Delta_{y_{0}%
}u:=\Delta u_{0}\equiv\Delta u+q\,\delta_{y_{0}}\,,\right.  \label{wd1}\\
& \nonumber\\
&  \left.  \mathrm{dom}(\Delta_{y_{0}}):=\left\{  u\in L^{2}(\mathbb{R}%
^{3}):u(x)=u_{0}(x)+q\,\mathcal{G}{_{0}(x-y_{0})}\,,\ u_{0}\in\dot{H}_{y_{0}%
}^{2}(\mathbb{R}^{3})\,,\ q\in\mathbb{C}\right\}  \,.\right.  \label{wd2}%
\end{align}
In this framework, the \emph{free Laplacian} is the Laplacian operator defined
on $\mathrm{dom}(\Delta)=H^{2}(\mathbb{R}^{3})$ and its resolvent has the
Green kernel%
\[
\mathcal{G}_{z}(x-y)=\frac{e^{iz|x-y|}}{4\pi\,|x-y|}\,.
\]
With reference to the notation and results in \cite[Chapter I.1]{Albeverio},
$\Delta_{y_{0}}$ corresponds to the operator denoted there by $\Delta
_{\alpha,y}$, with $\alpha=0$ and $y=y_{0}$ (see \cite[Theorem 1.1.2]%
{Albeverio}). It belongs to the class of point perturbations of the free
Laplacian and is a self-adjoint extension of the closed symmetric restriction
$\Delta|H_{y_{0}}^{2}(\mathbb{R}^{3})$. By the well-known Kre\u{\i}n resolvent
formula, one has%
\begin{equation}
\left(  -\Delta_{y_{0}}-z^{2}\right)  ^{-1}u=\left(  -\Delta-z^{2}\right)
^{-1}u+{4\pi}\,\frac{i}{z}\ \mathcal{G}_{z}(\cdot-y_{0})\left\langle
\,\mathcal{G}_{\bar{z}}(\cdot-y_{0}),u\right\rangle _{L^{2}(\mathbb{R}^{3}%
)}\,.\label{Krein_point_id}%
\end{equation}
Equivalently, its integral kernel is given by
\[
\left(  -\Delta_{y_{0}}-z^{2}\right)  ^{-1}(x,y)=\mathcal{G}_{z}(x-y)+{4\pi
}\,\frac{i}{z}\ \mathcal{G}_{z}(x-y_{0})\,\mathcal{G}_{z}(y-y_{0})\,.
\]
The next result shows that $H_{\omega}(\varepsilon)$ converges  in norm resolvent sense as $\varepsilon\rightarrow0_{+}$ to $\Delta_{y_{0}}$
if and only if $\omega=\omega_{M}$, otherwise the limit is trivial (i.e., it equals the free Laplacian).
\begin{theorem}
\label{Th 1}For any $z\in{\mathbb{C}}_{+}\backslash i{\mathbb{R}}_{+}$\ and
for any $\varepsilon>0$ sufficiently small, one has
\begin{align}
&  \left.  \omega\not =\omega_{M}\quad\Longrightarrow\quad\left\Vert \left(
-H_{\omega}(\varepsilon)-z^{2}\right)  ^{-1}-\left(  -\Delta-z^{2}\right)
^{-1}\right\Vert _{L^{2}(\mathbb{R}^{3}),L^{2}(\mathbb{R}^{3})}\leq
c\,\varepsilon\,,\right. \\
& \nonumber\\
&  \left.  \omega=\omega_{M}\quad\Longrightarrow\quad\left\Vert \left(
-H_{\omega}(\varepsilon)-z^{2}\right)  ^{-1}-\left(  -\Delta_{y_{0}}%
-z^{2}\right)  ^{-1}\right\Vert _{L^{2}(\mathbb{R}^{3}),L^{2}(\mathbb{R}^{3}%
)}\leq c\,\varepsilon^{1/2}\,.\right.
\end{align}

\end{theorem}

Perturbations with small support and high contrast have been considered in
connection with the low energy behaviour of Schr\"{o}dinger operators. In the
 case of a perturbation by a regular potential, the asymptotic problem was discussed in \cite{AlGeKr}, while
the case of a $\delta$-like perturbation supported on a small sphere was
considered in \cite{Shi}. In both cases, the
resolvent-convergence toward a point interaction model was proved under the
condition that the corresponding unperturbed dilated Hamiltonian has a
zero-energy resonance. While our model does not directly fit in the framework
considered in these works, it exhibits a similar physical behaviour. Following
the conclusions of \cite{AlGeKr} and \cite{Shi}, this suggests that a
zero-energy resonance should appear at $(\varepsilon,\omega)=(0,\omega_{M})$
for the dilated operator $H_{\varepsilon,\omega}$. As the resolvent analysis
shows indeed (see Remark \ref{Remark_spectrum}), the map $z\rightarrow\left(
-H_{\omega}(\varepsilon)-z^{2}\right)  ^{-1}$ is analytic in the neighbourhood
of any $\omega>0$: this corresponds to the absence of eigenvalues or
resonances of $H_{\omega}(\varepsilon)$ around any point $z^{2}=\omega
^{2\text{ }}$of the positive continuous spectrum. There is, however, a strong
indication that in the resonant regime $H_{\omega_{M}}(\varepsilon)$ may have
an eigenvalue/resonance localized in a neighbourhood of size $\sim\varepsilon$
of the origin. Despite its theoretical relevance, the analysis of this point
is outside the main scope of our work and it is postponed to a further development.

While acoustic resonances have been previously discussed by a direct approach
to the scattering problem (e.g. in \cite{Amm 1} and \cite{AmmFe 1}), their
characterization in terms of the resolvent asymptotics was not enlightened so
far. Our result clarifies the role of the physical resolvent for the
stationary acoustic equation in the emergence of such resonances in the
asymptotic regime.

\begin{quotation}
2. A global-in-space asymptotic expansion of the scattered field
\end{quotation}

According to the resolvent equation (\ref{Intro_h_eps_omega_res_eq}), the
acoustic scattering problem -- in the equivalent form (\ref{acoustic_eq_2}) --
identifies with the generalized eigenfunctions problem for $H_{\omega
}(\varepsilon)$. This important feature is discussed in details in Section
\ref{Sec-Scatt}. The interest in establishing such relation is not merely
formal. Indeed, it allows to represent the scattered field in terms of the
limit resolvent on the continuos spectrum. Hence, the global-in-space
asymptotic behavior of the solutions of (\ref{acoustic_eq_2}) is determined by
similar computation to the ones leading to the norm-resolvent asymptotic
expansions of $H_{\omega}(\varepsilon)$. In particular, we consider two
regimes in terms of $\omega$ and $\varepsilon$. In the first one, discussed in
the Section \ref{Sec_Asymptotic}, we fix $\omega$ and provide the expansion in
term of $\varepsilon$ only.

\begin{theorem}
\label{Th 2} Let $\omega>0$, $\alpha>1/2$, and let $L_{-\alpha}^{2}%
(\mathbb{R}^{3})$, $H_{-\alpha}^{2}(\mathbb{R}^{3})$ be the weighted spaces in
(\ref{W_Sob_0})-(\ref{W_Sob}). For $u_{\omega}^{\mathrm{in}}\in H_{-\alpha
}^{2}(\mathbb{R}^{3})$ a solution of the homogeneous Helmholtz equation%
\begin{equation}
\left(  \Delta+\omega^{2}\right)  u_{\omega}^{\mathrm{in}}=0\,,
\end{equation}
denote with $u_{\omega}^{\mathrm{sc}}(\varepsilon)$ the unique radiating
solution of the scattering problem%
\begin{equation}%
\begin{cases}
\left(  \nabla\cdot(1_{\mathbb{R}^{3}\backslash\Omega^{\varepsilon}%
}+\varepsilon^{-2}1_{\Omega^{\varepsilon}})\nabla+\omega^{2}(1_{\mathbb{R}%
^{3}\backslash\Omega^{\varepsilon}}+\varepsilon^{-2}1_{\Omega^{\varepsilon}%
})\right)  u_{\omega}^{\mathrm{sc}}(\varepsilon)=0\,,\qquad\text{in
}\mathbb{R}^{3}\backslash\Gamma^{\varepsilon}\,,\\
\\
\left(  \gamma_{0}^{\mathrm{in}}(\varepsilon)-\gamma_{0}^{\mathrm{ex}%
}(\varepsilon)\right)  \left(  u_{\omega}^{\mathrm{in}}+u_{\omega
}^{\mathrm{sc}}(\varepsilon)\right)  =0\,,\qquad\left(  \varepsilon^{-2}%
\gamma_{1}^{\mathrm{in}}(\varepsilon)-\gamma_{1}^{\mathrm{ex}}(\varepsilon
)\right)  \left(  u_{\omega}^{\mathrm{in}}+u_{\omega}^{\mathrm{sc}%
}(\varepsilon)\right)  =0\,,\qquad\text{on }\Gamma^{\varepsilon}\,.
\end{cases}
\end{equation}
Then, for any $\varepsilon>0$ sufficiently small, one has, uniformly with
respect to the choice of the incoming wave $u_{\omega}^{\mathrm{in}}$,
\begin{align}
&  \left.  \omega\not =\omega_{M}\quad\Longrightarrow\quad\left\{
\begin{array}
[c]{l}%
\left(  u_{\omega}^{\mathrm{sc}}(\varepsilon)\right)  (x)=\varepsilon
\,\frac{\omega^{2}\,c_{\Omega}}{\omega_{M}^{2}-\omega^{2}}\ u_{\omega
}^{\mathrm{in}}(y_{0})\,\frac{e^{i\omega|x-y_{0}|}}{4\pi|x-y_{0}|}+\left(
r_{\omega}(\varepsilon)\right)  (x)\,,\\
\\
\left\Vert r_{\omega}(\varepsilon)\right\Vert _{L_{-\alpha}^{2}(\mathbb{R}%
^{3})}\leq c\,\varepsilon^{3/2}\,,
\end{array}
\right.  \right.  \label{non_res_scattering}\\
& \nonumber\\
&  \left.  \omega=\omega_{M}\quad\Longrightarrow\quad\left\{
\begin{array}
[c]{l}%
\left(  u_{\omega}^{\mathrm{sc}}(\varepsilon)\right)  (x)=\frac{4\pi i}%
{\omega}\ u_{\omega}^{\mathrm{in}}(y_{0})\,\frac{e^{i\omega|x-y_{0}|}}%
{4\pi|x-y_{0}|}+\left(  r_{\omega}\left(  \varepsilon\right)  \right)
(x)\,,\\
\\
\left\Vert r_{\omega}(\varepsilon)\right\Vert _{L_{-\alpha}^{2}(\mathbb{R}%
^{3})}\leq c\,\varepsilon^{1/2}\,.
\end{array}
\right.  \right.  \label{res_scattering}%
\end{align}

\end{theorem}

In the second one, discussed in Section \ref{Sec_Asymptotic_1}, we provide the
expansions by varying both $\omega$ and $\varepsilon$.

\begin{theorem}
\label{Th 3}Let $c_{M}>0$ and ${\mathcal{I}}_{M}\subset{\mathbb{R}}_{+}$ be a
bounded interval containing $\omega_{M}$. For any $\varepsilon>0$ sufficiently
small, the expansion%
\begin{align}
&  \left.  \left(  u_{\omega}^{\mathrm{sc}}(\varepsilon)\right)
(x)=\frac{\varepsilon\,\omega^{2}c_{\Omega}}{\omega_{M}^{2}-\omega
^{2}-i\varepsilon\frac{\omega^{3}c_{\Omega}}{4\pi}}\,\,u_{\omega
}^{\mathrm{in}}(y_{0})\,\frac{e^{i\omega|x-y_{0}|}}{4\pi|x-y_{0}|}+\left(
r_{\omega}(\varepsilon)\right)  (x)\,,\right.
\label{Intro_scattering_exp_quasires}\\
& \nonumber\\
&  \left.  \left\Vert r_{\omega}(\varepsilon)\right\Vert _{L_{-\alpha}%
^{2}(\mathbb{R}^{3})}\leq c\ \frac{\varepsilon^{3/2}}{\omega_{M}^{2}%
-\omega^{2}}\,,\qquad\alpha>1/2\,.\right.
\end{align}
holds uniformly with respect to both  $\omega$ in $\left\{  \omega
\in\mathcal{I}_{M}:\left\vert \omega-\omega_{M}\right\vert \geq c_{M}%
\,\varepsilon\right\}  $ and $u_{\omega}^{\mathrm{in}}$ in any bounded subset
of $H_{-\alpha}^{2}(\mathbb{R}^{3})$.
\end{theorem}

Coherently with what emerges from the resolvent analysis, these expansions
shows that when $\omega$ approaches $\omega_{M}$, the scattering system
undergoes a transition between an asymptotically trivial scattering and a
non-trivial one. The transition is enlighten in
(\ref{Intro_scattering_exp_quasires}) with the lower-bound condition
$\left\vert \omega-\omega_{M}\right\vert >c_{M}\,\varepsilon$.
\par
The expansions of the scattered field in Theorems 1.2 and 1.3 hold in the
whole space for any fixed $\varepsilon>0$ small enough depending on $\omega$.
It is worth to recall that similar formulas were previously obtained in
\cite{Amm 1} and \cite{AmmFe 1} for the rescaled problem by a different
method. Nevertheless, while holding uniformly w.r.t. $\omega$, these formulas
were limited to the far-field zone $\left\vert x\right\vert \rightarrow+\infty
$, i.e. without an uniform control-in-space of the errors. The
uniform-in-space asymptotic expansion of $u_{\omega}^{\mathrm{sc}}%
(\varepsilon)$, here provided by the resolvent analysis, can be relevant in
applications involving the knowledge of the near field. Moreover, since
$u_{\omega}^{\mathrm{sc}}(\varepsilon)$ fulfills the outgoing radiation
condition, our results can be easily rephrased in terms of the far-fields
pattern, yielding, possibly under the additional constraint $\left\vert
\omega-\omega_{M}\right\vert >c_{M}\,\varepsilon$, to those formulas of the
scattering enhancement near the Minnaert frequency already presented in the
above mentioned works.

\begin{remark}
As a consequence of the absence of embedded eigenvalues and singular
continuous spectrum of $H_{\omega}(\varepsilon)$ (see Remark
\ref{Remark_spectrum}), the map $\omega\mapsto u_{\omega}^{\mathrm{sc}%
}(\varepsilon)$ is continuous on any interval of the real axis, with the only
possible exception of the origin, provided that $\varepsilon$ is small enough.
Hence, we expect that the constraint $\left\vert \omega-\omega_{M}\right\vert
>c_{M}\,\varepsilon$ in \eqref{Intro_scattering_exp_quasires} could be removed
by improving the control of the error $r_{\omega}(\varepsilon)$.
\end{remark}

\subsection{Approach and perspectives}

A common approach to the analysis of perturbations with small support and high
contrast consists in introducing an equivalent dilated system. In our case,
this allows to work with integral operators defined on the fixed boundary
$\Gamma$. Following this strategy, in Theorem \ref{heo}, we build a family of
self-adjoint operators $H_{\varepsilon,\omega}$, depending both on the scale
parameter $\varepsilon$ and the frequency $\omega$, which encode the dilated
interface conditions at $\Gamma$. The physical operators $H_{\omega
}(\varepsilon)$ are then obtained in Section \ref{mod-op} by conjugation of
the dilated resolvent with the unitary dilation operators. These singular
perturbation models are defined through Kre\u{\i}n-like resolvent formulas
representing the difference%
\[
\left(  -H_{\omega}(\varepsilon)-z^{2}\right)  ^{-1}-\left(  -\Delta
-z^{2}\right)  ^{-1}%
\]
in terms of an explicit \emph{interaction operator}, which is a boundary map
modeling the interaction between the background and the inhomogeneity (see
(\ref{krf}) below). By the limiting absorption principle, the scattered fields
also represent in terms of this boundary map through an exact formula
depending on the incoming wave. Hence the asymptotic analysis mainly reduces
to the analysis of this boundary operator.

Most of the computations can be carried out in the case of multiple low
density / low bulk bubbles, each having a specific Minnaert frequency. When
multiple bubbles share the same excitation frequency, the expected limit
scattering in the resonant regime will be described by a multiple
point-scattering system.

Expansions similar to (\ref{Intro_scattering_exp_quasires}) have been recently
provided in the regime of highly contrasted small inclusions with
non-homogeneous acoustic backgrounds and different bulk/density ratios (see
\cite{DaGhanSi}). These results suggest that the mechanism leading to the
scattering enhancement, enlighten in this work, actually characterizes a much
larger class of models. It is worth noticing that, while our simplified
setting allows to implement purely singular perturbation methods in the
resolvent analysis of the problem, in more general frameworks, the auxiliary
Schr\"{o}dinger operator associated to the scattering problem would exhibit
both singular and regular (potential-like) perturbation terms whose resolvent
analysis requires some generalization with respect to the one here employed.

\subsection{Notation}

$\cdot$ $\mathbb{R}_{\pm}:=\{x\in\mathbb{R}:\pm x>0\}$; $\mathbb{C}_{\pm
}:=\{z\in\mathbb{C}:\operatorname{Im}(z)\in\mathbb{R}_{\pm}\}$;\medskip

\noindent$\cdot$ $c\in\mathbb{R}_{+}$ denotes a generic constant which may
vary from line to line;\medskip

\noindent$\cdot$ $\Vert\cdot\Vert_{X}$ denotes the norm in the Banach space
$X$;\medskip

\noindent$\cdot$ $\langle\cdot,\cdot\rangle_{H}$ denotes the (conjugate-linear
w.r.t. the first variable) inner product in the Hilbert space $H$;\medskip

\noindent$\cdot$ $\langle\cdot,\cdot\rangle_{X^{\ast}\!,X}$ denotes the
duality, assumed to be conjugate-linear w.r.t. the first variable, between the
dual pair $(X^{\ast}\!,X)$;

\noindent$\cdot$ $\mathrm{dom}(L)$, $\mathrm{ker}(L)$ and $\mathrm{ran}(L)$
denote the domain, kernel and range of the linear operator $L$;\medskip

\noindent$\cdot$ $\sigma(L)$ and $\varrho(L)$ denote the spectrum and the
resolvent set of the closed operator $L$;\medskip

\noindent$\cdot$ $L^{\ast}:\mathrm{dom}(L^{\ast})\subseteq Y^{\ast}\rightarrow
X^{\ast}$ denotes the adjoint of the densely defined linear operator
$L:\mathrm{dom}(L)\subseteq X\rightarrow Y$;\medskip

\noindent$\cdot$ $\mathcal{L}(X,Y)$ denotes the space of continuous linear
maps on $X$ to $Y$, where $X$ and $Y$ are topological vector space;
$\mathcal{L}(X)\equiv\mathcal{L}(X,X)$;\medskip

\noindent$\cdot$ $\Vert\cdot\Vert_{X,Y}$ denotes the norm on the Banach space
$\mathcal{L}(X,Y)$, $X$ and $Y$ Banach spaces;\medskip

\noindent$\cdot$ $\mathcal{L}_{H\!S}(X,Y)$ denotes the set of Hilbert-Schmidt
operators on $X$ to $Y$;\medskip

\noindent$\cdot$ Given $\varepsilon\mapsto L(\varepsilon)\in\mathcal{L}(X,Y)$,
$L(\varepsilon)=O(\varepsilon^{\lambda})$ means $\Vert L(\varepsilon
)\Vert_{X,Y}\leq c\,\varepsilon^{\lambda}$;\medskip

\noindent$\cdot$ $\mathcal{C}_{\mathrm{comp}}^{\infty}(\mathcal{O})$ and
{$\mathcal{D}$}$(\mathcal{O})$ both denote the set of smooth, compactly
supported, test functions on the open set $\mathcal{O}\subseteq\mathbb{R}^{3}%
$; {$\mathcal{D}$}$^{\prime}(\mathcal{O})$ denotes the space of Schwartz's
distributions and {$\mathcal{E}$}$^{\prime}(\mathcal{O})$ denotes the spaces
of compactly supported Schwartz's distributions;\medskip

\noindent$\cdot$ $\Delta_{\mathcal{O}}$ denotes the distributional Laplacian
in {$\mathcal{D}$}$^{\prime}(\mathcal{O})$; $\Delta_{\mathbb{R}^{3}}$ is
simply denoted by $\Delta$;\medskip

\noindent$\cdot$ $u\ast v$ denotes convolution;\medskip

\noindent$\cdot$ $\delta_{y}\in${$\mathcal{E}$}$^{\prime}(\mathbb{R}^{3})$
denotes Dirac's delta distribution supported at the point $y$;\medskip

\noindent$\cdot$ Given the bounded open set $\Omega\subset\mathbb{R}^{3}$ with
a regular boundary $\Gamma$, $|\Omega|$ and $|\Gamma|$ denote the volume of
$\Omega$ and the area of $\Gamma$ respectively; $d_{\Omega}$ denotes the
diameter of $\Omega$ and $d\sigma$ denotes the surface measure on $\Gamma
$;\medskip

\noindent$\cdot$ $\Omega_{\mathrm{in}}\equiv\Omega$ and $\Omega_{\mathrm{ex}%
}:=\mathbb{R}^{3}\backslash\overline{\Omega}$;\medskip

\noindent$\cdot$ $H^{s}(\mathbb{R}^{3})$, $H^{s}(\Omega_{\mathrm{in/ex}})$,
$s\in\mathbb{R}$, denote the usual scales of Sobolev spaces on $\mathbb{R}%
^{3}$ and $\Omega_{\mathrm{in}/\mathrm{ex}}$ respectively;\medskip

\noindent$\cdot$ $H^{s}(\Gamma)$, $s\in\mathbb{R}$, denotes the usual scales
of Sobolev spaces on $\Gamma$;\medskip

\noindent$\cdot$ $H_{\alpha}^{s}(\mathbb{R}^{3})$, $H_{\alpha}^{s}%
(\Omega_{\mathrm{ex}})$, $s\in\mathbb{R}$, $\alpha\in\mathbb{R}$, denote the
scales of weighted Sobolev spaces on $\mathbb{R}^{3}$ and $\Omega
_{\mathrm{ex}}$ with weight $\langle x\rangle^{\alpha}\equiv(1+|x|^{2}%
)^{\alpha/2}$;\medskip

\noindent$\cdot$ $\gamma_{0}$ and $\gamma_{1}$ denote the Dirichlet and
Neumann traces at $\Gamma$; $\gamma_{0}^{\mathrm{in/ex}}$ and $\gamma
_{1}^{\mathrm{in/ex}}$ denote the analogous lateral traces at the boundary of
$\Omega_{\mathrm{in}/\mathrm{ex}}$;\medskip

\noindent$\cdot$ $S\!L_{z}$ and $D\!L_{z}$ denote the single- and double-layer
operators;\medskip

\noindent$\cdot$ $S_{z}$ and $K_{z}$ denote the boundary operators $\gamma
_{0}S\!L_{z}$ and $\gamma_{0}D\!L_{z}$ respectively;\medskip

\noindent$\cdot$ $D\!N_{z}$ denotes the Dirichlet-to-Neumann operator;\medskip

\noindent$\cdot$ $R_{z}$ denotes the resolvent of the free Laplacian:
$R_{z}:=(-\Delta-z^{2})^{-1}$;\medskip

\noindent$\cdot$ $\mathcal{G}_{z}(x)$ denotes $3$-dimensional Green's
function, i.e., $\mathcal{G}_{z}(x):=\frac{e^{iz|x|}}{4\pi|x|}$;\medskip

\noindent$\cdot$ $c_{\Omega}$ denotes the capacitance of $\Omega$;\medskip

\noindent$\cdot$ $\omega_{M}>0$ denotes the Minnaert frequency defined by
$\omega_{M}^{2}:={c_{\Omega}}/{|\Omega|}$;\medskip

\noindent$\cdot$ $\Delta_{\Omega}^{D}$ and $\Delta_{\Omega}^{N}$ denote the
self-adjoint Dirichlet and Neumann Laplacians in $L^{2}(\Omega)$.

\section{\label{eex}The main operator-expansions}

The asymptotic expansion presented in this section provides the main technical
tool of our work. The definitions and known properties regarding the
functional spaces and the boundary integral operators involved in this
construction are recalled in the Appendix.

In what follows, the convergence of the Neumann's series for $(1-L)^{-1}$ when
$\left\Vert L\right\Vert <1$, is used in the following form.

\begin{lemma}
\label{str}Let $L(\varepsilon)$ be a bounded operator family such that
$L(\varepsilon)=L_{0}+L_{1}O(\varepsilon^{\lambda_{1}})+O(\varepsilon
^{\lambda_{2}})$, $0<\lambda_{1}<\lambda_{2}$. If $L_{0}$ has a bounded
inverse and $\varepsilon$ is sufficiently small, then $L(\varepsilon)$ has a
bounded inverse as well and $L(\varepsilon)^{-1}=(1-L_{1}O(\varepsilon
^{\lambda_{1}}))L_{0}^{-1}+O(\varepsilon^{\lambda})$. Here $\lambda
=2\lambda_{1}$ whenever $L_{1}\not =0$, $\lambda=\lambda_{2}$ otherwise.
\end{lemma}

We next assume that $\Omega\subset\mathbb{R}^{3}$ is open and bounded with a
smooth boundary $\Gamma$ (a boundary of class $\mathcal{C}^{1,1}$ would
suffice but, in order to simplify the exposition, here we do not strive for
the maximum of generality).

\begin{theorem}
\label{Lemma_Schur}Let $P_{0}$ and $Q_{0}$ be the spectral projectors  
\eqref{P0}-\eqref{Q0} in $H^{1/2}(\Gamma)$ equipped with the inner product \eqref{ip}
and $\widehat{\oplus}$ let denote the corresponding orthogonal sum. With respect to the
decomposition $H^{1/2}(\Gamma)=\mathrm{ran}(P_{0})\,\widehat{\oplus
}\,\mathrm{ran}(Q_{0})$, the operator%
\begin{equation}
\varepsilon^{2}+\left(  1-\varepsilon^{2}\right)  \left(  1/2+K_{\varepsilon
\omega}\right)  S_{\varepsilon z}S_{\varepsilon\omega}^{-1}\in{\mathcal{L}%
}(H^{1/2}(\Gamma))
\end{equation}
writes as%
\begin{equation}
{\mathbb{M}}(\varepsilon)\equiv%
\begin{bmatrix}
M_{00}(\varepsilon) & M_{01}(\varepsilon)\\
M_{10}(\varepsilon) & M_{11}(\varepsilon)
\end{bmatrix}
:\mathrm{ran}(P_{0})\,\widehat{\oplus}\,\mathrm{ran}(Q_{0})\rightarrow
\mathrm{ran}(P_{0})\,\widehat{\oplus}\,\mathrm{ran}(Q_{0})\,,\label{K_op}%
\end{equation}%
\begin{align}
&  M_{00}(\varepsilon):=P_{0}\left(  \left(  1+\omega^{2}K_{(2)}\right)
\varepsilon^{2}+\left(  \left(  z-\omega\right)  \omega^{2}K_{(2)}S_{\left(
1\right)  }S_{0}^{-1}+\omega^{3}K_{(3)}\right)  \varepsilon^{3}+O(\varepsilon
^{4})\right)  P_{0}\,,\label{A}\\
&  M_{01}(\varepsilon):=P_{0}{O}(\varepsilon^{2})Q_{0}\,,\qquad M_{10}%
(\varepsilon):=Q_{0}{O}(\varepsilon^{2})P_{0}\,,\label{B,C}\\
&  M_{11}(\varepsilon):=Q_{0}\left(  1/2+K_{0}+{O}(\varepsilon^{2})\right)
Q_{0}\,,\label{D}%
\end{align}
where $S_{(n)}$ and $K_{(n)}$, $n\in\mathbb{N}$, are defined in \eqref{Sn} and
\eqref{Kn}. Moreover, the Schur complement of $M_{11}(\varepsilon)$, defined
by
\begin{equation}
C_{00}(\varepsilon):=M_{00}(\varepsilon)-M_{01}(\varepsilon)M_{11}%
(\varepsilon)^{-1}M_{10}(\varepsilon)\,,
\end{equation}
writes as
\begin{equation}
C_{00}(\varepsilon)=P_{0}\left(  \left(  1+\omega^{2}K_{(2)}\right)
\varepsilon^{2}+\left(  \left(  z-\omega\right)  \omega^{2}K_{(2)}S_{\left(
1\right)  }S_{0}^{-1}+\omega^{3}K_{(3)}\right)  \varepsilon^{3}+{O}%
(\varepsilon^{4})\right)  P_{0}\,.\label{Schur_complement}%
\end{equation}
For each $\omega\in\mathbb{C}$, such expansions hold whenever $\varepsilon$ is
sufficiently small, uniformly with respect to $z$ in any fixed ball of
$\mathbb{C}$.
\end{theorem}

\begin{proof}
By Lemma \ref{anS}, when $\varepsilon$ is sufficiently small, $\varepsilon
\omega\in\mathbb{C}\backslash D_{\Omega}$ and the expansion
\begin{align}
S_{\varepsilon z}S_{\varepsilon\omega}^{-1}=  &  1+\left(  S_{\varepsilon
z}-S_{\varepsilon\omega}\right)  S_{\varepsilon\omega}^{-1}\nonumber\\
=  &  1+\left(  \varepsilon(z-\omega)S_{\left(  1\right)  }+O((\varepsilon
|z|)^{2})+O((\varepsilon\omega)^{2})\right)  \left(  S_{0}^{-1}+{O}%
(\varepsilon\omega)\right) \nonumber\\
=  &  1+\varepsilon(z-\omega)S_{\left(  1\right)  }S_{0}^{-1}+{O}%
(\varepsilon^{2})\,. \label{S_exp}%
\end{align}
holds uniformly with respect to $z$ in any fixed ball of $\mathbb{C}$. By the
definition \eqref{Sn}, $\mathrm{ran}(S_{\left(  1\right)  })=\mathbb{C}%
=\mathrm{ran}(P_{0})$; it follows%
\[
S_{\varepsilon z}S_{\varepsilon\omega}^{-1}=1+\varepsilon\left(
z-\omega\right)  P_{0}S_{\left(  1\right)  }S_{0}^{-1}+{O}(\varepsilon^{2})
\]
and
\[
Q_{0}S_{\varepsilon z}S_{\varepsilon\omega}^{-1}=Q_{0}(1+O(\varepsilon
^{2}))\,.
\]
Using \eqref{tqq}
\[
\frac{1}{2}+K_{\varepsilon\omega}=\frac{1}{2}+K_{0}+(\varepsilon\omega
)^{2}K_{(2)}+(\varepsilon\omega)^{3}K_{(3)}+{O}((\varepsilon\omega)^{4})\,,
\]
and taking into account the decomposition \eqref{K_0_spectral_decomp}, one
gets
\begin{align}
\frac{1}{2}+K_{\varepsilon\omega}=  &  Q_{0}\left(  \frac{1}{2}+K_{0}\right)
Q_{0}+(\varepsilon\omega)^{2}K_{(2)}+(\varepsilon\omega)^{3}K_{(3)}%
+{O}((\varepsilon\omega)^{4})\nonumber\\
=  &  P_{0}\left(  (\varepsilon\omega)^{2}K_{(2)}+(\varepsilon\omega
)^{3}K_{(3)}+{O}((\varepsilon\omega)^{4})\right)  P_{0}+Q_{0}\left(  \frac
{1}{2}+K_{0}+O((\varepsilon\omega)^{2})\right)  Q_{0}\nonumber\\
&  +P_{0}O((\varepsilon\omega)^{2})Q_{0}+Q_{0}O((\varepsilon\omega)^{2}%
)P_{0}\,. \label{K_exp}%
\end{align}
Then, from \eqref{S_exp} and \eqref{K_exp} follows
\begin{align*}
&  \varepsilon^{2}+\left(  1-\varepsilon^{2}\right)  \left(  \frac{1}%
{2}+K_{\varepsilon\omega}\right)  S_{\varepsilon z}S_{\varepsilon\omega}%
^{-1}\\
=  &  P_{0}\left(  \varepsilon^{2}+\left(  1-\varepsilon^{2}\right)  \left(
\varepsilon^{2}\omega^{2}K_{(2)}+\varepsilon^{3}\omega^{3}K_{(3)}%
+{O}((\varepsilon\omega)^{4})\right)  \right)  P_{0}\left(  1+\varepsilon
\left(  z-\omega\right)  P_{0}S_{\left(  1\right)  }S_{0}^{-1}+{O}%
(\varepsilon^{2})\right)  (P_{0}+Q_{0})\\
&  +Q_{0}\left(  \left(  \frac{1}{2}+K_{0}\right)  +O((\varepsilon\omega
)^{2})\right)  Q_{0}(1+O(\varepsilon^{2}))(P_{0}+Q_{0})+P_{0}O(\varepsilon
^{2})Q_{0}+Q_{0}O(\varepsilon^{2})P_{0}\\
=  &  P_{0}\left(  \varepsilon^{2}\left(  1+\omega^{2}K_{(2)}\right)
+\varepsilon^{3}\left(  \omega^{2}K_{(2)}\left(  z-\omega\right)  S_{(1)}%
S_{0}^{-1}+\omega^{3}K_{(3)}\right)  +{O}(\varepsilon^{4})\right)  P_{0}\\
&  +Q_{0}\left(  \frac{1}{2}+K_{0}+O(\varepsilon^{2})\right)  Q_{0}%
+P_{0}O(\varepsilon^{2})Q_{0}+Q_{0}O(\varepsilon^{2})P_{0}\,.
\end{align*}
This yields (\ref{K_op}).

By $-1/2\in\varrho(Q_{0}K_{0}Q_{0})$, one has $Q_{0}(1/2+K_{0})^{-1}Q_{0}\in
${$\mathcal{L}$}$\left(  \mathrm{ran}(Q_{0})\right)  $; hence, by (\ref{D})
and Lemma \ref{str}, the inverse operator $M_{11}(\varepsilon)^{-1}\in
${$\mathcal{L}$}$\left(  \mathrm{ran}(Q_{0})\right)  $ exists whenever
$\varepsilon$ is sufficiently small and allows the expansion%
\begin{equation}
M_{11}(\varepsilon)^{-1}=Q_{0}(1/2+K_{0})^{-1}Q_{0}+Q_{0}O(\varepsilon
^{2})Q_{0}\,. \label{D_inv}%
\end{equation}
holding for any $\omega\in\mathbb{C}$ uniformly w.r.t. $z$ in any fixed ball
of $\mathbb{C}$. This leads to \eqref{Schur_complement}. \hfill
\end{proof}

\begin{lemma}
\label{Lemma_Schur_coeff}
\[
P_{0}\left(  1+\omega^{2}K_{(2)}\right)  P_{0}=\left(  1-\frac{\omega^{2}%
}{\omega_{M}^{2}}\right)  P_{0}\,.
\]

\end{lemma}

\begin{proof}
By \eqref{Kn} and Green's identity,
\begin{align}
(K_{(2)}1)(x)=  &  \frac{1}{8\pi}\int_{\Gamma}\nu(y)\!\cdot\!\frac{x-y}%
{|x-y|}\,d\sigma(y)=-\frac{1}{8\pi}\int_{\Omega}\Delta|x-y|\,dy\nonumber\\
=  &  -\frac{1}{8\pi}\int_{\mathbb{R}^{3}}\Delta|y|\,1_{\Omega}%
(x-y)\,dy=-\frac{1}{4\pi}\int_{\mathbb{R}^{3}}1_{\Omega}(x-y)\,\frac{dy}%
{|y|}\,, \label{K_2_id}%
\end{align}
i.e.,
\[
K_{(2)}1=-\gamma_{0}R_{0}1_{\Omega}=-S\!L_{0}^{\ast}1_{\Omega}\,.
\]
This yields%
\begin{equation}
\left\langle 1,(1+\omega^{2}K_{(2)}1)\right\rangle _{S_{0}^{-1}}=\Vert
1\Vert_{S_{0}^{-1}}^{2}-\omega^{2}\left\langle 1,\gamma_{0}S\!L_{0}^{\ast
}1_{\Omega}\right\rangle _{S_{0}^{-1}}=c_{\Omega}-\omega^{2}\left\langle
SL_{0}S_{0}^{-1}1,1\right\rangle _{L^{2}(\Omega)}\,. \label{pi}%
\end{equation}
Since $u:=(S\!L_{0}S_{0}^{-1}1)|\Omega$ solves the interior Dirichlet problem%
\[%
\begin{cases}
\Delta_{\Omega}u=0\,,\\
\gamma_{0}^{\mathrm{in}}u=1\,,
\end{cases}
\]
one gets $1_{\Omega}S\!L_{0}S_{0}^{-1}1=1_{\Omega}$ and \eqref{pi} reduces to%
\[
\big\langle1,(1+\omega^{2}K_{(2)}1)\big\rangle_{S_{0}^{-1}}=c_{\Omega}%
-\omega^{2}|\Omega|\,.
\]
Therefore
\begin{equation}
P_{0}\left(  1+\omega^{2}K_{(2)}\right)  P_{0}=c_{\Omega}^{-1}%
\,\big\langle1,\big(1+\omega^{2}K_{(2)}1\big)\big\rangle_{S_{0}^{-1}%
}1=(1-\omega^{2}\,c_{\Omega}^{-1}|\Omega|)\,P_{0}=\left(  1-\frac{\omega^{2}%
}{\omega_{M}^{2}}\right)  P_{0}\,. \label{thr}%
\end{equation}

\hfill
\end{proof}

\begin{lemma}
\label{k3}
\[
P_{0}K_{(3)}P_{0}=-i\,\frac{|\Omega|}{4\pi}\,P_{0}\,.
\]

\end{lemma}

\begin{proof}
By \eqref{Kn} and the divergence theorem,%
\begin{align*}
\langle1,K_{(3)}1\rangle_{S_{0}^{-1}}=  &  \frac{i}{12\pi}\int_{\Gamma}%
(S_{0}^{-1}1)(x)\left(  \int_{\Gamma}\nu(y)\cdot(x-y)\,d\sigma(y)\right)
d\sigma(x)\\
=  &  \frac{i}{12\pi}\int_{\Gamma}\nu(y)\cdot\left(  \int_{\Gamma}(S_{0}%
^{-1}1)(x)(x-y)\,d\sigma(x)\right)  d\sigma(y)\\
=  &  -i\,\frac{c_{\Omega}}{12\pi}\int_{\Gamma}\nu(y)\cdot y\,d\sigma
(y)=-i\,\frac{c_{\Omega}}{12\pi}\,\int_{\Omega}\nabla\!\cdot\!y\ dy\\
=  &  -i\,c_{\Omega}\,\frac{|\Omega|}{4\pi}\,.
\end{align*}
Hence
\[
P_{0}K_{(3)}P_{0}=c_{\Omega}^{-1}\langle1,K_{(3)}1\rangle_{S_{0}^{-1}%
}1=-i\,\frac{|\Omega|}{4\pi}\,P_{0}\,.
\]
\hfill
\end{proof}

The capacitance of the set $\Omega$ and the Minnaert frequency, already
introduced in (\ref{C_Omega})-(\ref{omega_M_def}), express as%
\begin{equation}
c_{\Omega}=\Vert1\Vert_{S_{0}^{-1}}^{2}=\int_{\Gamma}(S_{0}^{-1}%
1)(x)\,d\sigma(x)\,, \label{cap}%
\end{equation}
and%
\begin{equation}
\omega_{M}=\sqrt{\frac{c_{\Omega}}{|\Omega|}}\,,
\end{equation}
(see Appendix, Subsection \ref{Sec_SL}, for the definitions of $S_{0}$ and of the
associated inner product).

\begin{theorem}
\label{Proposition_Schur}Let $\omega\in\mathbb{C}$, $r_{0}>r_{1}>0$ and
$P_{0}$ be the rank-one projector defined in \eqref{P0}. There exists
$\varepsilon_{0}>0$ such that, whenever $0<\varepsilon<\varepsilon_{0}$ the
following holds true:\medskip\newline1) If $\omega\neq\omega_{M}$ and
$\left\vert z\right\vert <r_{0}$ then the linear operator
\begin{equation}
\varepsilon^{2}+\left(  1-\varepsilon^{2}\right)  \left(  1/2+K_{\varepsilon
\omega}\right)  S_{\varepsilon z}S_{\varepsilon\omega}^{-1} \label{to}%
\end{equation}
has a bounded inverse in {$\mathcal{L}$}$(H^{1/2}(\Gamma))$. Moreover, setting%
\begin{equation}
E_{\omega}^{0}:=1-\frac{\omega^{2}}{\omega_{M}^{2}}\,, \label{E_0}%
\end{equation}
the expansion
\begin{equation}
\varepsilon^{2}\left(  \varepsilon^{2}+\left(  1-\varepsilon^{2}\right)
\left(  1/2+K_{\varepsilon\omega}\right)  S_{\varepsilon z}S_{\varepsilon
\omega}^{-1}\right)  ^{-1}=\frac{1}{E_{\omega}^{0}}\,P_{0}+P_{0}%
{O}(\varepsilon)P_{0}+{O}(\varepsilon^{2})\,, \label{exp1}%
\end{equation}
holds uniformly w.r.t. $z$.\medskip\newline2) If $\omega=\omega_{M}$ and
$r_{1}<\left\vert z\right\vert <r_{0}$, then the linear operator \eqref{to}
has a bounded inverse in {$\mathcal{L}$}$(H^{1/2}(\Gamma))$ and the expansion
\begin{equation}
\varepsilon^{3}\left(  \varepsilon^{2}+\left(  1-\varepsilon^{2}\right)
\left(  1/2+K_{\varepsilon\omega}\right)  S_{\varepsilon z}S_{\varepsilon
\omega}^{-1}\right)  ^{-1}=\frac{4\pi}{c_{\Omega}}\,\frac{i}{z}\,P_{0}%
+P_{0}{O}(\varepsilon)P_{0}+{O}(\varepsilon^{2})\,. \label{exp2}%
\end{equation}
holds uniformly w.r.t. $z$.
\end{theorem}

\begin{proof}
By Lemma \ref{Lemma_Schur}, (\ref{to}) represents as a block operator matrix
${\mathbb{M}}(\varepsilon)$ acting on the decomposition $H^{1/2}%
(\Gamma)=\mathrm{ran}(P_{0})\,\widehat{\oplus}\,\mathrm{ran}(Q_{0})$,
$Q_{0}:=1-P_{0}$. By the formula for the inversion of block operator matrices,
one has
\begin{equation}
\left(  {\mathbb{M}}(\varepsilon)\right)  ^{-1}\equiv%
\begin{bmatrix}
\left(  C_{00}(\varepsilon)\right)  ^{-1} & -\left(  C_{00}(\varepsilon
)\right)  ^{-1}M_{01}(\varepsilon)\left(  M_{11}(\varepsilon)\right)  ^{-1}\\
-\left(  M_{11}(\varepsilon)\right)  ^{-1}M_{10}(\varepsilon)\left(
C_{00}(\varepsilon)\right)  ^{-1} & \left(  M_{11}(\varepsilon)\right)
^{-1}+\left(  M_{11}(\varepsilon)\right)  ^{-1}M_{10}(\varepsilon)\left(
C_{00}(\varepsilon)\right)  ^{-1}M_{01}(\varepsilon)M_{11}(\varepsilon)
\end{bmatrix}
\label{bbM}%
\end{equation}
where, setting%
\begin{equation}
E_{\omega}^{1}:=-i\,\omega^{3}\,\frac{|\Omega|}{4\pi}\,, \label{E_1}%
\end{equation}
by \eqref{Schur_complement} and Lemma \ref{k3} the expansion
\begin{equation}
C_{00}(\varepsilon)=P_{0}\left(  E_{\omega}^{0}\varepsilon^{2}+\left(
E_{\omega}^{1}+\omega^{2}\left(  z-\omega\right)  K_{(2)}S_{\left(  1\right)
}S_{0}^{-1}\right)  \varepsilon^{3}+{O}(\varepsilon^{4})\right)  P_{0}\,,
\label{swing_eq}%
\end{equation}
holds. In particular, for each $\omega\in\mathbb{C}$, the remainder
${O}(\varepsilon^{4})$ has a uniform bound $\sim\varepsilon^{4}$ w.r.t.
$z:\left\vert z\right\vert <r_{0}$, provided that $\varepsilon_{0}$ is small
enough depending on $\omega$ and $r_{0}$ (see Lemma \ref{Lemma_Schur}).
\par\noindent
1) If $\omega\neq\omega_{M}$, by Lemma \ref{Lemma_Schur_coeff} $E_{\omega}%
^{0}\neq0$; then the Schur complement writes as%
\[
C_{00}(\varepsilon)=E_{\omega}^{0}\varepsilon^{2}P_{0}\left(  1+\left(
E_{\omega}^{1}+\omega^{2}\left(  z-\omega\right)  K_{(2)}S_{\left(  1\right)
}S_{0}^{-1}\right)  \frac{\varepsilon}{E_{\omega}^{0}}+{O}(\varepsilon
^{2})\right)  P_{0}\,.
\]
Since ${O}(\varepsilon^{2})\lesssim\varepsilon^{2}$ uniformly w.r.t. $z$ s.t.
$\left\vert z\right\vert <r_{0}$, Lemma \ref{str} applies to the r.h.s.
whenever both $\varepsilon\left\vert \omega\right\vert $ and $\varepsilon|z|$
are sufficiently small. Hence, for each $\omega$ and $r_{0}$, there exists
$\varepsilon_{0}>0$ small enough that the expansion%
\begin{equation}
\left(  C_{00}(\varepsilon)\right)  ^{-1}=\frac{1}{\varepsilon^{2}}\,\frac
{1}{E_{\omega}^{0}}\,P_{0}\left(  1-\left(  E_{\omega}^{1}+\omega^{2}\left(
z-\omega\right)  K_{(2)}S_{\left(  1\right)  }S_{0}^{-1}\right)
\frac{\varepsilon}{E_{\omega}^{0}}+{O}(\varepsilon^{2})\right)  P_{0}\,,
\label{C00-1}%
\end{equation}
holds uniformly whenever $\left\vert z\right\vert <r_{0}$. From \eqref{bbM},
\eqref{B,C}, \eqref{D_inv} and \eqref{C00-1} there follows%
\[
\varepsilon^{2}\left(  {\mathbb{M}}(\varepsilon)\right)  ^{-1}=%
\begin{bmatrix}
({E_{\omega}^{0}})^{-1}{P_{0}}+P_{0}O(\varepsilon)P_{0} & P_{0}O(\varepsilon
^{2})Q_{0}\\
Q_{0}O(\varepsilon^{2})P_{0} & Q_{0}O(\varepsilon^{2})Q_{0}%
\end{bmatrix}
\,.
\]
2) Let us now assume that $\omega=\omega_{M}$, so that: $E_{\omega_{M}}^{0}%
=0$. Since, by \eqref{Sn} and \eqref{cap}, $S_{\left(  1\right)  }S_{0}%
^{-1}\left(  1\right)  =i\,{c_{\Omega}}/{4\pi}$, i.e.,
\[
S_{\left(  1\right)  }S_{0}^{-1}P_{0}=i\,\frac{c_{\Omega}}{4\pi}\,P_{0}\,,
\]
one has
\[
\left(  z-\omega\right)  \omega^{2}P_{0}K_{(2)}S_{\left(  1\right)  }%
S_{0}^{-1}P_{0}=i\,\frac{c_{\Omega}}{4\pi}\,\left(  z-\omega\right)
\omega^{2}P_{0}K_{(2)}P_{0}=-i\,\frac{c_{\Omega}}{4\pi}\,\left(
z-\omega\right)  P_{0}\,.
\]
Moreover, from (\ref{E_1}) and $\omega_{M}^{2}:=c_{\Omega}/\left\vert
\Omega\right\vert $ there follows: $E_{\omega_{M}}^{1}=-i\,\omega_{M}%
\,\frac{c_{\Omega}}{4\pi}$. Then, \eqref{swing_eq} recasts as%
\begin{equation}
C_{00}(\varepsilon)=\left(  -i\,\omega_{M}\,\frac{c_{\Omega}}{4\pi}%
-i\,\frac{c_{\Omega}}{4\pi}\,\left(  z-\omega_{M}\right)  \right)
\varepsilon^{3}P_{0}+P_{0}{O}(\varepsilon^{4})P_{0}=\left(  -i\,\frac
{c_{\Omega}}{4\pi}\,z\,P_{0}+P_{0}{O}(\varepsilon)P_{0}\right)  \varepsilon
^{3}\,. \label{swing_eq_res}%
\end{equation}
and, for $z\not =0$, we obtain%
\[
C_{00}(\varepsilon)=-i\,\frac{c_{\Omega}}{4\pi}\,z\,P_{0}\left(  1+P_{0}%
\frac{1}{z}{O}\left(  \varepsilon\right)  P_{0}\right)  \varepsilon^{3}\,.
\]
Let us recall that ${O}\left(  \varepsilon\right)  $ has a uniform bound:
$\sup_{\left\vert z\right\vert <r_{0}}\left\Vert {O}\left(  \varepsilon
\right)  \right\Vert <C\varepsilon$, provided that $\varepsilon_{0}$ is small
enough and $0<\varepsilon<\varepsilon_{0}$. Choosing any $r_{1}<r_{0}$ such
that: $C\varepsilon_{0}/r_{1}<1$, the Neumann series%
\[%
{\textstyle\sum_{j=0}^{+\infty}}
\left(  -1\right)  ^{j}\left(  P_{0}\frac{1}{z}{O}\left(  \varepsilon\right)
P_{0}\right)  ^{j}=\left(  1+P_{0}\frac{1}{z}{O}\left(  \varepsilon\right)
P_{0}\right)  ^{-1}\,,
\]
converges in {$\mathcal{L}$}$\left(  H^{1/2}\left(  \Gamma\right)  \right)  $
and the expansion%
\begin{equation}
C_{00}(\varepsilon)^{-1}=\left(  \frac{4\pi}{c_{\Omega}}\,\frac{i}{z}%
\,P_{0}+P_{0}{O}\left(  \varepsilon\right)  P_{0}\right)  \frac{1}%
{\varepsilon^{3}}\,, \label{Schur_complement_inv_res}%
\end{equation}
holds uniformly w.r.t. $z:r_{1}<\left\vert z\right\vert <r_{0}$. By
\eqref{bbM}, \eqref{B,C}, \eqref{D_inv} and (\ref{Schur_complement_inv_res}),
one gets
\[
\varepsilon^{3}{\mathbb{M}}(\varepsilon)^{-1}=%
\begin{bmatrix}
\frac{4\pi}{c_{\Omega}}\,\frac{i}{z}\,P_{0}+P_{0}O(\varepsilon)P_{0} &
P_{0}O(\varepsilon^{2})Q_{0}\\
Q_{0}O(\varepsilon^{2})P_{0} & Q_{0}O(\varepsilon^{2})Q_{0}%
\end{bmatrix}
\,.
\]
\hfill
\end{proof}

Let us define%
\begin{equation}
\Lambda_{z}^{\omega}(\varepsilon):=\varepsilon(1-\varepsilon^{2}%
)(\varepsilon^{2}+(1-\varepsilon^{2})D\!N_{\varepsilon\omega}S_{\varepsilon
z})^{-1}D\!N_{\varepsilon\omega}\,. \label{Lambda_def}%
\end{equation}
By Theorem \ref{Proposition_Schur} this operator has the following asymptotic representation.

\begin{theorem}
\label{Theorem_Schur}Let $\omega\in\mathbb{C}$, $r_{0}>r_{1}>0$, $P_{0}$ be
the rank-one projector defined in \eqref{P0} and $E_{\omega}^{0}$ be given by
\eqref{E_0}. There exists $\varepsilon_{0}>0$ such that, whenever
$0<\varepsilon<\varepsilon_{0}$, the following holds true:\medskip\newline1)
If $\omega\not =\omega_{M}$, then: $z\mapsto\Lambda_{z}^{\omega}(\varepsilon)$
is a $\mathcal{L}(H^{1/2}(\Gamma),H^{-1/2}(\Gamma))$-valued analytic map in
the ball $\left\{  z:\left\vert z\right\vert <r_{0}\right\}  $ where it has
the uniform-in-$z$ expansion%
\begin{equation}
\left.  \Lambda_{z}^{\omega}(\varepsilon)=\frac{1}{\varepsilon}%
\,S_{\varepsilon\omega}^{-1}\left(  \frac{P_{0}}{E_{\omega}^{0}}+P_{0}%
{O}(\varepsilon)P_{0}+{O}(\varepsilon^{2})\right)  S_{\varepsilon\omega
}D\!N_{\varepsilon\omega}\,.\right.  \label{vepsexp1}%
\end{equation}
2) If $\omega=\omega_{M}$, then: $z\mapsto\Lambda_{z}^{\omega}(\varepsilon)$
is a $\mathcal{L}(H^{1/2}(\Gamma),H^{-1/2}(\Gamma))$-valued analytic map in
$\left\{  z:r_{1}<\left\vert z\right\vert <r_{0}\right\}  \backslash\left\{
0\right\}  $ where it has the uniform-in-$z$ expansion%
\begin{equation}
\Lambda_{z}^{\omega}(\varepsilon)=\frac{1}{\varepsilon^{2}}\,S_{\varepsilon
\omega}^{-1}\left(  \frac{4\pi}{c_{\Omega}}\,\frac{i}{z}\,P_{0}+P_{0}%
{O}(\varepsilon)P_{0}+{O}(\varepsilon^{2})\right)  S_{\varepsilon\omega
}D\!N_{\varepsilon\omega}\,. \label{vepsexp2}%
\end{equation}

\end{theorem}

\begin{proof}
According to \eqref{DN_k_id},
\[
D\!N_{\varepsilon\omega}=S_{\varepsilon\omega}^{-1}\left(  \frac{1}%
{2}+K_{\varepsilon\omega}\right)  \,.
\]
Hence,
\begin{equation}
\varepsilon^{2}+\left(  1-\varepsilon^{2}\right)  D\!N_{\varepsilon\omega
}S_{\varepsilon z}=S_{\varepsilon\omega}^{-1}\left(  \varepsilon^{2}+\left(
1-\varepsilon^{2}\right)  \left(  \frac{1}{2}+K_{\varepsilon\omega}\right)
S_{\varepsilon z}S_{\varepsilon\omega}^{-1}\right)  S_{\varepsilon\omega}
\label{dir_op}%
\end{equation}
and, by (\ref{Lambda_def}),
\begin{equation}
\Lambda_{z}^{\omega}(\varepsilon)=\varepsilon(1-\varepsilon^{2})S_{\varepsilon
\omega}^{-1}\left(  \varepsilon^{2}+\left(  1-\varepsilon^{2}\right)  \left(
\frac{1}{2}+K_{\varepsilon\omega}\right)  S_{\varepsilon z}S_{\varepsilon
\omega}^{-1}\right)  ^{\!-1}S_{\varepsilon\omega}D\!N_{\varepsilon\omega}\,.
\label{Lambda_id}%
\end{equation}
By Theorem \ref{Proposition_Schur} and the mapping properties of
$S_{\varepsilon\omega}D\!N_{\varepsilon\omega}$ and $S_{\varepsilon\omega
}^{-1}$ (see the Appendix), \eqref{Lambda_id} defines an operator in
${\mathcal{L}}(H^{1/2}(\Gamma),H^{-1/2}(\Gamma))$ for any $z:\left\vert
z\right\vert <r_{0}$ and $\omega\neq\omega_{M}$ or for $z:r_{1}<\left\vert
z\right\vert <r_{0}$ and $\omega=\omega_{M}$, provided that $\varepsilon_{0}$
is small enough depending on $\omega$, $r_{0}$ and $r_{1}$. From the
analyticity of $z\mapsto S_{z}$ (see Lemma \ref{anS}) follows the analyticity
of the operator (\ref{to}); the analyticity of $z\mapsto\Lambda_{z}^{\omega
}(\varepsilon)$ is then consequence of the existence the inverse, shown in
Theorem \ref{Proposition_Schur}, and of the analyticity of the inverse (see
\cite[Theorem 5.1]{Tay}).

The $\varepsilon$-expansions \eqref{vepsexp1} and \eqref{vepsexp2} follow from
the ones provided in Theorem \ref{Proposition_Schur}.
\end{proof}

\section{\label{Sec_H_dil}The operator model for acoustic interface
conditions}

Here we introduce the Schr\"{o}dinger-type operators modeling acoustic
interface conditions. Their construction involves the boundary operators whose
existence and mapping properties have been discussed in the small-scale limit
in Section \ref{eex}.\hfill

\subsection{The dilated operator}

In this subsection we provide, together with its resolvent, a self-adjoint
realization of the Laplacian with boundary conditions at the interface
$\Gamma$ separating $\Omega_{\mathrm{in}}=\Omega$ from $\Omega_{ex}%
=\mathbb{R}^{3}\backslash\overline{\Omega}$ given by%
\begin{equation}
\lbrack\gamma_{0}]u=0\,,\qquad\lbrack\gamma_{1}]u=(\varepsilon^{-2}%
-1)D\!N_{\varepsilon\omega}\gamma_{0}u\,. \label{BC_Omega}%
\end{equation}
The vector space $H_{\Delta}^{0}(\mathbb{R}^{3}\backslash\Gamma)$ appearing in
the next theorem is defined as the set of the functions $u\in L^{2}%
(\mathbb{R}^{3})$ such that the distributional Laplacian $\Delta
_{\mathbb{R}^{3}\backslash\Gamma}\,u$ is in $L^{2}(\mathbb{R}^{3})$ (see
(\ref{Laplacian_weak}) and (\ref{hzd}) for more details).

\begin{theorem}
\label{heo}Let $\omega>0$ and $H_{\varepsilon,\omega}$ be the restriction of
\begin{equation}
\Delta:H_{\Delta}^{0}(\mathbb{R}^{3}\backslash\Gamma)\subset L^{2}%
(\mathbb{R}^{3})\rightarrow L^{2}(\mathbb{R}^{3}) \label{maxLap}%
\end{equation}
to the domain
\begin{equation}
\mathrm{dom}(H_{\varepsilon,\omega})=\left\{  u\in H_{\Delta}^{0}%
(\mathbb{R}^{3}\backslash\Gamma)\cap H^{1}(\mathbb{R}^{3}):[\gamma
_{1}]u=(\varepsilon^{-2}-1)D\!N_{\varepsilon\omega}\gamma_{0}u\right\}  \,.
\label{dominio}%
\end{equation}
There exists $\varepsilon_{0}>0$ sufficiently small that, for all
$0<\varepsilon<\varepsilon_{0}$, $H_{\varepsilon,\omega}$ is a self-adjoint
and semi-bounded operator in $L^{2}(\mathbb{R}^{3})$. Moreover, for any
$z\in\mathbb{C}_{+}$ such that $z^{2}\in\varrho(-H_{\varepsilon,\omega}%
)\cap(\mathbb{C}\backslash\lbrack0,+\infty))$, hence at least for any
$z\in\mathbb{C}_{+}\backslash i\mathbb{R}_{+}$, its resolvent is given by
\begin{equation}
R_{z}^{\varepsilon,\omega}:=(-H_{\varepsilon,\omega}-z^{2})^{-1}%
=R_{z}-S\!L_{z}\left(  (\varepsilon^{-2}-1)^{-1}+D\!N_{\varepsilon\omega}%
S_{z}\right)  ^{-1}D\!N_{\varepsilon\omega}\,\gamma_{0}R_{z}\,,
\label{resolvent}%
\end{equation}
where $R_{z}$ denotes the free resolvent, i.e., $R_{z}=(-\Delta-z^{2})^{-1}$.
\end{theorem}

\begin{proof}
Here, for the sake of brevity, we set
\begin{equation}
M_{z}^{\varepsilon,\omega}:=(\varepsilon^{-2}-1)^{-1}+D\!N_{\varepsilon\omega
}S_{z}\,. \label{defM}%
\end{equation}
By \eqref{DN_k_id} and Lemma \ref{anS},
\begin{align*}
&  M_{z}^{\varepsilon,\omega}=(\varepsilon^{-2}-1)^{-1}\left(  1+(\varepsilon
^{-2}-1)D\!N_{\varepsilon\omega}S_{z}\right) \\
=  &  \varepsilon^{-2}(\varepsilon^{-2}-1)^{-1}\left(  \varepsilon
^{2}+(1-\varepsilon^{2})D\!N_{\varepsilon\omega}S_{z}\right) \\
=  &  \varepsilon^{-2}(\varepsilon^{-2}-1)^{-1}\left(  \varepsilon
^{2}+(1-\varepsilon^{2})S_{\varepsilon\omega}^{-1}\left(  \frac{1}%
{2}+K_{\varepsilon\omega}\right)  S_{z}\right) \\
=  &  \varepsilon^{-2}(\varepsilon^{-2}-1)^{-1}S_{\varepsilon\omega}%
^{-1}\left(  \varepsilon^{2}+(1-\varepsilon^{2})\left(  \frac{1}%
{2}+K_{\varepsilon z}\right)  S_{z}S_{\varepsilon\omega}^{-1}\right)
S_{\varepsilon\omega}\,.
\end{align*}
Let us fix $z\in\mathbb{C}_{+}$; due to Theorem \ref{Proposition_Schur} and to
the mapping properties of $S_{z}$, there exists $\varepsilon_{0}>0$ such that,
whenever $0<\varepsilon<\varepsilon_{0}$, $M_{z}^{\varepsilon,\omega}$ has a
bounded inverse $(M_{z}^{\varepsilon,\omega})^{-1}\in\mathcal{L}%
(H^{-1/2}(\Gamma))$. By $S_{z}^{\ast}=S_{-\bar{z}}$ and $D\!N_{\varepsilon
\omega}^{\ast}=D\!N_{\varepsilon\omega}$,
\[
D\!N_{\varepsilon\omega}(M_{z}^{\varepsilon,\omega})^{\ast}=D\!N_{\varepsilon
\omega}\left(  (\varepsilon^{-2}-1)^{-1}+S_{-\bar{z}}D\!N_{\varepsilon\omega
})\right)  =M_{-\bar{z}}^{\varepsilon,\omega}D\!N_{\varepsilon\omega}\,,
\]
and so
\begin{equation}
(-M_{-\bar{z}}^{\varepsilon,\omega})^{-1}D\!N_{\varepsilon\omega
}=D\!N_{\varepsilon\omega}((-M_{z}^{\varepsilon,\omega})^{\ast})^{-1}%
=D\!N_{\varepsilon\omega}((-M_{z}^{\varepsilon,\omega})^{-1})^{\ast}=\left(
(-M_{z}^{\varepsilon,\omega})^{-1}D\!N_{\varepsilon\omega}\right)  ^{\ast}\,.
\label{w1}%
\end{equation}
By the first resolvent identity,
\[
S\!L_{w}-S\!L_{z}=(w^{2}-z^{2})R_{w}S\!L_{z}\,,
\]
and so
\[
M_{w}^{\varepsilon,\omega}-M_{z}^{\varepsilon,\omega}=(w^{2}-z^{2}%
)D\!N_{\varepsilon\omega}\gamma_{0}R_{w}S\!L_{z}=(w^{2}-z^{2}%
)D\!N_{\varepsilon\omega}S\!L_{-\bar{w}}^{\ast}S\!L_{z}%
\]
This gives
\begin{equation}
(-M_{w}^{\varepsilon,\omega})^{-1}D\!N_{\varepsilon\omega}-(-M_{z}%
^{\varepsilon,\omega})^{-1}D\!N_{\varepsilon\omega}=(-z^{2}-(-w^{2}%
))(M_{w}^{\varepsilon,\omega})^{-1}D\!N_{\varepsilon\omega}S\!L_{-\bar{z}%
}^{\ast}S\!L_{z}(M_{w}^{\varepsilon,\omega})^{-1}D\!N_{\varepsilon\omega}\,.
\label{w2}%
\end{equation}
Let us remark that \eqref{w1} and \eqref{w2} correspond to \cite[relations
(2.6) and (2.7)]{MaPo SM} (be aware of the different notation and convention:
our $(-M_{w}^{\varepsilon,\omega})^{-1}D\!N_{\varepsilon\omega}$ corresponds
to $\Lambda_{z}$ in \cite{MaPo SM}, while our $R_{z}:=(-\Delta-z^{2})^{-1}$ is
there denoted with $R_{z}^{0})$. Hence \cite[Theorem 2.4]{MaPo SM} applies and
we conclude that, for the fixed $z$
\begin{align}
&  \widetilde{R}_{z}^{\varepsilon,\omega}:=R_{z}+S\!L_{z}(-M_{z}%
^{\varepsilon,\omega})^{-1}D\!N_{\varepsilon\omega}\gamma_{0}R_{z}\nonumber\\
=  &  R_{z}-S\!L_{z}((\varepsilon^{-2}-1)^{-1}+D\!N_{\varepsilon\omega}%
S_{z})^{-1}D\!N_{\varepsilon\omega}\gamma_{0}R_{z} \label{sarf}%
\end{align}
is the resolvent of a self-adjoint operator $\widetilde{H}_{\varepsilon
,\omega}$ in $L^{2}(\mathbb{R}^{3})$ which extends $\Delta|\mathrm{ker}%
(\gamma_{0})$. By \cite[Theorem 2.19]{CFP}, such a resolvent formula extends
to all $z\in\mathbb{C}_{+}$ such that $z^{2}\in\varrho(-\widetilde{H}%
_{\varepsilon,\omega})\cap(\mathbb{C}\backslash\lbrack0,+\infty))$; in
particular, by the self-adjointness of $\widetilde{H}_{\varepsilon,\omega}$,
(\ref{sarf}) holds at least for any $z\in\mathbb{C}_{+}\backslash
i\mathbb{R}_{+}$.

Let us now show that $\widetilde{H}_{\varepsilon,\omega}=H_{\varepsilon
,\omega}$. By the mapping properties of $S\!L_{z}$ (see \eqref{slmp}),
$S\!L_{z}(M_{z}^{\varepsilon,\omega})^{-1}$ has values in $H^{1}%
(\mathbb{R}^{3}\backslash\Gamma)$ and so, by $[\gamma_{0}]S\!L_{z}=0$ (see
\eqref{jump}), one gets $\mathrm{dom}(\widetilde{H}_{\varepsilon,\omega
})\subseteq H^{1}(\mathbb{R}^{3})$. By Green's formula \eqref{hG}, taking into
account the boundary conditions in \eqref{dominio}, one readily can check that
$H_{\varepsilon,\omega}$ is a symmetric operator. Hence, since $\widetilde{H}%
_{\varepsilon,\omega}\subset(\Delta|\mathrm{ker}(\gamma_{0}))^{\ast}%
=\Delta|H_{\Delta}^{0}(\mathbb{R}^{3}\backslash\Gamma)$, it suffices to show
that
\[
\mathrm{dom}(\widetilde{H}_{\varepsilon,\omega})\equiv\{\widetilde{u}\in
H^{1}(\mathbb{R}^{3}):\widetilde{u}=u_{0}-S\!L_{z}(M_{z}^{\varepsilon,\omega
})^{-1}D\!N_{\varepsilon\omega}\gamma_{0}u_{0},\ u_{0}\in H^{2}(\mathbb{R}%
^{3})\}\subseteq\mathrm{dom}(H_{\varepsilon,\omega})\,.
\]
To this aim, let us notice at first that the identity%
\[
(-\Delta-z^{2})S\!L_{z}=(-\Delta-z^{2})\mathcal{G}_{z}\ast\gamma_{0}^{\ast
}=\delta_{0}\ast\gamma_{0}^{\ast}=\gamma_{0}^{\ast}\,,
\]
implies: $-\Delta S\!L_{z}(x)=z^{2}S\!L_{z}(x)$ for all $x\notin\Gamma$ and so
$\mathrm{dom}(\widetilde{H}_{\varepsilon,\omega})\subseteq H_{\Delta}%
^{0}(\mathbb{R}^{3}\backslash\Gamma)$.
Moreover, by \eqref{jump} and the definition \eqref{defM}, one gets
\begin{align*}
&  M_{z}^{\varepsilon,\omega}[\gamma_{1}]\widetilde{u}=-M_{z}^{\varepsilon
,\omega}[\gamma_{1}]S\!L_{z}(M_{z}^{\varepsilon,\omega})^{-1}D\!N_{\varepsilon
\omega}\gamma_{0}u_{0}=D\!N_{\varepsilon\omega}\gamma_{0}u_{0}\\
=  &  \left(  (\varepsilon^{-2}-1)M_{z}^{\varepsilon,\omega}+1-(\varepsilon
^{-2}-1)M_{z}^{\varepsilon,\omega}\right)  D\!N_{\varepsilon\omega}\gamma
_{0}u_{0}\\
=  &  \left(  (\varepsilon^{-2}-1)M_{z}^{\varepsilon,\omega}-(\varepsilon
^{-2}-1)M_{z}^{\varepsilon,\omega}D\!N_{\varepsilon\omega}S_{z}(M_{z}%
^{\varepsilon,\omega})^{-1}\right)  D\!N_{\varepsilon\omega}\gamma_{0}u_{0}\\
=  &  (\varepsilon^{-2}-1)M_{z}^{\varepsilon,\omega}D\!N_{\varepsilon\omega
}\gamma_{0}\left(  u_{0}-S\!L_{z}(M_{z}^{\varepsilon,\omega})^{-1}%
D\!N_{\varepsilon\omega}\gamma_{0}u_{0}\right) \\
=  &  (\varepsilon^{-2}-1)M_{z}^{\varepsilon,\omega}D\!N_{\varepsilon\omega
}\gamma_{0}\widetilde{u}\,.
\end{align*}
Since $M_{z}^{\varepsilon,\omega}$ is a bijection, this is equivalent to
\[
\lbrack\gamma_{1}]\widetilde{u}=(\varepsilon^{-2}-1)D\!N_{\varepsilon\omega
}\gamma_{0}\widetilde{u}\,.
\]

Finally, let us show that $H_{\varepsilon,\omega}$ is semi-bounded. Again by
Green's formula \eqref{hG}, for any $u\in\mathrm{dom}(H_{\varepsilon,\omega})$
and for any $s\in(0,1/2)$, one gets
\[
\left\langle -H_{\varepsilon,\omega}u,u\right\rangle _{L^{2}\left(
\mathbb{R}^{3}\right)  }=\left\Vert \nabla u\right\Vert _{L^{2}\left(
\mathbb{R}^{3}\right)  }^{2}+\left(  \varepsilon^{-2}-1\right)  \left\langle
D\!N_{\varepsilon\omega}\gamma_{0}u,\gamma_{0}u\right\rangle _{H^{-s}%
(\Gamma),H^{s}(\Gamma)}\,.
\]
By
\begin{align*}
\left\vert \left\langle D\!N_{\varepsilon\omega}\gamma_{0}u,\gamma
_{0}u\right\rangle _{H^{-s}\left(  \Gamma\right)  ,H^{s}\left(  \Gamma\right)
}\right\vert \leq &  \left\Vert D\!N_{\varepsilon\omega}\right\Vert
_{H^{s}\left(  \Gamma\right)  ,H^{-s}\left(  \Gamma\right)  }\left\Vert
\gamma_{0}\right\Vert _{H^{s+1/2}\left(  \mathbb{R}^{3}\right)  ,H^{s}\left(
\Gamma\right)  }^{2}\left\Vert u\right\Vert _{H^{s+1/2}\left(  \mathbb{R}%
^{3}\right)  }^{2}\\
\equiv &  \,c\,\left\Vert u\right\Vert _{H^{s+1/2}\left(  \mathbb{R}%
^{3}\right)  }^{2}\,,
\end{align*}
and since for any $a>0$ there exists $b>0$ such that
\[
\left\Vert u\right\Vert _{H^{s+1/2}\left(  \mathbb{R}^{3}\right)  }^{2}\leq
a\left\Vert \nabla u\right\Vert _{L^{2}\left(  \mathbb{R}^{3}\right)  }%
^{2}+b\left\Vert u\right\Vert _{L^{2}\left(  \mathbb{R}^{3}\right)  }^{2}\,,
\]
taking $a$ sufficiently small, we obtain
\begin{align*}
\left\langle -H_{\varepsilon,\omega}u,u\right\rangle _{L^{2}\left(
\mathbb{R}^{3}\right)  }\geq &  \left(  1-a\,c\,\left\vert \varepsilon
^{-2}-1\right\vert \right)  \left\Vert \nabla u\right\Vert _{L^{2}\left(
\mathbb{R}^{3}\right)  }^{2}-b\,c\,\left\vert \varepsilon^{-2}-1\right\vert
\left\Vert u\right\Vert _{L^{2}\left(  \mathbb{R}^{3}\right)  }^{2}\\
\geq &  -b\,c\,\left\vert \varepsilon^{-2}-1\right\vert \left\Vert
u\right\Vert _{L^{2}\left(  \mathbb{R}^{3}\right)  }^{2}\,.
\end{align*}

\hfill
\end{proof}

\begin{remark}
The jump condition $[\gamma_{0}]u=0$ holds for $u\in\mathrm{dom}%
(H_{\varepsilon,\omega})$ due to $\mathrm{dom}(H_{\varepsilon,\omega
})\subseteq H^{1}(\mathbb{R}^{3})$.
\end{remark}

\begin{remark}
The Dirichlet-to-Neumann operator $D\!N_{\varepsilon\omega}$ appearing in both
the definitions of $\mathrm{dom}(H_{\varepsilon,\omega})$ and $R_{z}%
^{\varepsilon,\omega}$, is well-defined for any $\omega>0$ and a sufficiently
small $\varepsilon>0$ such that $0<(\varepsilon\omega)^{2}<\lambda_{\Omega}$,
$\lambda_{\Omega}$ denoting the smallest eigenvalue of $-\Delta_{\Omega}^{D}$
(see Subsection \ref{Sec_DN} in the Appendix).
\end{remark}

\begin{remark}
Building upon the theory of singular perturbations presented in \cite{MaPoSi
SE}, the self-adjointness of $H_{\varepsilon,\omega}$ could be proved without
the assumption $\varepsilon\ll1$. Here we prefer to exploit a less technical
construction involving asymptotic estimates for the operator $(\varepsilon
^{-2}-1)^{-1}+D\!N_{\varepsilon\omega}S_{z}$: this allows us to avoid a
slightly burdensome abstract framework, while the asymptotic estimates as
$\varepsilon\rightarrow0_{+}$ provide the natural tool of the subsequent analysis.
\end{remark}

\subsection{\label{Sec_expansions}Dilation identities}

We introduce the smooth map%
\begin{equation}
\Phi_{\varepsilon}(y)=y_{0}+\varepsilon\,\left(  y-y_{0}\right)
\,,\qquad\varepsilon>0\,,\quad y_{0}\in\mathbb{R}^{3}\,; \label{dilation_def}%
\end{equation}
the contracted domain $\Omega^{\varepsilon}$ is then defined by
\begin{equation}
\Omega^{\varepsilon}:=\Phi_{\varepsilon}(\Omega)\equiv\left\{  x\in
\mathbb{R}^{3}:x=y_{0}+\varepsilon\,\left(  y-y_{0}\right)  \,,\ y\in
\Omega\right\}  \,, \label{Omega_eps}%
\end{equation}
while its boundary $\Gamma^{\varepsilon}$ is given by
\[
\Gamma^{\varepsilon}=\Phi_{\varepsilon}(\Gamma)\equiv\left\{  x\in
\mathbb{R}^{3}:x=y_{0}+\varepsilon\,\left(  y-y_{0}\right)  \,,\ y\in
\Gamma\right\}  \,.
\]
The map $\Phi_{\varepsilon}$ and its inverse induce unitary operators on
$L^{2}(\mathbb{R}^{3})$ defined by%
\begin{equation}
\left(  U_{\varepsilon}u\right)  (x):=\varepsilon^{-3/2}\,u(\Phi_{\varepsilon
}^{-1}(x))\,,\qquad\left(  U_{\varepsilon}^{-1}u\right)  (y):=\varepsilon
^{3/2}\,u(\Phi_{\varepsilon}(y))\,. \label{Uve}%
\end{equation}
By the definition of $U_{\varepsilon}$ one gets
\[
\Delta=\varepsilon^{-2}U_{\varepsilon}\Delta U_{\varepsilon}^{-1}%
\]
and hence
\begin{equation}
R_{z}=\varepsilon^{2}U_{\varepsilon}R_{\varepsilon z}U_{\varepsilon}^{-1}\,.
\label{Res_conj}%
\end{equation}
In the next Lemmata, for any linear operator $L$ in spaces of functions on
$\Omega$ (or $\Gamma)$ we denote by $L(\varepsilon)$ the corresponding
operator in spaces of functions on $\Omega^{\varepsilon}$ (or $\Gamma
^{\varepsilon}$).

\begin{lemma}
\label{Lemma_conj0}
\begin{equation}
U_{\varepsilon}\gamma_{0}^{\mathrm{in}/\mathrm{ex}}U_{\varepsilon}^{-1}%
=\gamma_{0}^{\mathrm{in}/\mathrm{ex}}(\varepsilon)\,. \label{trace_conj}%
\end{equation}%
\begin{equation}
U_{\varepsilon}\gamma_{1}^{\mathrm{in}/\mathrm{ex}}U_{\varepsilon}%
^{-1}=\varepsilon\,\gamma_{1}^{\mathrm{in}/\mathrm{ex}}(\varepsilon)\,.
\label{trace_prime_conj}%
\end{equation}

\end{lemma}

\begin{proof}
The statement is an immediate consequence of the definitions.
\end{proof}

\begin{lemma}
\label{Lemma_conj}Let $\omega>0$ and $\varepsilon>0$ such that $(\varepsilon
\omega)^{2}\in\varrho(-\Delta^{D}_{\Omega})$. Then
\begin{equation}
U_{\varepsilon}D\!N_{\varepsilon\omega}U_{\varepsilon}^{-1}=\varepsilon
\,D\!N_{\omega}( \varepsilon) \,. \label{conj_id}%
\end{equation}

\end{lemma}

\begin{proof}
By the definition of $D\!N_{\varepsilon\omega}$, it results%
\[
U_{\varepsilon}D\!N_{\varepsilon\omega}U_{\varepsilon}^{-1}\varphi
:=U_{\varepsilon}\gamma_{1}^{\mathrm{in}}u\,,\qquad%
\begin{cases}
(\Delta_{\Omega}+\varepsilon^{2}\omega^{2})u=0\,,\\
\gamma_{0}^{\mathrm{in}}u=U_{\varepsilon}^{-1}\varphi\,.
\end{cases}
\]
Setting $\tilde{u}=\varepsilon^{2}U_{\varepsilon}u$ one obtains%
\[
0=U_{\varepsilon}\left(  \Delta+\varepsilon^{2}\omega^{2}\right)
U_{\varepsilon}^{-1}U_{\varepsilon}u=\left(  U_{\varepsilon}\Delta
U_{\varepsilon}^{-1}+\varepsilon^{2}\omega^{2}\right)  U_{\varepsilon
}u=\left(  \Delta+\omega^{2}\right)  \varepsilon^{2}U_{\varepsilon}u=\left(
\Delta+\omega^{2}\right)  \tilde{u}\,,
\]
and%
\[
\varphi=U_{\varepsilon}U_{\varepsilon}^{-1}\varphi=U_{\varepsilon}\gamma
_{0}^{\mathrm{in}}u=\gamma_{0}^{\mathrm{in}}(\varepsilon)U_{\varepsilon
}u=\varepsilon^{-2}\gamma_{0}^{\mathrm{in}}(\varepsilon)\tilde{u}\,.
\]
Hence
\[%
\begin{cases}
(\Delta_{\Omega^{\varepsilon}}+\omega^{2})\tilde{u}=0\,,\\
\gamma_{0}^{\mathrm{in}}(\varepsilon)\tilde{u}=\varepsilon^{2}\varphi\,.
\end{cases}
\]
Using (\ref{trace_prime_conj}), this implies%
\[
U_{\varepsilon}D\!N_{\varepsilon\omega}U_{\varepsilon}^{-1}\varphi
=U_{\varepsilon}\gamma_{1}^{\mathrm{in}}U_{\varepsilon}^{-1}U_{\varepsilon
}u=\varepsilon\gamma_{1}^{\mathrm{in}}(\varepsilon)U_{\varepsilon
}u=\varepsilon^{-1}\,\gamma_{1}^{\mathrm{in}}(\varepsilon)\tilde
{u}=\varepsilon^{-1}\,D\!N_{\omega}(\varepsilon)\gamma_{0}^{\mathrm{in}%
}(\varepsilon)\tilde{u}=\varepsilon\,D\!N_{\omega}(\varepsilon)\varphi\,.
\]

\hfill
\end{proof}

\subsection{The model operator $H_{\omega}(\varepsilon)$\medskip}

\label{mod-op}$\left.  \ \right.  $

Let $\omega>0$ and let $\varepsilon>0$ be sufficiently small; we define%
\begin{equation}
\mathrm{dom}(H_{\omega}(\varepsilon)):=U_{\varepsilon}\left(  \mathrm{dom}%
(H_{\varepsilon,\omega})\right)  \,,\qquad H_{\omega}(\varepsilon
):=\varepsilon^{-2}U_{\varepsilon}H_{\varepsilon,\omega}U_{\varepsilon}%
^{-1}\,. \label{H_omega_eps_def}%
\end{equation}
By Theorem \ref{heo}, $H_{\omega}(\varepsilon)$ is a well-defined self-adjoint
and semi-bounded operator in $L^{2}(\mathbb{R}^{3})$ and, by relations
\eqref{trace_conj}, \eqref{trace_prime_conj}, \eqref{conj_id}, it can be more
explicitly defined as the restriction of%
\[
\Delta:H_{\Delta}^{0}(\mathbb{R}^{3}\backslash\Gamma^{\varepsilon})\subset
L^{2}(\mathbb{R}^{3})\rightarrow L^{2}(\mathbb{R}^{3})
\]
to the domain
\begin{equation}
\mathrm{dom}(H_{\omega}(\varepsilon))=\left\{  u\in H_{\Delta}^{0}%
(\mathbb{R}^{3}\backslash\Gamma^{\varepsilon})\cap H^{1}(\mathbb{R}%
^{3}):\left[  \gamma_{1}\left(  \varepsilon\right)  \right]  u=(\varepsilon
^{-2}-1)D\!N_{\omega}(\varepsilon)\gamma_{0}(\varepsilon)u\right\}  \,.
\label{H_omega_eps_dom}%
\end{equation}
Notice that the jump condition $[\gamma_{0}(\varepsilon)]u=0$ is incorporated
into $\mathrm{dom}(H_{\omega}(\varepsilon))\subseteq H^{1}(\mathbb{R}^{3})$.
Moreover, by \eqref{resolvent}, and Lemmata \ref{Lemma_conj0} and
\ref{Lemma_conj}, its resolvent
\[
R_{z}^{\omega}(\varepsilon):=(-H_{\omega}(\varepsilon)-z^{2})^{-1}%
=\varepsilon^{2}U_{\varepsilon}(-H_{\varepsilon,\omega}-\varepsilon^{2}%
z^{2})^{-1}U_{\varepsilon}^{-1}%
\]
is given by
\begin{equation}
R_{z}^{\omega}(\varepsilon)=R_{z}-\varepsilon^{2}U_{\varepsilon}%
S\!L_{\varepsilon z}((\varepsilon^{-2}-1)^{-1}+D\!N_{\varepsilon\omega
}S_{\varepsilon z})^{-1}D\!N_{\varepsilon\omega}\gamma_{0}R_{\varepsilon
z}U_{\varepsilon}^{-1}\,. \label{krf0}%
\end{equation}
For successive notational convenience, let introduce%
\begin{equation}
G_{z}(\varepsilon):=\varepsilon^{1/2}U_{\varepsilon}S\!L_{\varepsilon z}\,\,.
\label{nc}%
\end{equation}
Then, using (\ref{Lambda_def}) the resolvent formula \eqref{krf0} re-writes
as
\begin{equation}
R_{z}^{\omega}(\varepsilon)=R_{z}-G_{z}(\varepsilon)\Lambda_{z}^{\omega
}(\varepsilon)G_{-\bar{z}}(\varepsilon)^{\ast}\,. \label{krf}%
\end{equation}

\begin{remark}
\label{Remark_spectrum}By \eqref{krf}, the eigenvalues and the resonances of
$H_{\omega}(\varepsilon)$\medskip\ are those $-z^{2}$ such that $z\in
\mathbb{C}$ is a pole of the map $z\mapsto\Lambda_{z}^{\omega}(\varepsilon)$.
By the results in Theorem \ref{Theorem_Schur}, this map is analytic in the
ball $\left\{  z:\left\vert z\right\vert <r_{0}\right\}  $, whenever
$\omega\not =\omega_{M}$, or in $\left\{  z:r_{1}<\left\vert z\right\vert
<r_{0}\right\}  $ whenever $\omega=\omega_{M}$, provided that $\varepsilon$ is
small enough depending on $\omega$ and $r_{1}<r_{0}$. In the
small-$\varepsilon$ regime, this shows the absence of eigenvalues/resonances
in any open bounded region of the Riemann surface if $\omega\not =\omega_{M}$,
or the absence of eigenvalues/resonances away from the origin if
$\omega=\omega_{M}$.
\end{remark}

The operator $H_{\omega}(\varepsilon)$ provides a self-adjoint realization of
the Laplacian with boundary conditions at the interface $\Gamma^{\epsilon}$
and, by exploiting its definition and taking into account the boundary
conditions appearing in (\ref{H_omega_eps_dom}), for any $f\in L^{2}%
(\mathbb{R}^{3})$ one gets the resolvent equation%
\begin{equation}
\left(  -H_{\omega}(\varepsilon)-z^{2}\right)  u=f\iff\left\{
\begin{array}
[c]{l}%
\begin{array}
[c]{ccc}%
\left(  \Delta+z^{2}\right)  u=f\,, &  & \text{in }\mathbb{R}^{3}%
\backslash\Gamma^{\varepsilon}\,,
\end{array}
\\
\\%
\begin{array}
[c]{ccc}%
\left[  \gamma_{0}\left(  \varepsilon\right)  \right]  u=0\,, &  & \left[
\gamma_{1}\left(  \varepsilon\right)  \right]  u=(\varepsilon^{-2}%
-1)D\!N_{\omega}\left(  \varepsilon\right)  \gamma_{0}\left(  \varepsilon
\right)  u\,.
\end{array}
\end{array}
\right.  \label{H_omega_eps_id}%
\end{equation}
Notice that (\ref{H_omega_eps_id}) is equivalent to
\[
H_{\omega}(\varepsilon)u=f\quad\iff\quad%
\begin{cases}
\nabla\cdot(1_{\mathbb{R}^{3}\backslash\Omega^{\varepsilon}}+\varepsilon
^{-2}1_{\Omega^{\varepsilon}})\nabla u=(1_{\mathbb{R}^{3}\backslash
\Omega^{\varepsilon}}+\varepsilon^{-2}1_{\Omega^{\varepsilon}})f\,,\\
\\
\gamma_{0}^{\mathrm{in}}(\varepsilon)u=\gamma_{0}^{\mathrm{ex}}(\varepsilon
)u\,,\qquad\varepsilon^{-2}\gamma_{1}^{\mathrm{in}}(\varepsilon)u=\gamma
_{1}^{\mathrm{ex}}(\varepsilon)u\,.
\end{cases}
\]

\subsection{\label{res-conv}Proof of Theorem \ref{Th 1}}

We consider here the norm-resolvent limits of our model operator $H_{\omega
}(\varepsilon)$ as $\varepsilon\rightarrow0_{+}$. As should be clear from the
resolvent formula \eqref{krf}, to study the resolvent convergence of
$H_{\omega}(\varepsilon)$, the behavior of $\Lambda_{z}^{\omega}(\varepsilon)$
when $\varepsilon\ll1$ is of pivotal importance. The related asymptotic
formula, provided in Theorem \ref{Theorem_Schur}, undergoes a sudden change
depending on $\omega\neq\omega_{M}$ or $\omega=\omega_{M}$. This mechanism
produces a discontinuity of the map $\omega\mapsto R_{z}^{\omega}%
(\varepsilon)$ in the limit $\varepsilon\rightarrow0_{+}$. The proof of Theorem \ref{Th 1} goes as follows. \par By Theorem \ref{Theorem_Schur},%
\begin{equation}
\Lambda_{z}^{\omega}(\varepsilon)=\varepsilon^{-2+\alpha_{\omega}%
}S_{\varepsilon\omega}^{-1}\left(  \beta_{\omega,z}P_{0}+P_{0}O(\varepsilon
)P_{0}+O(\varepsilon^{2})\right)  S_{\varepsilon\omega}D\!N_{\varepsilon
\omega}\,, \label{Lexp}%
\end{equation}
where
\begin{equation}
\alpha_{\omega}=%
\begin{cases}
1\,, & \omega\not =\omega_{M}\\
0\,, & \omega=\omega_{M}\,,
\end{cases}
\qquad\beta_{\omega,z}=%
\begin{cases}
(E_{\omega}^{0})^{-1}\,, & \omega\not =\omega_{M}\,,\\
4\pi i(c_{\Omega}\,z)^{-1}\,, & \omega=\omega_{M}\,.
\end{cases}
\label{ab}%
\end{equation}
By Lemma \ref{anS},
\[
S_{\varepsilon\omega}^{-1}=S_{0}^{-1}+O(\varepsilon)\,,
\]
and, by Lemma \ref{Lemma_S_DN} (which relies on Lemma \ref{anK}),
\[
S_{\varepsilon\omega}D\!N_{\varepsilon\omega}=\widetilde{K}_{0}+\varepsilon
^{2}\omega^{2}K_{(2)}+O\left(  \varepsilon^{3}\right)  \,,\qquad
\widetilde{K}_{0}:=Q_{0}(1/2+K_{0})Q_{0}\,.
\]
Inserting these relations into \eqref{Lexp}, one gets
\begin{align}
&  \left.  \Lambda_{z}^{\omega}(\varepsilon)=\left(  (S_{0}^{-1}%
+O(\varepsilon)\right)  \left(  \varepsilon^{\alpha_{\omega}}\omega^{2}%
\beta_{\omega,z}P_{0}K_{(2)}+O\left(  \varepsilon^{1+\alpha_{\omega}}\right)
\right)  \right. \nonumber\\
& \nonumber\\
&  \left.  =\varepsilon^{\alpha_{\omega}}\omega^{2}\beta_{\omega,z}S_{0}%
^{-1}P_{0}K_{(2)}+O\left(  \varepsilon^{1+\alpha_{\omega}}\right)  \right.
\label{Lambda_exp}%
\end{align}
By Corollary \eqref{Lemma_trace_R_eps_k_U_eps},
\begin{equation}
G_{z}(\varepsilon)=G_{z}+O(\varepsilon^{1/2})\,, \label{G_z_exp}%
\end{equation}
where
\begin{align}
&  \left.
\begin{array}
[c]{ccc}%
G_{z}:H^{-1/2}(\Gamma)\rightarrow L^{2}(\mathbb{R}^{3})\,, &  & G_{z}%
\phi:=\langle\phi\rangle\,\mathcal{G}_{z}^{y_{0}}\,,
\end{array}
\right. \label{G_z_def_1}\\
& \nonumber\\
&  \left.
\begin{array}
[c]{ccc}%
\mathcal{G}_{z}^{y_{0}}(x):=\mathcal{G}_{z}(x-y_{0})\,, &  & \langle
\phi\rangle:=\langle1,\phi\rangle_{H^{1/2}(\Gamma),H^{-1/2}(\Gamma)}%
\equiv\langle1,\phi\rangle_{H^{3/2}(\Gamma),H^{-3/2}(\Gamma)}\,.
\end{array}
\right.  \label{G_z_def_2}%
\end{align}
By $G_{z}^{\ast}u=\langle\mathcal{G}_{z}^{y_{0}},u\rangle_{L^{2}%
(\mathbb{R}^{3})}1$, one gets
\[
\mathrm{ran}(G_{z}^{\ast})=\mathbb{C}=\mathrm{ran}(P_{0})=\mathrm{ker}%
(Q_{0})\,,
\]
and from Lemma \ref{Lemma_Schur_coeff} follows%
\begin{equation}
\omega^{2}P_{0}K_{(2)}P_{0}=-\frac{\omega^{2}}{\omega_{M}^{2}}\,P_{0}\,.
\label{qt}%
\end{equation}
Hence, the resolvent formula \eqref{krf} rephrases as%
\begin{align}
&  \left.  R_{z}^{\omega}(\varepsilon)-R_{z}=-G_{z}(\varepsilon)\Lambda
_{z}^{\omega}(\varepsilon)G_{-\bar{z}}(\varepsilon)^{\ast}\right. \nonumber\\
& \nonumber\\
&  \left.  =-\left(  G_{z}+O(\varepsilon^{1/2})\right)  \left(  \varepsilon
^{\alpha_{\omega}}\omega^{2}\beta_{\omega,z}S_{0}^{-1}P_{0}K_{(2)}%
+O(\varepsilon^{1+\alpha_{\omega}})\right)  \left(  G_{-\bar{z}}^{\ast
}+O(\varepsilon^{1/2})\right)  \right. \nonumber\\
& \nonumber\\
&  \left.  =\varepsilon^{\alpha_{\omega}}\omega^{2}\beta_{\omega,z}G_{z}%
S_{0}^{-1}P_{0}K_{(2)}G_{-\bar{z}}^{\ast}+O(\varepsilon^{1/2+\alpha_{\omega}%
})\right. \nonumber\\
& \nonumber\\
&  \left.  =\left\{
\begin{array}
[c]{ll}%
O(\varepsilon)\,, & \omega\not =\omega_{M}\,,\\
& \\
\frac{4\pi i}{c_{\Omega}\,z}\,G_{z}S_{0}^{-1}P_{0}G_{-\bar{z}}^{\ast
}+O(\varepsilon^{1/2})\,, & \omega=\omega_{M}\,.
\end{array}
\right.  \right.  \label{qq}%
\end{align}
The proof of Theorem \ref{Th 1} is then concluded by \eqref{Krein_point_id}
and the relation
\[
G_{z}S_{0}^{-1}P_{0}G_{-\bar{z}}^{\ast}u=\langle S_{0}^{-1}1,1\rangle
_{H^{-1/2}(\Gamma),H^{1/2}(\Gamma)}\,\mathcal{G}_{z}^{y_{0}}\langle
\mathcal{G}_{-\bar{z}}^{y_{0}},u\rangle_{L^{2}(\mathbb{R}^{3})}=c_{\Omega
}\,\mathcal{G}_{z}^{y_{0}}\langle\mathcal{G}_{-\bar{z}}^{y_{0}},u\rangle
_{L^{2}(\mathbb{R}^{3})}\,.
\]
\hfill

\section{\label{Sec-Scatt}Generalized eigenfunctions and asymptotic scattering
solutions}

Let $\omega,\kappa>0$; by Theorem \ref{Theorem_Schur}, $\Lambda_{\kappa
}^{\omega}(\varepsilon)$ is well-defined provided that $\varepsilon$ is
sufficiently small. We next use this property and consider the stationary
scattering problem related to $H_{\omega}(\varepsilon)$. According to the
definitions of Section \ref{mod-op}, $H_{\omega}(\varepsilon)$ acts as
$\Delta$ outside $\Gamma^{\varepsilon}$, since $H_{\omega}(\varepsilon
)\subset\Delta|H_{\Delta}^{0}(\mathbb{R}^{3}\backslash\Gamma^{\varepsilon})$,
while its domain is characterized by the interface conditions%
\[%
\begin{array}
[c]{ccc}%
\left[  \gamma_{0}(\varepsilon)\right]  u=0\,, &  & \left[  \gamma
_{1}(\varepsilon)\right]  u=\left(  \varepsilon^{-2}-1\right)
D\!N_{\varepsilon\omega}\gamma_{0}u\,.
\end{array}
\]
Hence, a generalized eigenfunctions $u_{\kappa}^{\omega}\left(  \varepsilon
\right)  \in H_{\mathrm{loc}}^{2}(\mathbb{R}^{3}\backslash\Gamma^{\varepsilon
}):=H^{2}(\Omega^{\varepsilon})\oplus H_{\mathrm{loc}}^{2}(\mathbb{R}%
^{3}\backslash\overline{\Omega^{\varepsilon}})$ of $H_{\omega}(\varepsilon)$
with eigenvalue $-\kappa^{2}$ solves the problem%
\begin{equation}
\left\{
\begin{array}
[c]{l}%
\begin{array}
[c]{ccc}%
\left(  \Delta+\kappa^{2}\right)  u_{\kappa}^{\omega}\left(  \varepsilon
\right)  =0 &  & \text{in }\mathbb{R}^{3}\backslash\Gamma^{\varepsilon}\,,
\end{array}
\\
\\
\left[  \gamma_{0}(\varepsilon)\right]  u_{\kappa}^{\omega}(\varepsilon
)=0\,,\\
\\
\left[  \gamma_{1}(\varepsilon)\right]  u_{\kappa}^{\omega}(\varepsilon
)=\left(  \varepsilon^{-2}-1\right)  D\!N_{\omega}(\varepsilon)\gamma
_{0}(\varepsilon)u_{\kappa}^{\omega}\left(  \varepsilon\right)  \,.
\end{array}
\right.  \label{gen-eig}%
\end{equation}
By the next result, the generalized eigenfunctions of $H_{\omega}\left(
\varepsilon\right)  $ relate to the scattering solutions and to the functions
of the kind $G_{\kappa}(\varepsilon)\Lambda_{\kappa}^{\omega}(\varepsilon
)\phi$, with: $\phi\in H^{1/2}(\Gamma)$. Notice that, according to the mapping
properties of $\Lambda_{\kappa}^{\omega}\left(  \varepsilon\right)  $ and
$R_{\kappa}$, the functions%
\[
G_{\kappa}(\varepsilon)\Lambda_{\kappa}^{\omega}\left(  \varepsilon\right)
\phi=\varepsilon^{1/2}U_{\varepsilon}S\!L_{\varepsilon\kappa}\Lambda_{\kappa
}^{\omega}(\varepsilon)\phi=\varepsilon^{1/2}U_{\varepsilon}R_{\varepsilon
\kappa}\ast\gamma_{0}^{\ast}\Lambda_{\kappa}^{\omega}(\varepsilon)\phi\,,
\]
belong to the weighted Sobolev space $H_{-\alpha}^{2}(\mathbb{R}^{3})$,
$\alpha>1/2$.

\begin{theorem}
\label{Lemma_LAP}Let $\kappa>0$ and $u_{\kappa}^{\mathrm{in}}\in H_{-\alpha
}^{2}(\mathbb{R}^{3})$, $\alpha>1/2$, be a solution of the homogeneous
Helmholtz equation%
\begin{equation}
\left(  \Delta+\kappa^{2}\right)  u_{\kappa}^{\mathrm{in}}=0\,.
\label{Helmholtz_eq}%
\end{equation}
The scattering problem%
\begin{equation}
\left\{
\begin{array}
[c]{l}%
\begin{array}
[c]{ccc}%
\left(  \Delta+\kappa^{2}\right)  \left(  u_{\kappa}^{\mathrm{in}}%
+u_{\kappa,\omega}^{\mathrm{sc}}(\varepsilon)\right)  =0 &  & \text{in
}\mathbb{R}^{3}\backslash\Gamma^{\varepsilon}\,,
\end{array}
\\
\\
\left[  \gamma_{0}(\varepsilon)\right]  u_{\kappa,\omega}^{\mathrm{sc}}\left(
\varepsilon\right)  =0\,,\\
\\
\left[  \gamma_{1}(\varepsilon)\right]  u_{\kappa,\omega}^{\mathrm{sc}%
}(\varepsilon)=\left(  \varepsilon^{-2}-1\right)  D\!N_{\omega}(\varepsilon
)\gamma_{0}(\varepsilon)\left(  u_{\kappa}^{\mathrm{in}}+u_{\kappa,\omega
}^{\mathrm{sc}}\left(  \varepsilon\right)  \right)  \,,\\
\\
\lim_{\left\vert x\right\vert \rightarrow\infty}\left\vert x\right\vert
\left(  \frac{x}{|x|}\cdot\nabla-i\kappa\right)  u_{\kappa,\omega
}^{\mathrm{sc}}(\varepsilon)=0\,,
\end{array}
\right.  \label{scattering_eq}%
\end{equation}
admits an unique solution in $H_{-\alpha}^{2}(\mathbb{R}^{3}\backslash
\Gamma^{\varepsilon})$ given by%
\begin{equation}
u_{\kappa,\omega}^{\mathrm{sc}}(\varepsilon):=-G_{\kappa}(\varepsilon
)\Lambda_{\kappa}^{\omega}(\varepsilon)\,\gamma_{0}\left(  u_{\kappa
}^{\mathrm{in}}\circ\Phi_{\varepsilon}\right)  \,. \label{scattering_sol}%
\end{equation}

\end{theorem}

\begin{proof}
We proceed in two steps: at first we consider a dilated problem with interface
conditions assigned on $\Gamma$ and prove the result in this setting. Then, we
discuss (\ref{scattering_eq}) by using the dilation mapping.

Let $\psi_{\kappa}^{\mathrm{in}}\in H_{-\alpha}^{2}(\mathbb{R}^{3})$ be a
solution of the Helmholtz equation (\ref{Helmholtz_eq}) and consider the
dilated scattering problem%
\begin{equation}
\left\{
\begin{array}
[c]{l}%
\begin{array}
[c]{ccc}%
\left(  \Delta+\kappa^{2}\right)  \left(  \psi_{\kappa}^{\mathrm{in}}%
+\psi_{\kappa,\omega}^{\mathrm{sc}}(\varepsilon)\right)  =0\,, &  & \text{in
}\mathbb{R}^{3}\backslash\Gamma\,,
\end{array}
\\
\\
\left[  \gamma_{0}\right]  \psi_{\kappa,\omega}^{\mathrm{sc}}\left(
\varepsilon\right)  =0\,,\\
\\
\left[  \gamma_{1}\right]  \psi_{\kappa,\omega}^{\mathrm{sc}}\left(
\varepsilon\right)  =\left(  \varepsilon^{-2}-1\right)  D\!N_{\varepsilon
\omega}\gamma_{0}\left(  \psi_{\kappa}^{\mathrm{in}}+\psi_{\kappa,\omega
}^{\mathrm{sc}}(\varepsilon)\right)  \,,\\
\\
\lim_{\left\vert x\right\vert \rightarrow\infty}\left\vert x\right\vert
\left(  \frac{x}{|x|}\cdot\nabla-i\kappa\right)  \psi_{\kappa,\omega
}^{\mathrm{sc}}(\varepsilon)=0\,,
\end{array}
\right.  \label{scattering_eq_dil}%
\end{equation}
Notice that, due to (\ref{conj_id}), $D\!N_{\omega}(\varepsilon)$ is here
replaced by $D\!N_{\varepsilon\omega}$. We next proceed as in \cite[Lemmata
5.1 and 5.3]{MaPoSi LAP}, where a similar problem involving abstract boundary
conditions were discussed. Let us look for a solution of the form:
$\psi_{\kappa,\omega}^{\mathrm{sc}}(\varepsilon):=-\varepsilon^{-1}%
S\!L_{\kappa}\Lambda_{\kappa/\varepsilon}^{\omega}\left(  \varepsilon\right)
\gamma_{0}\psi_{\kappa}^{\mathrm{in}}$. By \cite[Lemma 5.3]{MaPoSi LAP},
$\psi_{\kappa,\omega}^{\mathrm{sc}}\left(  \varepsilon\right)  $ satisfies the
Sommerfeld radiation condition in (\ref{scattering_eq_dil}). Since the
distributions $\gamma_{0}^{\ast}\phi$, $\phi\in H^{1/2}(\Gamma)$, are
supported on $\Gamma$, from
\[
\left(  \Delta+k^{2}\right)  S\!L_{\kappa}\phi=\left(  \Delta+k^{2}\right)
R_{\kappa}\gamma_{0}^{\ast}\phi=-\gamma_{0}^{\ast}\phi\,,
\]
and (\ref{Helmholtz_eq}) there follows%
\[%
\begin{array}
[c]{ccc}%
\left(  \Delta+k^{2}\right)  \left(  \psi_{\kappa}^{\mathrm{in}}+\psi
_{\kappa,\omega}^{\mathrm{sc}}(\varepsilon)\right)  =0\,, &  & \text{in
}\mathbb{R}^{3}\backslash\Gamma\,.
\end{array}
\]
The boundary conditions in (\ref{scattering_eq_dil}) follows by proceeding
along the same lines as in the proof of Theorem \ref{heo} (see the
calculations there involving the function $\widetilde{u}$); by (\ref{jump})
and $\psi_{\kappa}^{\mathrm{in}}\in H_{-\alpha}^{2}(\mathbb{R}^{3})$ it
results $\left[  \gamma_{0}\right]  \psi_{\kappa,\omega}^{\mathrm{sc}}\left(
\varepsilon\right)  =0$ and%
\[
\left[  \gamma_{1}\right]  \psi_{\kappa,\omega}^{\mathrm{sc}}\left(
\varepsilon\right)  =-\varepsilon^{-1}\left[  \gamma_{1}\right]  SL_{\kappa
}\Lambda_{\kappa/\varepsilon}^{\omega}(\varepsilon)\gamma_{0}\psi_{\kappa
}^{\mathrm{in}}=\varepsilon^{-1}\Lambda_{\kappa/\varepsilon}^{\omega
}(\varepsilon)\gamma_{0}\psi_{\kappa}^{\mathrm{in}}\,.
\]
Furthermore, from%
\begin{align*}
&  \left.  \left(  \varepsilon^{-2}-1\right)  D\!N_{\varepsilon\omega}%
\gamma_{0}\left(  \psi_{\kappa}^{\mathrm{in}}+\psi_{\kappa,\omega
}^{\mathrm{sc}}(\varepsilon)\right)  =\left(  \varepsilon^{-2}-1\right)
D\!N_{\varepsilon\omega}\gamma_{0}\left(  \psi_{\kappa}^{\mathrm{in}%
}-\varepsilon^{-1}S\!L_{\kappa}\Lambda_{\kappa/\varepsilon}^{\omega}\left(
\varepsilon\right)  \gamma_{0}\psi_{\kappa}^{\mathrm{in}}\right)  \right. \\
& \\
&  \left.  =\left(  \varepsilon^{-2}-1\right)  D\!N_{\varepsilon\omega}%
\gamma_{0}\psi_{\kappa}^{\mathrm{in}}-\varepsilon^{-1}\left(  \varepsilon
^{-2}-1\right)  D\!N_{\varepsilon\omega}\gamma_{0}S\!L_{\kappa}\Lambda
_{\kappa/\varepsilon}^{\omega}(\varepsilon)\gamma_{0}\psi_{\kappa
}^{\mathrm{in}}\right. \\
& \\
&  \left.  =\left(  \varepsilon^{-2}-1\right)  D\!N_{\varepsilon\omega}%
\gamma_{0}\psi_{\kappa}^{\mathrm{in}}-\varepsilon^{-1}\left(  1+\left(
\varepsilon^{-2}-1\right)  D\!N_{\varepsilon\omega}S_{\kappa}\right)
\Lambda_{\kappa/\varepsilon}^{\omega}(\varepsilon)\gamma_{0}\psi_{\kappa
}^{\mathrm{in}}+\varepsilon^{-1}\Lambda_{\kappa/\varepsilon}^{\omega
}(\varepsilon)\gamma_{0}\psi_{\kappa}^{\mathrm{in}}\right. \\
& \\
&  \left.  =\left(  \varepsilon^{-2}-1\right)  D\!N_{\varepsilon\omega}%
\gamma_{0}\psi_{\kappa}^{\mathrm{in}}-\left(  \varepsilon^{-2}-1\right)
D\!N_{\varepsilon\omega}\gamma_{0}\psi_{\kappa}^{\mathrm{in}}+\varepsilon
^{-1}\Lambda_{\kappa/\varepsilon}^{\omega}(\varepsilon)\gamma_{0}\psi_{\kappa
}^{\mathrm{in}}=\varepsilon^{-1}\Lambda_{\kappa/\varepsilon}^{\omega
}(\varepsilon)\gamma_{0}\psi_{\kappa}^{\mathrm{in}}\,,\right.
\end{align*}
there follows%
\[
\left[  \gamma_{1}\right]  \psi_{\kappa,\omega}^{\mathrm{sc}}\left(
\varepsilon\right)  =\left(  \varepsilon^{-2}-1\right)  D\!N_{\varepsilon
\omega}\gamma_{0}\left(  \psi_{\kappa}^{\mathrm{in}}+\psi_{\kappa,\omega
}^{\mathrm{sc}}(\varepsilon)\right)  \,.
\]
Hence, $\psi_{\kappa,\omega}^{\mathrm{sc}}(\varepsilon)$ solves the dilated
scattering problem (\ref{scattering_eq_dil}). To conclude this part of the
proof, we need to show that such a solution is unique. Let us assume that
$v_{\kappa}$ solve the same scattering problem; then, the difference
$w_{\kappa}:=\psi_{\kappa,\omega}^{\mathrm{sc}}(\varepsilon)-v_{\kappa}$
solves the exterior Helmholtz equation%
\begin{equation}
(\Delta_{\mathbb{R}^{3}\backslash\overline{\Omega}}+\kappa^{2})w_{\kappa}=0
\label{Helmholtz_eq_ext}%
\end{equation}
and satisfies the radiation condition%
\begin{equation}
\lim_{\left\vert x\right\vert \rightarrow\infty}\left\vert x\right\vert
\left(  \frac{x}{|x|}\cdot\nabla-i\kappa\right)  w_{\kappa}=0\,. \label{RC}%
\end{equation}
Let $R>0$ such that $\Omega\subset B_{R}=\{x\in\mathbb{R}^{3}:|x|<R\}$; by
\cite[eq. (9.19)]{McLe} there follows%
\[
\lim_{R\rightarrow\infty}\int_{|x|=R}|w_{\kappa}(x)|^{2}\,d\sigma(x)=0\,,
\]
and, by Rellich's Lemma (see, e.g., \cite[Lemma 9.8]{McLe}), this entails
$w_{\kappa}=0$ in $\mathbb{R}^{3}\backslash\overline{B_{R}}$. The Green
identity in $B_{R}\backslash\overline{\Omega}$ then yields
\[
\int_{B_{R}\backslash\overline{\Omega}}|\nabla w_{\kappa}(x)|^{2}%
dx-\int_{\Gamma^{\varepsilon}}\overline{\gamma_{0}w_{\kappa}}(x)\,\gamma
_{1}w_{\kappa}(x)\,d\sigma(x)=\kappa^{2}\,\int_{B_{R}\backslash\overline
{\Omega}}|w_{\kappa}(x)|^{2}\,dx
\]
and hence%
\[
\operatorname{Im}\int_{\Gamma}\overline{\gamma_{0}w_{\kappa}}(x)\,\gamma
_{1}w_{\kappa}(x)\,d\sigma(x)=0\,.
\]
By \cite[Lemma 9.9]{McLe}, this gives $w_{\kappa}(x)=0$ whenever
$x\in\mathbb{R}^{3}\backslash\overline{\Omega}$. Since $[\gamma_{0}]w_{\kappa
}=0$, the boundary conditions in (\ref{scattering_eq_dil}) imply $\gamma
_{0}^{\mathrm{in}}w_{\kappa}=\gamma_{1}^{\mathrm{in}}w_{\kappa}=0$ and so
$w_{\kappa}$ solves the interior Helmholtz equation with both zero Dirichlet
and Neumann boundary conditions. This implies $w_{\kappa}(x)=0$ whenever
$x\in\Omega$.

Let us next consider (\ref{scattering_eq}); setting $u_{\kappa,\omega}\left(
\varepsilon\right)  :=u_{\kappa}^{\mathrm{in}}+u_{\kappa,\omega}^{\mathrm{sc}%
}(\varepsilon)$ and using the identities (\ref{trace_conj}),
(\ref{trace_prime_conj}), (\ref{conj_id}) and: $U_{\varepsilon}^{-1}\Delta
U_{\varepsilon}=\varepsilon^{-2}\Delta$, we get
\begin{equation}
\left\{
\begin{array}
[c]{l}%
\begin{array}
[c]{ccc}%
\varepsilon^{-2}\left(  -\Delta-\varepsilon^{2}k^{2}\right)  U_{\varepsilon
}^{-1}u_{\kappa,\omega}(\varepsilon)=0\,, &  & \text{in }\mathbb{R}%
^{3}\backslash\Gamma^{\varepsilon}\,,
\end{array}
\\
\\
\left[  \gamma_{0}(\varepsilon)\right]  U_{\varepsilon}^{-1}u_{\kappa,\omega
}(\varepsilon)=0\,,\\
\\
\left[  \gamma_{1}(\varepsilon)\right]  U_{\varepsilon}^{-1}u_{\kappa,\omega
}(\varepsilon)=\left(  \varepsilon^{-2}-1\right)  D\!N_{\varepsilon\omega
}\gamma_{0}U_{\varepsilon}^{-1}u_{\kappa,\omega}\left(  \varepsilon\right)
\,.
\end{array}
\right.  \label{tpr}%
\end{equation}
Then, the function%
\begin{equation}
\psi_{\varepsilon\kappa,\omega}(\varepsilon):=\varepsilon^{-2}U_{\varepsilon
}^{-1}u_{\kappa,\omega}(\varepsilon)\,, \label{scattering_sol_scaled}%
\end{equation}
solves the problem%
\[
\left\{
\begin{array}
[c]{l}%
\begin{array}
[c]{ccc}%
\left(  -\Delta-\varepsilon^{2}k^{2}\right)  \psi_{\varepsilon\kappa,\omega
}\left(  \varepsilon\right)  =0\,, &  & \text{in }\mathbb{R}^{3}%
\backslash\Gamma\,,
\end{array}
\\
\\
\left[  \gamma_{0}(\varepsilon)\right]  \psi_{\varepsilon\kappa,\omega}\left(
\varepsilon\right)  =0\,,\\
\\
\left[  \gamma_{1}(\varepsilon)\right]  \psi_{\varepsilon\kappa,\omega}\left(
\varepsilon\right)  =\left(  \varepsilon^{-2}-1\right)  D\!N_{\varepsilon
\omega}\gamma_{0}\psi_{\varepsilon\kappa,\omega}(\varepsilon)\,.
\end{array}
\right.
\]
Let $\psi_{\varepsilon\kappa,\omega}^{\mathrm{sc}}(\varepsilon):=\varepsilon
^{-2}U_{\varepsilon}^{-1}u_{\kappa,\omega}^{\mathrm{sc}}\left(  \varepsilon
\right)  $; since $u_{\kappa,\omega}^{\mathrm{sc}}\left(  \varepsilon\right)
$ fulfills the radiation conditions in (\ref{scattering_eq}), it follows%
\begin{align*}
&  \left.  \lim_{\left\vert x\right\vert \rightarrow\infty}\left(  \hat
{x}\cdot\nabla\mp i\varepsilon\kappa\right)  \psi_{\varepsilon\kappa,\omega
}^{\mathrm{sc}}(\varepsilon)=\lim_{\left\vert x\right\vert \rightarrow\infty
}\left(  \hat{x}\cdot\nabla\mp i\varepsilon\kappa\right)  \varepsilon
^{-2}\left(  U_{\varepsilon}^{-1}u_{\kappa,\omega}^{\mathrm{sc}}%
(\varepsilon)\right)  \left(  x\right)  \right. \\
& \\
&  \left.  =\varepsilon^{-1/2}\lim_{\left\vert x\right\vert \rightarrow\infty
}\left(  \hat{x}\cdot\nabla\mp i\varepsilon\kappa\right)  u_{\kappa,\omega
}^{\mathrm{sc}}(\varepsilon)\left(  y_{0}+\varepsilon\,\left(  x-y_{0}\right)
\right)  \right. \\
& \\
&  \left.  =\varepsilon^{1/2}\lim_{\left\vert x\right\vert \rightarrow\infty
}\left(  \hat{x}\cdot\left(  \nabla u_{\kappa,\omega}^{\mathrm{sc}}\left(
\varepsilon\right)  \right)  \left(  y_{0}+\varepsilon\,\left(  x-y_{0}%
\right)  \right)  \mp i\kappa u_{\kappa,\omega}^{\mathrm{sc}}\left(
\varepsilon\right)  \left(  y_{0}+\varepsilon\,\left(  x-y_{0}\right)
\right)  \right)  =0\,.\right.
\end{align*}
Hence, a scaled radiation condition holds for $\psi_{\varepsilon\kappa,\omega
}^{\mathrm{sc}}(\varepsilon)$. Moreover, from (\ref{Helmholtz_eq}) follows%
\[
\left(  -\Delta-\varepsilon^{2}\kappa^{2}\right)  \varepsilon^{-2}%
U_{\varepsilon}^{-1}u_{\kappa}^{\mathrm{in}}=U_{\varepsilon}^{-1}\left(
-\Delta-\kappa^{2}\right)  u_{\kappa}^{\mathrm{in}}=0\,.
\]
Then, $\psi_{\varepsilon\kappa}^{\mathrm{in}}:=\varepsilon^{-2}U_{\varepsilon
}^{-1}u_{\kappa}^{\mathrm{in}}$ is a solution of the Helmholtz equation at
energy $\varepsilon^{2}\kappa^{2}$. Therefore, the total field $\psi
_{\varepsilon\kappa,\omega}(\varepsilon)=\psi_{\varepsilon\kappa}%
^{\mathrm{in}}+\psi_{\varepsilon\kappa,\omega}^{\mathrm{sc}}(\varepsilon)$
solves the dilated scattering problem (\ref{scattering_eq_dil}) at energy
$\varepsilon^{2}\kappa^{2}$, whose unique solution, by the first part of the
proof, writes as%
\[
\psi_{\varepsilon\kappa,\omega}^{\mathrm{sc}}(\varepsilon)=-\varepsilon
^{-1}S\!L_{\varepsilon\kappa}\Lambda_{\kappa}^{\omega}\left(  \varepsilon
\right)  \gamma_{0}\psi_{\varepsilon\kappa}^{\mathrm{in}}\,.
\]
From (\ref{scattering_sol_scaled}) there follows%
\begin{align*}
&  \left.  u_{\kappa,\omega}(\varepsilon)=\varepsilon^{2}U_{\varepsilon}%
\psi_{\varepsilon\kappa,\omega}(\varepsilon)=\varepsilon^{2}U_{\varepsilon
}\left(  \psi_{\varepsilon\kappa}^{\mathrm{in}}-\varepsilon^{-1}%
S\!L_{\varepsilon\kappa}\Lambda_{\kappa}^{\omega}\left(  \varepsilon\right)
\gamma_{0}\psi_{\varepsilon\kappa}^{\mathrm{in}}\right)  \right. \\
& \\
&  \left.  =\varepsilon^{2}U_{\varepsilon}\left(  \varepsilon^{-2}%
U_{\varepsilon}^{-1}u_{\varepsilon\kappa}^{\mathrm{in}}-\varepsilon
^{-1}S\!L_{\varepsilon\kappa}\Lambda_{\kappa}^{\omega}\left(  \varepsilon
\right)  \gamma_{0}\varepsilon^{-2}U_{\varepsilon}^{-1}u_{\kappa}%
^{\mathrm{in}}\right)  \right. \\
& \\
&  \left.  =u_{\varepsilon\kappa}^{\mathrm{in}}-\varepsilon^{-1}%
U_{\varepsilon}S\!L_{\varepsilon\kappa}\Lambda_{\kappa}^{\omega}\left(
\varepsilon\right)  \gamma_{0}U_{\varepsilon}^{-1}u_{\kappa}^{\mathrm{in}%
}\,.\right.
\end{align*}
Using the definition (\ref{nc}), this leads us to%
\[
u_{\kappa,\omega}^{\mathrm{sc}}(\varepsilon)=-\varepsilon^{-3/2}G_{\kappa
}(\varepsilon)\Lambda_{\kappa}^{\omega}(\varepsilon)\gamma_{0}U_{\varepsilon
}^{-1}u_{\kappa}^{\mathrm{in}}\,,
\]
and from%
\[
\gamma_{0}U_{\varepsilon}^{-1}u_{\kappa}^{\mathrm{in}}=\varepsilon
^{3/2}\,\gamma_{0}\left(  u_{\kappa}^{\mathrm{in}}\circ\Phi_{\varepsilon
}\right)  \,,
\]
the representation (\ref{scattering_sol}) follows. \hfill
\end{proof}

\begin{remark}
\label{acc} According to \eqref{gen-eig}, the solution $u_{\kappa
}^{\mathrm{in}}+u_{\kappa,\omega}^{\mathrm{sc}}(\varepsilon)$ in
\eqref{scattering_eq} can be equivalently defined as the unique generalized
eigenfunction of $H_{\omega}(\varepsilon)$ with eigenvalue $-\kappa^{2}$ such
that $u_{\kappa,\omega}^{\mathrm{sc}}\left(  \varepsilon\right)  $ satisfies
the (outgoing) Sommerfeld radiation condition.
\end{remark}

According to the above Remark, the next result is the analogous of Theorem
\ref{Th 1} as regards the behavior of generalized eigenfunctions of
$H_{\omega}(\varepsilon)$ whenever $\varepsilon\ll1$.

\begin{theorem}
\label{conv-gen}For any $\kappa>0$ and for any $\varepsilon>0$ sufficiently
small, let $u_{\kappa}^{\omega}( \varepsilon) :=u_{\kappa}^{\mathrm{in}%
}+u_{\kappa,\omega}^{\mathrm{sc}}( \varepsilon) $ be as in Theorem
\ref{Lemma_LAP}. Then, one has
\begin{align}
&  \left.  \omega\not =\omega_{M}\quad\Longrightarrow\quad\left\Vert
u_{\kappa}^{\omega}( \varepsilon) -u_{\kappa}^{\omega}\right\Vert
_{L_{-\alpha}^{2}(\mathbb{R}^{3})}\leq c\,\varepsilon^{3/2}\,,\right.
\label{scattering_exp_1}\\
& \nonumber\\
&  \left.  \omega=\omega_{M}\quad\Longrightarrow\quad\left\Vert u_{\kappa
}^{\omega}( \varepsilon) -\widehat{u}_{\kappa}\right\Vert _{L_{-\alpha}%
^{2}(\mathbb{R}^{3})}\leq c\,\varepsilon^{1/2}\,,\right.
\end{align}
where $\alpha>1/2$,
\begin{align}
&  \left.  u_{\kappa}^{\omega}(x):=u_{\kappa}^{\mathrm{in}}(x)+\varepsilon
\,\frac{c_{\Omega}\,\omega^{2}}{\omega_{M}^{2}-\omega^{2}}\ u_{\kappa
}^{\mathrm{in}}(y_{0})\,\mathcal{G}_{\kappa}(x-y_{0})\,,\right. \\
& \nonumber\\
&  \left.  \widehat{u}_{\kappa}(x):=u_{\kappa}^{\mathrm{in}}(x)+4\pi\,\frac
{i}{\kappa}\ u_{\kappa}^{\mathrm{in}}(y_{0})\,\mathcal{G}_{\kappa}%
(x-y_{0})\,,\right.  \label{scattering_exp_4}%
\end{align}
and the estimates hold uniformly with respect to the choice of $u_{\kappa
}^{\mathrm{in}}$ in any bounded subset of $H_{-\alpha}^{2}(\mathbb{R}^{3})$.
\end{theorem}

\begin{proof}
By \eqref{sts} applied to $\psi_{\varepsilon}=u_{\kappa}^{\mathrm{in}}%
\circ\Phi_{\varepsilon}$ one gets%
\[
\Vert u_{\kappa}^{\mathrm{in}}\circ\Phi_{\varepsilon}-u_{\kappa}^{\mathrm{in}%
}(y_{0})\Vert_{H^{2}(B)}^{2}\leq c_{\alpha,B}\,\varepsilon^{1/2}\,\Vert
u_{\kappa}^{\mathrm{in}}\Vert_{H_{-\alpha}^{2}(\mathbb{R}^{3})}^{2}\,,
\]
where $\alpha>1/2$ and $B\subset\mathbb{R}^{3}$ is any star-shaped bounded
open set. Since $\gamma_{0}\in\mathcal{L}(H^{2}(B),H^{3/2}(\Gamma))$ and
$u_{\kappa}^{\mathrm{in}}\in H_{-\alpha}^{2}(\mathbb{R}^{3})$, it follows%
\begin{equation}
\Vert\gamma_{0}\left(  u_{\kappa}^{\mathrm{in}}\circ\Phi_{\varepsilon
}-u_{\kappa}^{\mathrm{in}}(y_{0})\right)  \Vert_{H^{1/2}(\Gamma)}\leq
c\,\varepsilon^{1/2}\Vert u_{\kappa}^{\mathrm{in}}\Vert_{H_{-\alpha}%
^{2}(\mathbb{R}^{3})}^{2}\,, \label{tut}%
\end{equation}
and hence%
\[
\gamma_{0}\left(  u_{\kappa}^{\mathrm{in}}\circ\Phi_{\varepsilon}\right)
=u_{\kappa}^{\mathrm{in}}(y_{0})+O(\varepsilon^{1/2})\,.
\]
We next proceed along the same lines as in the proof of Theorem \ref{Th 1}
(see Section \ref{res-conv}). By (\ref{ab}), (\ref{Lambda_exp}), (\ref{tut})
and (\ref{E_0}), it results%
\begin{align*}
&  \left.  \Lambda_{\kappa}^{\omega}(\varepsilon)\gamma_{0}\left(  u_{\kappa
}^{\mathrm{in}}\circ\Phi_{\varepsilon}\right)  =\left(  \varepsilon
^{\alpha_{\omega}}\omega^{2}\beta_{\omega,\kappa}S_{0}^{-1}P_{0}%
K_{(2)}+O\left(  \varepsilon^{1+\alpha_{\omega}}\right)  \right)  \left(
u_{\kappa}^{\mathrm{in}}(y_{0})+O(\varepsilon^{1/2})\right)  \right. \\
& \\
&  \left.  =\left\{
\begin{array}
[c]{ll}%
\left(  \varepsilon\frac{\omega^{2}\omega_{M}^{2}}{\omega^{2}-\omega^{2}}%
S_{0}^{-1}P_{0}K_{(2)}+O(\varepsilon^{2})\right)  \left(  u_{\kappa
}^{\mathrm{in}}(y_{0})+O(\varepsilon^{1/2})\right)  \,, & \omega\not =%
\omega_{M}\,,\\
& \\
\left(  \frac{\omega_{M}^{2}4\pi i}{c_{\Omega}\,\kappa}S_{0}^{-1}P_{0}%
K_{(2)}+O(\varepsilon)\right)  \left(  u_{\kappa}^{\mathrm{in}}(y_{0}%
)+O(\varepsilon^{1/2})\right)  \,, & \omega=\omega_{M}\,,
\end{array}
\right.  \right.
\end{align*}
and, by (\ref{qt}), we get%
\begin{equation}
\left.  \Lambda_{\kappa}^{\omega}(\varepsilon)\gamma_{0}\left(  u_{\kappa
}^{\mathrm{in}}\circ\Phi_{\varepsilon}\right)  =\left\{
\begin{array}
[c]{ll}%
-\varepsilon\frac{\omega^{2}}{\omega_{M}^{2}-\omega^{2}}\,u_{\kappa
}^{\mathrm{in}}(y_{0})S_{0}^{-1}\,(1)+O(\varepsilon^{3/2})\,, & \omega
\not =\omega_{M}\,,\\
& \\
-\frac{4\pi i}{c_{\Omega}\,\kappa}\,u_{\kappa}^{\mathrm{in}}(y_{0})S_{0}%
^{-1}(1)+O(\varepsilon^{1/2})\,, & \omega=\omega_{M}\,.
\end{array}
\right.  \right.  \label{Lambda_exp_id}%
\end{equation}
The expansion (\ref{G_z_exp}) implies%
\begin{equation}
G_{\kappa}(\varepsilon)S_{0}^{-1}\,(1)=G_{\kappa}S_{0}^{-1}\,(1)+O(\varepsilon
^{1/2})=\langle1,S_{0}^{-1}\,(1)\rangle_{H^{-1/2}(\Gamma),H^{1/2}(\Gamma
)}\,\mathcal{G}_{\kappa}^{y_{0}}+O(\varepsilon^{1/2})=c_{\Omega}%
\,\mathcal{G}_{\kappa}^{y_{0}}+O(\varepsilon^{1/2})\,. \label{G_k_id}%
\end{equation}
Finally, combining (\ref{scattering_sol}) with (\ref{Lambda_exp_id}) and
(\ref{G_k_id}), one gets (\ref{scattering_exp_1})-(\ref{scattering_exp_4}).
\hfill
\end{proof}

\begin{remark}
Suppose $u_{\kappa}^{\mathrm{in}}(x)$ is a plane wave with direction
$\hat{\theta}$ and frequency $\kappa$, i.e.: $u_{\kappa}^{\mathrm{in}%
}(x)=e^{i\omega\hat{\theta}\cdot x}$ . Then, in consistency with Theorem
\ref{Th 1}, $\widehat{u}_{\kappa}$ in \eqref{scattering_exp_4} is a
generalized eigenfunction with eigenvalue $-\kappa^{2}$ of the self-adjoint
operator $\Delta_{y_{0}}$ defined in \eqref{wd1} and \eqref{wd2} (see
\cite[equation 1.4.11]{Albeverio}, there $\alpha=0$ and $y=y_{0}$).
\end{remark}

\subsection{\label{Sec_Asymptotic}Proof of Theorem \ref{Th 2}}

We are now in the position to prove our results for the acoustic scattering
problem. As pointed out in Section \ref{Sec_Model}, the acoustic equation
\begin{equation}%
\begin{cases}
\left(  \nabla\cdot(1_{\mathbb{R}^{3}\backslash\Omega^{\varepsilon}%
}+\varepsilon^{-2}1_{\Omega^{\varepsilon}})\nabla+\omega^{2}\left(
1_{\mathbb{R}^{3}\backslash\Omega^{\varepsilon}}+\varepsilon^{-2}%
1_{\Omega^{\varepsilon}}\right)  \right)  u_{\omega}\left(  \varepsilon
\right)  =0\,,\\
\\
\gamma_{0}^{\mathrm{in}}(\varepsilon)u_{\omega}\left(  \varepsilon\right)
=\gamma_{0}^{\mathrm{ex}}(\varepsilon)u_{\omega}(\varepsilon)\,,\qquad
\varepsilon^{-2}\gamma_{1}^{\mathrm{in}}(\varepsilon)u_{\omega}\left(
\varepsilon\right)  =\gamma_{1}^{\mathrm{ex}}(\varepsilon)u_{\omega
}(\varepsilon)\,,
\end{cases}
\label{bvp}%
\end{equation}
is equivalent to the generalized eigenvalue problem for the operator
$H_{\omega}(\varepsilon)$ at energy $\omega^{2}$. By Theorem \ref{Lemma_LAP},
the corresponding scattering problem is well posed and the diffusion of an
incident wave $u_{\omega}^{\mathrm{in}}$ with frequency $\omega>0$ is
described by the outgoing radiating solution (\ref{scattering_sol}) and allows
the asymptotic expansions provided in Theorem \ref{conv-gen}. Hence, for
$u_{\omega}^{\mathrm{in}}\in H_{-\alpha}^{2}(\mathbb{R}^{3})$, $\alpha>1/2$, a
solution of the homogeneous Helmholtz equation%
\begin{equation}
\left(  \Delta+\omega^{2}\right)  u_{\omega}^{\mathrm{in}}=0\,,
\label{Helmholtz_eq_omega}%
\end{equation}
the boundary value problem (\ref{bvp}) admits an unique solution $u_{\omega
}(\varepsilon)\in H_{-\alpha}^{2}(\mathbb{R}^{3}\backslash\Gamma^{\varepsilon
})$, $\alpha>1/2$, such that $u_{\omega}^{\mathrm{sc}}(\varepsilon
):=u_{\omega}(\varepsilon)-u_{\omega}^{\mathrm{in}}$ satisfies the outgoing
Sommerfeld radiation condition. The scattered field represents as%
\begin{equation}
u_{\omega}^{\mathrm{sc}}(\varepsilon):=-G_{\omega}(\varepsilon)\Lambda
_{\omega}^{\omega}(\varepsilon)\,\gamma_{0}\left(  u_{\omega}^{\mathrm{in}%
}\circ\Phi_{\varepsilon}\right)  \,, \label{ac_scattering_sol}%
\end{equation}
i.e., $u_{\omega}(\varepsilon)$ is a generalized eigenfunction of $H_{\omega
}(\varepsilon)$ with eigenvalue $-\omega^{2}$. Moreover, for any
$\varepsilon>0$ sufficiently small, the expansions (\ref{non_res_scattering}%
)-(\ref{res_scattering}) follow from (\ref{scattering_exp_1}%
)-(\ref{scattering_exp_4}).\hfill

\subsection{Quasi-resonant asymptotic scattering solutions}

The estimates in the expansions provided in the Theorem \ref{Th 2} (see also
Theorem \ref{conv-gen}) are frequency-dependent and so they are useless as
regards an accurate descriptions of the transitions between the two different
asymptotic scattering regimes as the frequency $\omega$ approaches the
Minnaert one $\omega_{M}$. In this section we provide more refined estimates
which are uniform with respect to the frequency $\omega$. Their proof relies
on $\omega$-uniform estimates on the $\varepsilon$-expansion of the operator
$\Lambda_{\omega}^{\omega}(\varepsilon)$ which we provide at first.

\begin{theorem}
\label{teo-unif}Let $c_{M}>0$, $\mathcal{I}_{M}\subset\mathbb{R}_{+}$ be a
bounded interval containing $\omega_{M}$ and $E_{\omega}^{0}$, $E_{\omega}%
^{1}$ be given by \eqref{E_0}, \eqref{E_1}. For $\varepsilon>0$ sufficiently
small the expansion%
\begin{equation}
\Lambda_{\omega}^{\omega}(\varepsilon)=\frac{1}{\varepsilon}\,S_{\varepsilon
\omega}^{-1}\left(  \frac{1}{E_{\omega}^{0}}\left(  \frac{E_{\omega}^{0}%
}{E_{\omega}^{0}+E_{\omega}^{1}\,\varepsilon}\,P_{0}+P_{0}O(\varepsilon
)P_{0}\right)  +O(\varepsilon^{2})\right)  S_{\varepsilon\omega}%
D\!N_{\varepsilon\omega}\,, \label{ool}%
\end{equation}
holds uniformly w.r.t. $\omega$ in $\left\{  \omega\in\mathcal{I}%
_{M}:\left\vert \omega-\omega_{M}\right\vert \geq c_{M}\,\varepsilon\right\}
$, i.e.,%
\begin{equation}
\sup_{\omega\in\mathcal{I}_{M}:\left\vert \omega-\omega_{M}\right\vert \geq
c_{M}\,\varepsilon}\Vert O(\varepsilon^{j})\Vert_{H^{1/2}(\Gamma
),H^{1/2}(\Gamma)}\leq C_{M}\varepsilon^{j}\,,
\end{equation}
with $C_{M}$ depending only on $c_{M}$.
\end{theorem}

\begin{proof}
From \eqref{Lambda_id} follows
\[
\Lambda_{\omega}^{\omega}(\varepsilon)=\varepsilon(1-\varepsilon
^{2})S_{\varepsilon\omega}^{-1}\left(  \varepsilon^{2}+\left(  1-\varepsilon
^{2}\right)  \left(  \frac{1}{2}+K_{\varepsilon\omega}\right)  \right)
^{\!-1}S_{\varepsilon\omega}D\!N_{\varepsilon\omega}\,.
\]
Thus, by \eqref{K_exp}, $\mathbb{M}(\varepsilon)$ in \eqref{K_op} has the
following components:
\begin{align*}
&  M_{00}(\varepsilon)=P_{0}\left(  E_{\omega}^{0}\varepsilon^{2}+E_{\omega
}^{1}\varepsilon^{3}+\varepsilon^{2}O((\varepsilon\omega)^{2})\right)
P_{0}\,,\\
& \\
&  M_{01}(\varepsilon)=P_{0}{O}((\varepsilon\omega)^{2})Q_{0}\,,\qquad
M_{10}(\varepsilon)=Q_{0}{O}((\varepsilon\omega)^{2})P_{0}\,,\\
& \\
&  M_{11}(\varepsilon)=Q_{0}\left(  1/2+K_{0}+\varepsilon^{2}O((\varepsilon
\omega)^{0})\right)  Q_{0}\,.
\end{align*}
The requirement $|\omega-\omega_{M}|\geq c_{M}\,\varepsilon$ is equivalent to
$\varepsilon/E_{\omega}^{0}=O_{u}(1)$, where $O_{u}\left(  \varepsilon
^{j}\right)  $ means that the corresponding estimate is uniform with respect
to $\omega$. It follows%
\begin{align*}
M_{00}(\varepsilon)=  &  E_{\omega}^{0}\varepsilon^{2}P_{0}\left(
1+E_{\omega}^{1}\,\varepsilon/E_{\omega}^{0}+O_{u}(\varepsilon)\varepsilon
/E_{\omega}^{0}\right)  P_{0}\\
=  &  E_{\omega}^{0}\varepsilon^{2}P_{0}\left(  1+E_{\omega}^{1}%
\,\varepsilon/E_{\omega}^{0}+O_{u}(\varepsilon)\right)  P_{0}\,,\\
& \\
M_{01}(\varepsilon)=  &  P_{0}{O_{u}}(\varepsilon^{2})Q_{0}\,,\qquad
M_{10}(\varepsilon)=Q_{0}{O_{u}}(\varepsilon^{2})P_{0}\,,\\
& \\
M_{11}(\varepsilon)=  &  Q_{0}\left(  1/2+K_{0}+O_{u}(\varepsilon^{2})\right)
Q_{0}\,.
\end{align*}
Then,
\[
C_{00}(\varepsilon)=E_{\omega}^{0}\varepsilon^{2}P_{0}\left(  1+E_{\omega}%
^{1}\,\varepsilon/E_{\omega}^{0}+O_{u}(\varepsilon)\right)  P_{0}%
\]
and, by Lemma \ref{str},
\begin{equation}
C_{00}(\varepsilon)^{-1}=\frac{1}{\varepsilon^{2}}\,\frac{1}{E_{\omega}^{0}%
}\,\left(  \left(  1+E_{\omega}^{1}\,\varepsilon/E_{\omega}^{0}\right)
^{-1}\,P_{0}+P_{0}O_{u}(\varepsilon)P_{0}\right)  \,. \label{C00-1bis}%
\end{equation}
Notice that $1+E_{\omega}^{1}\,\varepsilon/E_{\omega}^{0}\not =0$ since
$E_{\omega}^{1}\in i\mathbb{R}$. Then, by \eqref{C00-1bis}, one gets, as in
the proof of point (1) in Theorem \ref{Proposition_Schur},
\[
\varepsilon^{2}{\mathbb{M}}(\varepsilon)^{-1}=%
\begin{bmatrix}
(E_{\omega}^{0})^{-1}\,\left(  \left(  1+E_{\omega}^{1}\,\varepsilon
/E_{\omega}^{0}\right)  ^{-1}\,P_{0}+P_{0}O_{u}(\varepsilon)P_{0}\right)  &
P_{0}O(\varepsilon^{2})Q_{0}\\
Q_{0}O(\varepsilon^{2})P_{0} & Q_{0}O(\varepsilon^{2})Q_{0}%
\end{bmatrix}
\,.
\]
This entails
\[
\varepsilon^{2}\left(  \varepsilon^{2}+\left(  1-\varepsilon^{2}\right)
\left(  1/2+K_{\varepsilon\omega}\right)  \right)  ^{-1}=\frac{1}{E_{\omega
}^{0}}\,\left(  \left(  1+E_{\omega}^{1}\,\varepsilon/E_{\omega}^{0}\right)
^{-1}\,P_{0}+P_{0}O_{u}(\varepsilon)P_{0}\right)  +{O_{u}}(\varepsilon
^{2})\,.
\]
The proof is then concluded by proceeding as in Theorem \ref{Theorem_Schur}.
\end{proof}

\subsection{\label{Sec_Asymptotic_1}Proof of Theorem \ref{Th 3}\hfill}

The result of Theorem \ref{teo-unif} allows to improve our analysis of the
asymptotic acoustic scattering including the quasi-resonant regime
$|\,\omega-\omega_{M}|\gtrsim\varepsilon$. Let $c_{M}>0$, $\mathcal{I}%
_{M}\subset\mathbb{R}_{+}$ be a bounded interval containing $\omega_{M}$. For
$u_{\omega}^{\mathrm{in}}\in H_{-\alpha}^{2}(\mathbb{R}^{3})$, $\alpha>1/2$, a
solution of the homogeneous Helmholtz equation (\ref{Helmholtz_eq_omega}), we
denote as: $u_{\omega}\left(  \varepsilon\right)  =u_{\omega}^{\mathrm{sc}%
}(\varepsilon)+u_{\omega}^{\mathrm{in}}$ the unique solution of the problem
(\ref{bvp}) and proceed as in the proofs of Theorem \ref{conv-gen}.\par Since, by
Lemma \ref{anS}, $S_{\varepsilon\omega}^{-1}=S_{0}^{-1}+O(\varepsilon\omega)$
and, by Lemma \ref{Lemma_S_DN}, $S_{\varepsilon\omega}D\!N_{\varepsilon\omega
}=Q_{0}(1/2+K_{0})Q_{0}+(\varepsilon\omega)^{2}K_{(2)}+O((\varepsilon
\omega)^{3})$, one has
\[
S_{\varepsilon\omega}^{-1}=S_{0}^{-1}+O_{u}(\varepsilon)
\]
and
\[
S_{\varepsilon\omega}D\!N_{\varepsilon\omega}=Q_{0}(1/2+K_{0})Q_{0}%
+(\varepsilon\omega)^{2}K_{(2)}+O_{u}(\varepsilon^{3})\,,
\]
where $O_{u}(\varepsilon^{\lambda})$ means that the corresponding estimate
holds -- in the appropriate norm -- uniformly with respect to $\omega$ in
$\left\{  \omega\in\mathcal{I}_{M}:\left\vert \omega-\omega_{M}\right\vert
\geq c_{M}\,\varepsilon\right\}  $. Since $\varepsilon/E_{\omega}^{0}%
=O_{u}(1)$, combining such relations with \eqref{ool}, one gets
\begin{equation}
\Lambda_{\omega}^{\omega}(\varepsilon)=\frac{\varepsilon}{E_{\omega}^{0}}%
\frac{E_{\omega}^{0}}{E_{\omega}^{0}+E_{\omega}^{1}\,\varepsilon}\ \omega
^{2}S_{0}^{-1}P_{0}K_{(2)}+P_{0}O_{u}(\varepsilon)P_{0}+O_{u}(\varepsilon
^{2})\,. \label{Lambda_exp_1}%
\end{equation}
From (\ref{qt}), (\ref{tut}) and (\ref{Lambda_exp_1}) there follows%
\begin{align}
&  \left.  \Lambda_{\omega}^{\omega}(\varepsilon)\gamma_{0}\left(  u_{\omega
}^{\mathrm{in}}\circ\Phi_{\varepsilon}\right)  =\right. \nonumber\\
& \nonumber\\
&  \left.  =\left(  \frac{\varepsilon}{E_{\omega}^{0}}\frac{E_{\omega}^{0}%
}{E_{\omega}^{0}+E_{\omega}^{1}\,\varepsilon}\ \omega^{2}S_{0}^{-1}%
P_{0}K_{(2)}+P_{0}O_{u}(\varepsilon)P_{0}+O_{u}(\varepsilon^{2})\right)
\left(  u_{\omega}^{\mathrm{in}}(y_{0})+\left\Vert u_{\omega}^{\mathrm{in}%
}\right\Vert _{H_{-\alpha}^{2}(\mathbb{R}^{3})}^{2}O_{u}(\varepsilon
^{1/2})\right)  \right. \nonumber\\
& \nonumber\\
&  \left.  =-\frac{\varepsilon}{E_{\omega}^{0}+E_{\omega}^{1}\,\varepsilon
}\left(  u_{\omega}^{\mathrm{in}}(y_{0})\frac{\omega^{2}}{\omega_{M}^{2}%
}\ S_{0}^{-1}(1)+\left\Vert u_{\omega}^{\mathrm{in}}\right\Vert _{H_{-\alpha
}^{2}(\mathbb{R}^{3})}^{2}\,O_{u}(\varepsilon^{1/2})\right)  +\,u_{\omega
}^{\mathrm{in}}\left(  y_{0}\right)  O_{u}(\varepsilon)\right. \nonumber\\
&  \left.  +O_{u}(\varepsilon^{2})\left(  u_{\omega}^{\mathrm{in}}\left(
y_{0}\right)  +\Vert u_{\omega}^{\mathrm{in}}\Vert_{H_{-\alpha}^{2}%
(\mathbb{R}^{3})}^{2}\,O(\varepsilon^{1/2})\right)  \,.\right. \nonumber
\end{align}
By the definitions of $E_{\omega}^{0}$, $E_{\omega}^{1}$ and $\omega_{M}^{2}$,
results
\[
\frac{\varepsilon}{E_{\omega}^{0}+E_{\omega}^{1}\,\varepsilon}=\frac
{\varepsilon\omega_{M}^{2}}{\omega_{M}^{2}-\omega^{2}}\left(  1-i\,\frac
{\omega^{3}\varepsilon}{\omega_{M}^{2}-\omega^{2}}\frac{c_{\Omega}}{4\pi
}\right)  ^{-1}=O_{u}(1)\,,
\]
and the r.h.s. simplifies to
\begin{align}
&  \left.  \Lambda_{\omega}^{\omega}(\varepsilon)\gamma_{0}\left(  u_{\omega
}^{\mathrm{in}}\circ\Phi_{\varepsilon}\right)  =\right. \nonumber\\
& \nonumber\\
&  \left.  =-\frac{\varepsilon\omega^{2}}{\omega_{M}^{2}-\omega^{2}}\left(
1-i\,\frac{\omega^{3}\varepsilon}{\omega_{M}^{2}-\omega^{2}}\frac{c_{\Omega}%
}{4\pi}\right)  ^{-1}u_{\omega}^{\mathrm{in}}(y_{0})S_{0}^{-1}(1)+\left\Vert
u_{\omega}^{\mathrm{in}}\right\Vert _{H_{-\alpha}^{2}(\mathbb{R}^{3})}%
^{2}\,O_{u}(\varepsilon^{1/2})\,.\right.  \label{Lambda_exp_unif}%
\end{align}
Since $\mathcal{I}_{M}\ni\omega\mapsto\Vert R_{\omega}\Vert_{L_{\alpha}%
^{2}(\mathbb{R}^{3}),H_{-\alpha}^{2}(\mathbb{R}^{3})}$ is continuous (see
Subsection \ref{Section_LAP}), by Corollary \ref{Lemma_trace_R_eps_k_U_eps},
one gets
\[
G_{\omega}(\varepsilon)=G_{\omega}+O_{u}(\varepsilon^{1/2})\,.
\]
Therefore, from
\[
u_{\omega}^{\mathrm{sc}}(\varepsilon)=-G_{\omega}(\varepsilon)\Lambda_{\omega
}^{\omega}(\varepsilon)\gamma_{0}\left(  u_{\kappa}^{\mathrm{in}}\circ
\Phi_{\varepsilon}\right)  \,,
\]
combining the above expansions, we obtain
\[
u_{\omega}^{\mathrm{sc}}(\varepsilon)=\left(  G_{\omega}+O_{u}(\varepsilon
^{1/2})\right)  \left(  \frac{\varepsilon\omega^{2}}{\omega_{M}^{2}-\omega
^{2}}\left(  1-i\,\frac{\omega^{3}\varepsilon}{\omega_{M}^{2}-\omega^{2}}%
\frac{c_{\Omega}}{4\pi}\right)  ^{-1}u_{\omega}^{\mathrm{in}}\left(
y_{0}\right)  S_{0}^{-1}(1)+\left\Vert u_{\omega}^{\mathrm{in}}\right\Vert
_{H_{-\alpha}^{2}(\mathbb{R}^{3})}^{2}\,O_{u}(\varepsilon^{1/2})\right)  \,.
\]
The definition of $G_{\omega}$ and in particular: $G_{\omega}S_{0}%
^{-1}\,(1)=c_{\Omega}\,\mathcal{G}_{\omega}\left(  \cdot-y_{0}\right)
=c_{\Omega}\,\mathcal{G}_{\omega}^{y_{0}}$ (see (\ref{G_k_id})), leads to the
expansion%
\[
u_{\omega}^{\mathrm{sc}}(\varepsilon)=\frac{\varepsilon\omega^{2}}{\omega
_{M}^{2}-\omega^{2}}\left(  1-i\,\frac{\omega^{3}\varepsilon}{\omega_{M}%
^{2}-\omega^{2}}\frac{c_{\Omega}}{4\pi}\right)  ^{-1}u_{\omega}^{\mathrm{in}%
}(y_{0})c_{\Omega}\,\mathcal{G}_{\omega}^{y_{0}}+\left\Vert u_{\omega
}^{\mathrm{in}}\right\Vert _{H_{-\alpha}^{2}(\mathbb{R}^{3})}^{2}%
\,O_{u}(\varepsilon^{1/2})\,,
\]
which corresponds to our statement after noticing that%
\[
\frac{\varepsilon\omega^{2}}{\omega_{M}^{2}-\omega^{2}}\left(  1-i\,\frac
{\omega^{3}\varepsilon}{\omega_{M}^{2}-\omega^{2}}\frac{c_{\Omega}}{4\pi
}\right)  ^{-1}=\,\frac{\varepsilon\,\omega^{2}}{\omega_{M}^{2}-\omega
^{2}-i\varepsilon\frac{\omega^{3}c_{\Omega}}{4\pi}}\,,
\]
and $\varepsilon/\left(  \omega_{M}^{2}-\omega^{2}\right)  =O_{u}(1)$.\hfill

\appendix\hfill

\section{Resolvent analysis, boundary integral operators and operator
expansions}

\subsection{(Weighted) Sobolev spaces}

Given $\Omega\subset\mathbb{R}^{3}$ open and bounded, with smooth boundary
$\Gamma$, we adopt the notation
\[
\Omega_{\mathrm{in}}=\Omega\,,\qquad\Omega_{\mathrm{ex}}=\mathbb{R}%
^{3}\backslash\overline{\Omega}\,.
\]
The symbols $H^{s}(\mathbb{R}^{3})$, $H^{s}(\Omega_{\mathrm{in}})$,
$H^{s}(\Omega_{\mathrm{ex}})$, $H^{s}(\Gamma)$, $s\in\mathbb{R}$, denote the
usual scales of Sobolev-Hilbert spaces of function on $\mathbb{R}^{3}$,
$\Omega_{\mathrm{in}}$, $\Omega_{\mathrm{ex}}$ and $\Gamma$ respectively (see,
e.g., \cite{McLe}). We use the notation
\[
H^{s}(\mathbb{R}^{3}\backslash\Gamma):=H^{s}(\Omega_{\mathrm{in}})\oplus
H^{s}(\Omega_{\mathrm{ex}})\,.
\]
Let $\left\langle x\right\rangle :=(1+|x|^{2})^{1/2}$ and $\alpha\in
\mathbb{R}$. Then we define the weighted $L^{2}$-space by
\begin{equation}
L_{\alpha}^{2}(\mathbb{R}^{3}):=\bigl\{u\in L_{\mathrm{loc}}^{2}%
(\mathbb{R}^{3}):\Vert u\Vert_{L_{\alpha}^{2}(\mathbb{R}^{3})}<+\infty
\bigr\}\,,\quad\Vert u\Vert_{L_{\alpha}^{2}(\mathbb{R}^{3})}:=\Vert
\left\langle x\right\rangle ^{\alpha}u\Vert_{L^{2}(\mathbb{R}^{3})}\,.
\label{W_Sob_0}%
\end{equation}
The weighted Sobolev spaces of positive integer order $\ell$ are defined by
\begin{equation}
H_{\alpha}^{\ell}(\mathbb{R}^{3})=\bigl\{u\in L_{\alpha}^{2}(\mathbb{R}%
^{3}):\Vert u\Vert_{H_{\alpha}^{\ell}(\mathbb{R}^{3})}<+\infty\bigr\},\quad
\Vert u\Vert_{H_{\alpha}^{\ell}(\mathbb{R}^{3})}^{2}:=\sum_{|k|\leq\ell}\Vert
D^{k}u\Vert_{L_{\alpha}^{2}(\mathbb{R}^{3})}^{2}. \label{W_Sob}%
\end{equation}
If $s>0$ is not integer, $H_{\alpha}^{s}(\mathbb{R}^{3})$ is defined via
interpolation and for $s<0$ we define $H_{\alpha}^{s}(\mathbb{R}^{3})$ as the
dual of $H_{-\alpha}^{-s}(\mathbb{R}^{3}))$.

The spaces $L_{\alpha}^{2}(\Omega_{\mathrm{ex}})$ and $H_{\alpha}^{s}%
(\Omega_{\mathrm{ex}})$ are defined in a similar way. One has
\[
L_{\alpha}^{2}(\mathbb{R}^{3})=L^{2}(\Omega_{\mathrm{in}})\oplus L_{\alpha
}^{2}(\Omega_{\mathrm{ex}})
\]
and we set%
\[
H_{\alpha}^{s}(\mathbb{R}^{3}\backslash\Gamma):=H^{s}(\Omega_{\mathrm{in}%
})\oplus H_{\alpha}^{s}(\Omega_{\mathrm{ex}})\,.
\]

\subsection{\label{Section_LAP}The free resolvent}

Let $\Delta$ be the distributional Laplacian; whenever restricted to
$H^{2}(\mathbb{R}^{3})$, it is a self-adjoint operator in $L^{2}%
(\mathbb{R}^{3})$ and its resolvent
\begin{equation}
R_{z}:=\left(  -\Delta-z^{2}\right)  ^{-1}\,,\qquad z\in\mathbb{C}_{+}\,,
\label{res}%
\end{equation}
provides a map $R_{z}\in${$\mathcal{L}$}$(H^{s}(\mathbb{R}^{3}),H^{s+2}%
(\mathbb{R}^{3}))$ for any $s\geq0$. For any $u\in L^{1}(\mathbb{R}^{3})\cap
L^{2}(\mathbb{R}^{3})$ and $z\in\mathbb{C}_{+}$ one has the integral
representation
\begin{equation}
R_{z}u(x)=\int_{\mathbb{R}^{3}}\mathcal{G}_{z}(x-y)u(y)\,dy\,,\qquad
\mathcal{G}_{z}(x):=\frac{e^{iz|x|}}{4\pi\,|x|}\,. \label{Res_kernel}%
\end{equation}
$R_{z}$ in \eqref{res} extends to a map $R_{z}\in${$\mathcal{L}$}%
$(H^{s}(\mathbb{R}^{3}),H^{s+2}(\mathbb{R}^{3}))$ for any real $s$; moreover,
$\mathbb{C}_{+}\ni z\mapsto R_{z}$ is a {$\mathcal{L}$}$(H^{s}(\mathbb{R}%
^{3}),H^{s+2}(\mathbb{R}^{3}))\,$-valued continuous map for any real $s$. By
the resolvent identity
\[
R_{z}-R_{w}=(z^{2}-w^{2})R_{w}R_{z}\,,
\]
the latter entails that $\mathbb{C}_{+}\ni z\mapsto R_{z}$ is a {$\mathcal{L}%
$}$(H^{s}(\mathbb{R}^{3}),H^{s+2}(\mathbb{R}^{3}))\,$-valued analytic map for
any real $s$.

By the Limiting Absorption Principle (see, e.g., \cite[Theorem 18.3]{KoKo}),
$\mathbb{C}_{+}\ni z\mapsto R_{z}$ extends to a map $\overline{\mathbb{C}_{+}%
}\ni z\mapsto R_{z}$ defined as
\begin{equation}
z\mapsto%
\begin{cases}
(-\Delta-z^{2})^{-1}\,, & z\in\mathbb{C}_{+}\\
\lim_{\,\delta\rightarrow0_{+}}\left(  -\Delta-(\kappa+i\delta)^{2}\right)
^{-1}\,, & z=\kappa\in\mathbb{R}\,.
\end{cases}
\label{C+}%
\end{equation}
The above limit exists in $\mathcal{L}(H_{\alpha}^{-s}(\mathbb{R}%
^{3}),H_{-\alpha}^{-s+2}(\mathbb{R}^{3}))$ for any $s\in\lbrack0,2]$, where
$\alpha>1/2$ whenever $\kappa\not =0$, or $\alpha>1$ if $\kappa=0$; moreover,
$\overline{\mathbb{C}_{+}}\backslash\{0\}\ni z\mapsto R_{z}$ is continuous as
a {$\mathcal{L}$}$(H_{\alpha}^{-s}(\mathbb{R}^{3}),H_{-\alpha}^{-s+2}%
(\mathbb{R}^{3}))\,$-valued map for any $\alpha>1/2$ and $\overline
{\mathbb{C}_{+}}\ni z\mapsto R_{z}$ is continuous as a {$\mathcal{L}$%
}$(H_{\alpha}^{-s}(\mathbb{R}^{3}),H_{-\alpha}^{-s+2}(\mathbb{R}^{3}%
))\,$-valued map for any $\alpha>1$.

We extend $z\mapsto R_{z}$ in \eqref{C+} to the whole $\mathbb{C}$ by
\begin{equation}
R_{z}u:=\mathcal{G}_{z}\ast u\,,\qquad z\in\mathbb{C}\,. \label{R*}%
\end{equation}
The two definitions \eqref{C+} and \eqref{R*} agree when $z\in\overline
{\mathbb{C}_{+}}$, while the integral representation \eqref{Res_kernel} still
holds for $z\in\overline{\mathbb{C}_{-}}$ and $u\in\mathcal{D}(\mathbb{R}%
^{3})\equiv\mathcal{C}_{\mathrm{comp}}^{\infty}(\mathbb{R}^{3})$. Since
$\mathcal{G}_{z}\in L_{\mathrm{loc}}^{1}(\mathbb{R}^{3})\subset\mathcal{D}%
^{\prime}(\mathbb{R}^{3})$, $R_{z}$ in \eqref{R*} belongs to $\mathcal{L}%
(${$\mathcal{E}$}$^{\prime}(\mathbb{R}^{3}),${$\mathcal{D}$}$^{\prime
}(\mathbb{R}^{3}))$ (see, e.g., \cite[Theorem 27.6]{Trev}). Since the series
\begin{equation}
\mathcal{G}_{z}=\mathcal{G}_{0}+\sum_{n=1}^{+\infty}\mathcal{G}_{(n)}%
\,z^{n}\,,\qquad\mathcal{G}_{(n)}(x):=\frac{1}{4\pi}\,\frac{i^{n}}%
{n!}\ |x|^{n-1} \label{gexp}%
\end{equation}
converges in {$\mathcal{D}$}$^{\prime}(\mathbb{R}^{3})$ and the map $f\mapsto
f\ast u$ belongs to $\mathcal{L}(${$\mathcal{D}$}$^{\prime}(\mathbb{R}^{3}))$
for any $u\in\mathcal{E}^{\prime}(\mathbb{R}^{3})$ (see, e.g., \cite[Theorem
27.6]{Trev}), one has
\begin{equation}
R_{z}=R_{0}+\sum_{n=1}^{+\infty}R_{(n)}\,z^{n}\,,\qquad R_{(n)}u:=\mathcal{G}%
_{(n)}\ast u\,, \label{esp}%
\end{equation}
and the series strongly converges in $\mathcal{L}(${$\mathcal{E}$}$^{\prime
}(\mathbb{R}^{3}),${$\mathcal{D}$}$^{\prime}(\mathbb{R}^{3}))$.

\subsection{Trace maps}

Here we recall some well known definitions and results about traces in Sobolev
spaces (see, e.g., \cite{McLe}). The zero and first-order traces on $\Gamma$
are defined on smooth functions as%
\begin{equation}
\gamma_{0}u=\left.  u\right\vert _{\Gamma}\,,\qquad\gamma_{1}u=\nu\cdot\nabla
u|{\Gamma}\,, \label{trace}%
\end{equation}
where $\nu$ is the exterior unit normal to $\Gamma$, and extend to bounded
linear operators
\begin{equation}
\gamma_{0}\in{\mathcal{L}}(H^{s}(\mathbb{R}^{3}),H^{{s-\frac{1}{2}}}%
(\Gamma))\,,\quad s>\frac{1}{2}\,,\qquad\gamma_{1}\in{\mathcal{L}}%
(H^{s}(\mathbb{R}^{3}),H^{{s-\frac{3}{2}}}(\Gamma))\,,\quad s>\frac{3}{2}\,.
\label{trace_est}%
\end{equation}
The one-sided trace maps
\[
\gamma_{0}^{\mathrm{in/ex}}\in{\mathcal{L}}(H^{s}(\Omega_{\mathrm{in/ex}%
}),H^{s-\frac{1}{2}}(\Gamma))\,,\quad s>\frac{1}{2}\,,\qquad\gamma
_{1}^{\mathrm{in/ex}}\in{\mathcal{L}}(H^{s}(\Omega_{\mathrm{in/ex}%
}),H^{{s-\frac{3}{2}}}(\Gamma))\,,\quad s>\frac{3}{2}\,,
\]
defined on smooth (up to the boundary) functions by%
\[
\gamma_{0}^{\mathrm{in/ex}}u_{\mathrm{in/ex}}=u_{\mathrm{in/ex}}%
|\Gamma\,,\qquad\gamma_{1}^{\mathrm{in/ex}}u_{\mathrm{in/ex}}=\nu\cdot\nabla
u_{\mathrm{in/ex}}|\Gamma\,,
\]
can be extended to
\[
\gamma_{0}^{\mathrm{in/ex}}\in{\mathcal{L}}(H_{\Delta}^{0}(\Omega
_{\mathrm{in/ex}})),H^{-{\frac{1}{2}}}(\Gamma))\,,\qquad\gamma_{1}%
^{\mathrm{in/ex}}\in{\mathcal{L}}(H_{\Delta}^{0}(\Omega_{\mathrm{in/ex}%
})),H^{-{\frac{3}{2}}}(\Gamma))\,,
\]
where%
\[
H_{\Delta}^{0}(\Omega_{\mathrm{in/ex}}):=\{u_{\mathrm{in/ex}}\in L^{2}%
(\Omega_{\mathrm{in/ex}}):\Delta u_{\mathrm{in/ex}}\in L^{2}(\Omega
_{\mathrm{in/ex}})\}\,,
\]%
\[
\Vert u_{\mathrm{in/ex}}\Vert_{H_{\Delta}^{0}(\Omega_{\mathrm{in/ex}})}%
^{2}:=\Vert\Delta u_{\mathrm{in/ex}}\Vert_{L^{2}(\Omega_{\mathrm{in/ex}})}%
^{2}+\Vert u_{\mathrm{in/ex}}\Vert_{L^{2}(\Omega_{\mathrm{in/ex}})}^{2}\,.
\]
Setting
\[
\Delta_{\Omega_{\mathrm{in/ex}}}^{\mathrm{max}}:=\Delta|H_{\Delta}^{0}\left(
\Omega_{\mathrm{in/ex}}\right)  \,,
\]
by the \textquotedblright half\textquotedblright\ Green formula (see
\cite[Theorem 4.4]{McLe}), one has, for any $u,v\in H^{1}(\Omega
_{\mathrm{in/ex}})\cap H_{\Delta}^{0}(\Omega_{\mathrm{in/ex}})$,
\begin{align}
&  \langle-\Delta_{\Omega_{\mathrm{in/ex}}}^{\mathrm{max}}u_{\mathrm{in/ex}%
},v_{\mathrm{in/ex}}\rangle_{L^{2}\left(  \Omega_{\mathrm{in/ex}}\right)
}\label{hG}\\
=  &  \langle\nabla u_{\mathrm{in/ex}},\nabla v_{\mathrm{in/ex}}\rangle
_{L^{2}(\Omega_{\mathrm{in/ex}})}+\epsilon_{\mathrm{in/ex}}\langle\gamma
_{1}^{\mathrm{in/ex}}u_{\mathrm{in/ex}},\gamma_{0}^{\mathrm{in/ex}%
}v_{\mathrm{in/ex}}\rangle_{H^{-1/2}(\Gamma),H^{1/2}(\Gamma)}\,,\quad\nonumber
\end{align}
where $\epsilon_{\mathrm{in}}=-1$ and $\epsilon_{\mathrm{ex}}=1$. Setting
\begin{equation}
H_{\Delta}^{0}(\mathbb{R}^{3}\backslash\Gamma):=H_{\Delta}^{0}(\Omega
_{\mathrm{in}})\oplus H_{\Delta}^{0}(\Omega_{\mathrm{ex}})\,, \label{hzd}%
\end{equation}
the extended traces allow to define
\[
\gamma_{\ell}\in{\mathcal{L}}(H_{\Delta}^{0}(\mathbb{R}^{3}\backslash
\Gamma),H^{-{\frac{1}{2}}-\ell}(\Gamma))\,,\qquad\lbrack\gamma_{\ell}%
]\in{\mathcal{L}}(H_{\Delta}^{0}(\mathbb{R}^{3}\backslash\Gamma),H^{-{\frac
{1}{2}}-\ell}(\Gamma))\,,\quad\ell=0,1\,,
\]
by
\[
\gamma_{\ell}u:=\frac{1}{2}\left(  \gamma_{\ell}^{\mathrm{in}}\left(
u|\Omega_{\mathrm{in}}\right)  +\gamma_{\ell}^{\mathrm{ex}}\left(
u|\Omega_{\mathrm{ex}}\right)  \right)  \,,\qquad\lbrack\gamma_{\ell
}]u:=\gamma_{\ell}^{\mathrm{ex}}\left(  u|\Omega_{\mathrm{ex}}\right)
-\gamma_{\ell}^{\mathrm{in}}\left(  u|\Omega_{\mathrm{in}}\right)  \,.
\]
Notice that the maps $\gamma_{\ell}|H^{2}(\mathbb{R}^{3}\backslash\Gamma)$,
$\ell=0,1$, coincide with the ones in \eqref{trace_est} when restricted to
$H^{2}(\mathbb{R}^{3})$.

These operators can be further extended to $H_{\alpha}^{2}(\mathbb{R}%
^{3}\backslash\Gamma)$, $\alpha<0$, by%
\[
\gamma_{\ell}^{\mathrm{in/ex}}u_{\mathrm{in/ex}}:=\gamma_{\ell}%
^{\mathrm{in/ex}}\left(  \chi u_{\mathrm{in/ex}}\right)  \,,\quad\ell=0,1\,,
\]
where $\chi$ belongs to $\mathcal{C}_{\mathrm{comp}}^{\infty}(\Omega^{c})$ and
$\chi=1$ on a neighborhood of $\Gamma$.

\subsection{\label{Sec_SL}The single layer boundary operator}

From (\ref{trace_est}) follows that $\gamma_{0}^{\ast}$ is a bounded mapping:
$H^{1/2-{s}}\left(  \Gamma\right)  \rightarrow H^{-s}\left(  \mathbb{R}%
^{3}\right)  $ for any $s>1/2$. Since $\gamma_{0}^{\ast}\varphi$ has bounded
support, results: $\gamma_{0}^{\ast}\in\mathcal{L}\left(  H^{1/2-{s}}\left(
\Gamma\right)  ,\mathcal{E}^{\prime}\left(  \mathbb{R}^{3}\right)  \right)  $.
Let $z\in\mathbb{C}$; by $R_{z}\in\mathcal{L}\left(  {\mathcal{E}}^{\prime
}\left(  \mathbb{R}^{3}\right)  ,{\mathcal{D}}^{\prime}\left(  \mathbb{R}%
^{3}\right)  \right)  $, we get: $R_{z}\gamma_{0}^{\ast}\in\mathcal{L}\left(
H^{1/2-{s}}\left(  \Gamma\right)  ,{\mathcal{D}}^{\prime}\left(
\mathbb{R}^{3}\right)  \right)  $, $s>1/2$. This defines the well known single
layer operator
\[
S\!L_{z}=R_{z}\gamma_{0}^{\ast}\,.
\]
Let recall from \cite[Corollary 6.14]{McLe} that the mapping properties%
\begin{equation}%
\begin{array}
[c]{ccc}%
\chi SL_{z}\in\mathcal{L}(H^{s}(\Gamma),H^{s+3/2}(\mathbb{R}^{3}%
\backslash\Gamma))\,, &  & s>-1
\end{array}
\label{slmp}%
\end{equation}
and the jump relations%
\begin{equation}%
\begin{array}
[c]{ccc}%
\left[  \gamma_{0}\right]  S\!L_{z}=0\,, &  & \left[  \gamma_{1}\right]
S\!L_{z}=-1\,,
\end{array}
\label{jump}%
\end{equation}
hold for any $\chi\in\mathcal{C}_{\mathrm{comp}}^{\infty}\left(
\mathbb{R}^{3}\right)  $ and $z\in{\mathbb{C}}$. Moreover, $S\!L_{z}$ has the
integral representation
\begin{equation}
S\!L_{z}\phi=\int_{\Gamma}\mathcal{G}_{z}\left(  \cdot-y\right)
\,\phi(y)\,d\sigma(y)\,, \label{intrep}%
\end{equation}
where $\sigma$ denotes the surface measure. When $z\in{\mathbb{C}}_{+}$, the
mapping properties of $R_{z}$, the identity: $S\!L_{z}=(\gamma_{0}R_{-\bar{z}%
})^{\ast}$ and a duality argument allow to improve (\ref{slmp}) as follows%
\[%
\begin{array}
[c]{ccccc}%
S\!L_{z}\in\mathcal{L}\left(  H^{s}\left(  \Gamma\right)  ,H^{s+3/2}\left(
{\mathbb{R}}^{3}\right)  \right)  \,, &  & s\geq-3/2\,, &  & z\in{\mathbb{C}%
}_{+}\,.
\end{array}
\]

Next we define the single layer boundary operator
\[
S_{z}:=\gamma_{0}S\!L_{z}\,.
\]
By \cite[Theorem 7.2]{McLe}, $S_{z}\in\mathcal{L}(H^{s}(\Gamma),H^{s+1}%
(\Gamma))$ for any real $s$. The operator $S_{0}$ plays a central role in our
construction. By \cite[Corollary 8.13]{McLe}, $S_{0}:H^{-1/2}(\Gamma
)\rightarrow H^{1/2}(\Gamma))$ is (self-adjoint) positive and bounded from
below:
\begin{equation}
\langle\phi,S_{0}\phi\rangle_{H^{-1/2}(\Gamma),H^{1/2}(\Gamma)}\geq
c_{0}\,\Vert\phi\Vert_{H^{-1/2}(\Gamma)}^{2}\,,\qquad c_{0}>0. \label{coe}%
\end{equation}
Hence $S_{0}^{-1}\in${$\mathcal{L}$}$(H^{1/2}(\Gamma),H^{-1/2}(\mathbb{R}%
^{3}))$ provides an inner product in $H^{1/2}(\Gamma)$ defined by
\begin{equation}
\langle\phi,\varphi\rangle_{S_{0}^{-1}}:=\langle S_{0}^{-1}\phi,\varphi
\rangle_{H^{-1/2}(\Gamma),H^{1/2}(\Gamma)}\,. \label{ip}%
\end{equation}
By \eqref{coe},
the inner product in \eqref{ip} induces a norm $\Vert\cdot\Vert_{S_{0}^{-1}}$
on $H^{1/2}(\Gamma)$ which is equivalent to the original one.

\begin{lemma}
\label{anS}$\mathbb{C}\ni z\mapsto S_{z}$ is a {$\mathcal{L}$}$\left(
H^{-1/2}\left(  \Gamma\right)  ,H^{1/2}\left(  \mathbb{R}^{3}\right)  \right)
\,$-valued analytic map and
\begin{equation}
S_{z}=S_{0}+\sum_{n=1}^{+\infty}S_{(n)}\,z^{n}\,,
\end{equation}
where $S_{(n)}$ has the integral representation
\begin{equation}
S_{(n)}\phi(x)=\frac{1}{4\pi}\,\frac{i^{n}}{n!}\,\int_{\Gamma}|x-y|^{n-1}%
\phi(y)\,d\sigma(y)\,, \label{Sn}%
\end{equation}
and the series converges in {$\mathcal{L}$}$_{H\!S}(H^{-1/2}(\Gamma
),H^{1/2}(\mathbb{R}^{3}))$. Let $D_{\Omega}$ be the discrete set $D_{\Omega
}:=\{z\in\mathbb{C}:z^{2}\in\sigma(-\Delta_{\Omega}^{D})\}$. Then:
$\mathbb{C}\backslash D_{\Omega}\ni z\mapsto S_{z}^{-1}$ is a {$\mathcal{L}$%
}$\left(  H^{1/2}\left(  \Gamma\right)  ,H^{-1/2}\left(  \mathbb{R}%
^{3}\right)  \right)  \,$-valued analytic map.
\end{lemma}

\begin{proof}
Let $\{\phi_{k}^{\pm}\}_{1}^{\infty}\subset\mathcal{C}^{\infty}(\Gamma)$ be
the orthonormal basis in $H^{\pm1/2}(\Gamma)$ defined by $\phi_{k}^{\pm
}:=(-\Delta_{LB}+1)^{\mp1/4}\varphi_{k}$, where $\{\varphi_{k}\}_{1}^{\infty
}\subset\mathcal{C}^{\infty}(\Gamma)$ is the set of normalized eigenfunctions
of the self-adjoint operator $\Delta_{LB}:H^{2}(\Gamma)\subset L^{2}%
(\Gamma)\rightarrow L^{2}(\Gamma)$. Here $\Delta_{LB}$ denotes the
Laplace-Beltrami operator on the surface $\Gamma$ with respect to the
Riemannian metric induced by the embedding $\Gamma\subset\mathbb{R}^{3}$. For
any couple of indices $(i,j)$, one has
\begin{align*}
&  \langle\phi_{i}^{+},S_{(n)}\phi_{j}^{-}\rangle_{H^{1/2}(\Gamma)}%
=\langle\phi_{i}^{-},S_{(n)}\phi_{j}^{-}\rangle_{L^{2}(\Gamma)}=\int%
_{\Gamma\times\Gamma}\phi_{i}^{-}(x)\,\mathcal{G}_{(n)}(x-y)\,\phi_{j}%
^{-}(y)\,d\sigma(x)d\sigma(y)\\
=  &  \langle\phi_{i}^{-}\otimes\phi_{j}^{-},\mathcal{R}_{(n)}\rangle
_{L^{2}(\Gamma)\otimes L^{2}(\Gamma)}=\big\langle\varphi_{i}\otimes\varphi
_{j},\big((-\Delta_{LB}+1)^{1/4}\otimes(-\Delta_{LB}+1)^{1/4}\big)\mathcal{R}%
_{(n)}\big\rangle_{L^{2}(\Gamma)\otimes L^{2}(\Gamma)}\,,
\end{align*}
where $\mathcal{R}_{(n)}(x,y):=\mathcal{G}_{(n)}(x-y)$. Therefore $S_{(n)}$ is
a Hilbert-Schmidt operator and its Hilbert-Schmidt norm is estimated by (the
penultimate inequality follows from \cite[Proposition 4.3]{DR})
\begin{align*}
&  \left.  \Vert S_{(n)}\Vert_{H\!S}^{2}=\sum_{k=1}^{\infty}\Vert S_{(n)}%
\phi_{k}^{-}\Vert_{H^{1/2}}^{2}\right. \\
&  \left.  =\sum_{i,j=1}^{\infty}\left\vert \big\langle\varphi_{i}%
\otimes\varphi_{i},\big((-\Delta_{LB}+1)^{1/4}\otimes(-\Delta_{LB}%
+1)^{1/4}\big)\mathcal{R}_{(n)}\big\rangle_{L^{2}(\Gamma)\otimes L^{2}%
(\Gamma)}\right\vert ^{2}\right. \\
&  \left.  =\Vert((-\Delta_{LB}+1)^{1/4}\otimes(-\Delta_{LB}+1)^{1/4}%
)\mathcal{R}_{\left(  n\right)  }\Vert_{L^{2}(\Gamma)\otimes L^{2}(\Gamma
)}^{2}\right. \\
&  \left.  \leq c\,\Vert\mathcal{R}_{\left(  n\right)  }\Vert_{H^{1}%
(\Gamma\times\Gamma)}^{2}\right. \\
&  \left.  \leq c\ \frac{|\Gamma|^{2}}{(n!)^{2}}\,\big((d_{\Omega}^{n-1}%
)^{2}+((n-1)\,d_{\Omega}^{n-2})^{2}\big)\,.\right.
\end{align*}
Hence the series
\[
\widetilde{S}_{z}:=\sum_{n=2}^{+\infty}S_{(n)}\,z^{n}%
\]
converges in $\mathcal{L}_{H\!S}(H^{1/2}(\Gamma))$ for any $z\in\mathbb{C}$
and defines the {$\mathcal{L}$}$_{H\!S}(H^{1/2}(\Gamma))\,$-valued analytic
map $\mathbb{C}\ni z\mapsto\widetilde{S}_{z}$. By \eqref{gexp} and
\eqref{intrep}, one has $\langle\phi_{i},\widetilde{S}_{z}\phi_{j}%
\rangle_{H^{1/2}(\Gamma)}=\langle\phi_{i},(S_{z}-S_{0})\phi_{j}\rangle
_{H^{1/2}(\Gamma)}$ for any $z\in\mathbb{C}$ and for any couple $(i,j)$.
Therefore $S_{z}=S_{0}+\widetilde{S}_{z}$ for any $z\in\mathbb{C}$.

By \cite[Theorem 7.6]{McLe}, $S_{z}$ is Fredholm with zero index; by
\cite[Theorem 7.5]{McLe}, $\mathrm{ker}(S_{z})\not =\{0\}$ is equivalent to
the existence of non-trivial solutions of $(\Delta_{\Omega}^{D}+z^{2})u=0$.
Hence $S_{z}^{-1}\in${$\mathcal{L}$}$(H^{-1/2}(\Gamma),H^{1/2}(\mathbb{R}%
^{3}))$ for any $z\in\mathbb{C}\backslash D_{\Omega}$. The proof is then
concluded by Lemma \ref{anS} and by the analyticity of the inverse (see
\cite[Theorem 5.1]{Tay}
). \hfill
\end{proof}

\subsection{\label{Sec_K}The Neumann-Poincar\'{e} operator}

Proceeding as before we observe from (\ref{trace_est}) that $\gamma_{1}^{\ast
}\in\mathcal{L}\left(  H^{3/2-{s}}\left(  \Gamma\right)  ,\mathcal{E}^{\prime
}\left(  \mathbb{R}^{3}\right)  \right)  $ for $s>3/2$. Then, for
$z\in\mathbb{C}$ we get: $R_{z}\gamma_{0}^{\ast}\in\mathcal{L}\left(
H^{1/2-{s}}\left(  \Gamma\right)  ,{\mathcal{D}}^{\prime}\left(
\mathbb{R}^{3}\right)  \right)  $, $s>1/2$. This defines the double layer
operator
\[
D\!L_{z}=R_{z}\gamma_{1}^{\ast}\,.
\]
Let recall from \cite[Corollary 6.14]{McLe} that the mapping properties%
\[%
\begin{array}
[c]{ccc}%
\chi DL_{z}\in\mathcal{L}(H^{s}(\Gamma),H^{s+1/2}(\mathbb{R}^{3}%
\backslash\Gamma))\,, &  & s>0
\end{array}
\]
hold for any $\chi\in\mathcal{C}_{\mathrm{comp}}^{\infty}\left(
\mathbb{R}^{3}\right)  $ and $z\in{\mathbb{C}}$. Moreover, $DL_{z}$ has the
integral representation
\[
D\!L_{z}\phi=\int_{\Gamma}\nu(y)\!\cdot\!\nabla_{\!y}\mathcal{G}_{z}%
(\cdot-y)\phi(y)\,d\sigma(y)\,.
\]

Next we define the \emph{Neumann-Poincar\'{e}} boundary operator
\[
K_{z}:=\gamma_{0}D\!L_{z}\,.
\]
By \cite[Theorem 7.2]{McLe}, $K_{z}\in\mathcal{L}(H^{s}(\Gamma))$ for any real
$s$. The next Lemma resumes the spectral properties of $K_{0}$ (see, e.g.,
\cite[Section 4]{MS}):

\begin{lemma}
$K_{0}$ is a compact operator in $L^{2}(\Gamma)$ and $\sigma(K_{0}%
)\subseteq[-1/2,1/2)$; $\lambda_{0}=-1/2$ is a simple eigenvalue and the
corresponding eigenfunction is $\phi_{0}=1$.
\end{lemma}

Let $P_{0}:H^{1/2}(\Gamma)\rightarrow H^{1/2}(\Gamma)$ be the orthogonal
(w.r.t. the inner product $\langle\cdot,\cdot\rangle_{S_{0}^{-1}}$) projector
onto $\mathbb{C}$, i.e. onto the subspace generated by the eigenfunction
$\phi_{0}$:
\begin{equation}
P_{0}\phi:=c_{\Omega}^{-1}\langle\phi_{0},\phi\rangle_{S_{0}^{-1}}\,\phi
_{0}\equiv c_{\Omega}^{-1}\langle S_{0}^{-1}1,\phi\rangle_{H^{-1/2}%
(\Gamma),H^{1/2}(\Gamma)}\,1\,. \label{P0}%
\end{equation}
Denoting then by $Q_{0}$ the orthogonal projector onto $\mathrm{ran}%
(P_{0})^{\perp}$, i.e.
\begin{equation}
Q_{0}\phi:=\phi-P_{0}\phi\,, \label{Q0}%
\end{equation}
$K_{0}\in\mathcal{L}(H^{1/2}(\Gamma))$ has the decomposition
\begin{equation}
K_{0}=P_{0}K_{0}P_{0}+Q_{0}K_{0}Q_{0}=-\frac{1}{2}\,P_{0}+Q_{0}K_{0}Q_{0}\,.
\label{K_0_spectral_decomp}%
\end{equation}

\begin{lemma}
\label{anK}$\mathbb{C}\ni z\mapsto K_{z}$ is a {$\mathcal{L}$}$(H^{1/2}%
(\Gamma))\,$-valued analytic map and
\begin{equation}
K_{z}=K_{0}+K_{(2)}\,z^{2}+\sum_{n=3}^{+\infty}K_{(n)}\,z^{n}\,, \label{tqq}%
\end{equation}
where the series converges in {$\mathcal{L}$}$_{H\!S}(H^{1/2}(\Gamma))$. On
smooth functions, $K_{(n)}$ has the integral representation
\begin{equation}
K_{(n)}\phi(x)=-(n-1)\,\frac{1}{4\pi}\,\frac{i^{n}}{n!}\,\int_{\Gamma}%
\nu(y)\!\cdot\!(x-y)|x-y|^{n-3}\phi(y)\,d\sigma(y)\,. \label{Kn}%
\end{equation}

\end{lemma}

\begin{proof}
We proceed as in the proof of Lemma \ref{anS}; see that proof for the
definitions of $\phi_{k}^{\pm}$. For any couple of indices $(i,j)$, one has
\begin{align*}
&  \langle\phi_{i}^{+},K_{(n)}\phi_{j}^{+}\rangle_{H^{1/2}(\Gamma)}%
=\langle\phi_{i}^{-},K_{(n)}\phi_{j}^{+}\rangle_{L^{2}(\Gamma)}=\int%
_{\Gamma\times\Gamma}\phi_{i}^{-}(x)\,\nu(y)\!\cdot\!\nabla_{y}\mathcal{G}%
_{(n)}(x-y)\,\phi_{j}^{+}(y)\,d\sigma(x)d\sigma(y)\\
=  &  \langle\phi_{i}^{-}\otimes\phi_{j}^{+},\mathcal{R}_{n}^{\prime}%
\rangle_{L^{2}(\Gamma)\otimes L^{2}(\Gamma)}=\langle\varphi_{i}\otimes
\varphi_{j},((-\Delta_{LB}+1)^{1/4}\otimes(-\Delta_{LB}+1)^{-1/4}%
)\,\mathcal{R}_{n}^{\,\prime}\rangle_{L^{2}(\Gamma)\otimes L^{2}(\Gamma)}\,,
\end{align*}
where
\begin{equation}
\mathcal{R}_{n}^{\,\prime}(x,y):=-(n-1)\,\frac{1}{4\pi}\,\frac{i^{n}}{n!}%
\,\nu(y)\!\cdot\!(x-y)|x-y|^{n-3}\,. \label{R'}%
\end{equation}
Therefore, for any $n\geq3$,
\begin{align*}
&  \left.  \Vert K_{(n)}\Vert_{H\!S}^{2}=\sum_{k=1}^{\infty}\Vert K_{(n)}%
\phi_{k}^{+}\Vert_{H^{1/2}}^{2}\right. \\
&  \left.  =\sum_{i,j=1}^{\infty}\left\vert \langle\varphi_{i}\otimes
\varphi_{j},((-\Delta_{LB}+1)^{1/4}\otimes(-\Delta_{LB}+1)^{-1/4}%
)\mathcal{R}_{n}^{\,\prime}\rangle_{L^{2}(\Gamma)\otimes L^{2}(\Gamma
)}\right\vert ^{2}\right. \\
&  \left.  =\Vert((-\Delta_{LB}+1)^{1/4}\otimes(-\Delta_{LB}+1)^{-1/4}%
)\mathcal{R}_{n}^{\,\prime}\Vert_{L^{2}(\Gamma)\otimes L^{2}(\Gamma)}%
^{2}\right. \\
&  \left.  \leq\Vert\mathcal{R}_{n}^{\,\prime}\Vert_{H^{1}(\Gamma\times
\Gamma)}\Vert\mathcal{R}_{n}^{\,\prime}\Vert_{L^{2}(\Gamma\times\Gamma
)}\right. \\
&  \left.  \leq\,c\ \frac{|\Gamma|^{2}}{(n!)^{2}}\ (n-1)^{2}\,d_{\Omega}%
^{n-2}\,\left(  d_{\Omega}^{n-2}+d_{\Omega}^{n-3}+(n-3)d_{\Omega}%
^{n-4}\right)  \,.\right.
\end{align*}
The proof is then concluded as in Lemma \ref{anS}. \hfill
\end{proof}

\subsection{\label{Sec_DN}The Dirichlet-to-Neumann operator}

In this section $z\in\mathbb{C}$ is such that $z^{2}\in\varrho(-\Delta
_{\Omega}^{D})$. The Dirichlet-to-Neumann operator
related to the interior Helmholtz equation is defined by
\begin{equation}
D\!N_{z}\varphi:=\gamma_{1}^{\mathrm{in}}u\,,\qquad%
\begin{cases}
\left(  \Delta_{\Omega}+z^{2}\right)  u=0\,,\\
\gamma_{0}^{\mathrm{in}}u=\varphi\,.
\end{cases}
\label{DN_k_def}%
\end{equation}
As is well known (see, e.g., \cite[Theorem 4.10]{McLe}) the solution exists
and is unique. By elliptic regularity (see, e.g., \cite[Theorem 4.21]{McLe}),
$D\!N_{z}$ extends to a pseudo-differential operator of order one on the whole
scale of Sobolev spaces $H^{s}(\Gamma)$:
\begin{equation}
D\!N_{z}\in{\mathcal{L}}(H^{s}(\Gamma),H^{s-1}(\Gamma))\,.
\label{DN_k_mapping}%
\end{equation}
If $z^{2}\in\mathbb{R}$, then $D\!N_{z}$ is self-adjoint as unbounded operator
between the dual couple $H^{s}(\Gamma)$-$H^{-s}(\Gamma)$, $D\!N_{z}%
:H^{s+1}(\Gamma)\subset H^{-s}(\Gamma)\rightarrow H^{s}(\Gamma)$ (see e.g. in
\cite[Sec. 2 and Example 4.9]{ArElKeSa}).

By \cite[Theorem 7.5]{McLe}, $u$ in \eqref{DN_k_def} is uniquely determined by
the solution of the boundary integral equation%
\begin{equation}
S_{z}\gamma_{1}^{\mathrm{in}}u=\left(  \frac{1}{2}+K_{z}\right)
\varphi\label{DN_k_eq}%
\end{equation}
so that $D\!N_{z}\in${$\mathcal{L}$}$(H^{1/2}(\Gamma),H^{-1/2}(\Gamma))$ has
the representation
\begin{equation}
D\!N_{z}=S_{z}^{-1}\left(  \frac{1}{2}+K_{z}\right)  \,. \label{DN_k_id}%
\end{equation}
By Lemma \eqref{anK} and Lemma \eqref{anS}, \eqref{DN_k_id} entails that
$\mathbb{C}\backslash D_{\Omega}\ni z\mapsto D\!N_{z}$ is a {$\mathcal{L}$%
}$\left(  H^{1/2}\left(  \Gamma\right)  ,H^{-1/2}\left(  \mathbb{R}%
^{3}\right)  \right)  \,$-valued analytic map.

\subsection{Further auxiliary operator expansions}

\begin{lemma}
\label{Lemma_R_k_eps_exp} Let $\operatorname{Im}z\geq0$, $z\not =0$, let
$y_{0}\in\mathbb{R}^{3}$ and define the linear operator $\Psi_{\varepsilon,z}$
by
\begin{equation}
(\Psi_{\varepsilon,z}u)( y) =\left(  R_{z}u\right)  \left(  y_{0}%
+\varepsilon\,\left(  y-y_{0}\right)  \right)  -\left(  R_{z}u\right)  (
y_{0}) \,. \label{R_k_eps_exp}%
\end{equation}
Then, for any star-shaped, bounded open set $B\subset\mathbb{R}^{3}$, one has
the estimate
\begin{equation}
\left\Vert \Psi_{\varepsilon,z}\right\Vert _{L_{\alpha}^{2}\left(
\mathbb{R}^{3}\right)  ,H^{2}\left(  B\right)  }\leq c_{\alpha,B}\,\Vert
R_{z}\Vert_{L_{\alpha}^{2}(\mathbb{R}^{3}),H_{-\alpha}^{2}(\mathbb{R}^{3}%
)}\,\varepsilon^{1/2}\,, \label{R_k_eps_rem_est_{lap}}%
\end{equation}
where $\alpha=0$ whenever $\operatorname{Im}z>0$ and $\alpha>1/2$ whenever
$\operatorname{Im}z=0$.
\end{lemma}

\begin{proof}
Without loss of generality we can suppose that $y_{0}=0$. Given $u\in
L_{\alpha}^{2}(\mathbb{R}^{3})$, let us set
\[
\psi_{\varepsilon}(y):=\psi(\varepsilon y)\,,\qquad\psi:=R_{z}u\,.
\]
By $R_{z}\in\mathcal{L}(L_{\alpha}^{2}(\mathbb{R}^{3}),H_{-\alpha}%
^{2}(\mathbb{R}^{3}))$ (see Subsection \ref{Section_LAP}), one has
\[
\Vert\psi\Vert_{H_{-\alpha}^{2}(\mathbb{R}^{3})}^{2}\leq\Vert R_{z}%
\Vert_{L_{\alpha}^{2}(\mathbb{R}^{3}),H_{-\alpha}^{2}(\mathbb{R}^{3})}\,\Vert
u\Vert_{L_{\alpha}^{2}(\mathbb{R}^{3})}^{2}%
\]
and so it suffices to show that
\begin{align}
\Vert\psi_{\varepsilon}-\psi_{0}\Vert_{H^{2}(B)}^{2}=  &  \sum_{1\leq
i,j\leq3}\Vert\partial_{ij}^{2}\psi_{\varepsilon}\Vert_{L^{2}(B)}^{2}%
+\sum_{1\leq i\leq3}\Vert\partial_{i}\psi_{\varepsilon}\Vert_{L^{2}(B)}%
^{2}+\Vert\psi_{\varepsilon}-\psi_{0}\Vert_{L^{2}(B)}^{2}\nonumber\\
& \nonumber\\
&  \left.  \leq\,c_{\alpha,B}\,\Vert\psi\Vert_{H_{-\alpha}^{2}(\mathbb{R}%
^{3})}^{2}\,\varepsilon^{1/2}\,.\right.  \label{sts}%
\end{align}
The latter is consequence of the following estimates:

\noindent1)
\begin{align*}
&  \left.  \Vert\partial_{ij}^{2}\psi_{\varepsilon}\Vert_{L^{2}(B)}^{2}%
=\int_{B}|\partial_{ij}^{2}\psi_{\varepsilon}(y)|^{2}dy=\varepsilon
^{4}\!\!\int_{B}|\partial_{ij}^{2}\psi({\varepsilon}y)|^{2}dy=\varepsilon
\!\!\int_{B_{\varepsilon}}|\partial_{ij}^{2}\psi(y)|^{2}dy\leq\varepsilon
\!\!\int_{B}|\partial_{ij}^{2}\psi(y)|^{2}dy\right. \\
& \\
&  \left.  \leq\varepsilon\,\Vert\langle x\rangle^{\alpha}\Vert_{L^{\infty
}(B)}\!\!\int_{B}|\partial_{ij}^{2}\psi(y)|^{2}\langle y\rangle^{-\alpha
}\,dy\leq\varepsilon\,\Vert\langle x\rangle^{\alpha}\Vert_{L^{\infty}%
(B)}\,\Vert\partial_{ij}^{2}\psi\Vert_{L_{-\alpha}^{2}(\mathbb{R}^{3})}%
^{2}\leq\varepsilon\,\Vert\langle x\rangle^{\alpha}\Vert_{L^{\infty}(B)}%
\Vert\psi\Vert_{H_{-\alpha}^{2}(\mathbb{R}^{3})}^{2}\,;\right.
\end{align*}
2) by the Sobolev embedding $H^{1}(B)\subset L^{6}(B)$,
\begin{align*}
&  \left.  \Vert\partial_{i}\psi_{\varepsilon}\Vert_{L^{2}(B)}^{2}%
\leq|B|^{\frac{3}{2}}\left(  \int_{B}|\partial_{i}\psi_{\varepsilon}%
(y)|^{6}\,dy\right)  ^{\frac{1}{3}}=|B|^{\frac{3}{2}}\left(  \varepsilon
^{6}\int_{B}|\partial_{i}\psi(\varepsilon y|^{6}\,dy\right)  ^{\frac{1}{3}%
}=|B|^{\frac{3}{2}}\varepsilon\,\Vert\partial_{i}\psi\Vert_{L^{6}%
(B_{\varepsilon})}^{2}\right. \\
& \\
&  \left.  \leq|B|^{\frac{3}{2}}\varepsilon\,\Vert\partial_{i}\psi\Vert
_{L^{6}(B)}^{2}\leq c\,|B|^{\frac{3}{2}}\varepsilon\,\Vert\partial_{i}%
\psi\Vert_{H^{1}(B)}^{2}\leq c\,|B|^{\frac{3}{2}}\varepsilon\,\Vert\psi
\Vert_{H^{2}(B)}^{2}\leq c\,|B|^{\frac{3}{2}}\varepsilon\,\Vert\langle
x\rangle^{\alpha}\Vert_{L^{\infty}(B)}\Vert\psi\Vert_{H_{-\alpha}%
^{2}(\mathbb{R}^{3})}^{2}\,;\right.
\end{align*}
3) by the continuous embedding of $H^{2}(B)$ into the space of
H\"{o}lder-continuous functions of order $\frac{1}{2}$,
\[
\Vert\psi_{\varepsilon}-\psi_{0}\Vert_{L^{2}(B)}^{2}=\int_{B}|\psi(\varepsilon
y)-\psi(0)|^{2}\,dy\leq c\,\varepsilon\int_{B}|y|\,dy\,\,\Vert\psi\Vert
_{H^{2}(B)}^{2}\leq c\,\varepsilon\,\Vert\langle x\rangle^{\alpha}%
\Vert_{L^{\infty}(B)}\int_{B}|y|\,dy\,\,\Vert\psi\Vert_{H_{-\alpha}%
^{2}(\mathbb{R}^{3})}^{2}\,.
\]
\hfill
\end{proof}

\begin{corollary}
\label{Lemma_trace_R_eps_k_U_eps} Let $\operatorname{Im}z\geq0$, $z\not =0$,
let $y_{0}\in\mathbb{R}^{3}$ and define the linear operators $\Xi
_{\varepsilon,z}$ and $\Phi_{\varepsilon,z}$ by
\[
\Xi_{\varepsilon,z}u=\gamma_{0}R_{\varepsilon z}U_{\varepsilon}^{-1}%
u-\varepsilon^{-1/2}\left(  R_{z}u\right)  (y_{0})\,,
\]%
\[
\Phi_{\varepsilon,z}\,\phi=U_{\varepsilon}R_{\varepsilon z}\gamma_{0}^{\ast
}\phi-\varepsilon^{-1/2}\langle\phi\rangle\ \mathcal{G}_{z}\,,\quad
\qquad\langle\phi\rangle:=\langle1,\phi\rangle_{H^{3/2}(\Gamma),H^{-3/2}%
(\Gamma)}\,.
\]
Then, for any $\epsilon>0$,
\begin{equation}
\left\Vert \Xi_{\varepsilon,z}\right\Vert _{L_{\alpha}^{2}\left(
\mathbb{R}^{3}\right)  ,H^{3/2}\left(  \Gamma\right)  }\leq c_{\alpha
,B}\,\Vert\gamma_{0}\Vert_{H^{2}(B),H^{3/2}\left(  \Gamma\right)  }\,\Vert
R_{z}\Vert_{L_{\alpha}^{2}(\mathbb{R}^{3}),H_{-\alpha}^{2}(\mathbb{R}^{3})}
\label{Xi}%
\end{equation}
and
\begin{equation}
\left\Vert \Phi_{\varepsilon,z}\right\Vert _{H^{-3/2}(\Gamma),L_{\alpha}%
^{2}(\mathbb{R}^{3})}\leq c_{\alpha,B}\,\Vert\gamma_{0}\Vert_{H^{2}%
(B),H^{3/2}\left(  \Gamma\right)  }\,\Vert R_{z}\Vert_{L_{\alpha}%
^{2}(\mathbb{R}^{3}),H_{-\alpha}^{2}(\mathbb{R}^{3})}\,, \label{Phi}%
\end{equation}
where $B$ is any star-shaped open and bounded set such that $B\supset
\overline{\Omega}$, $c_{\alpha,B}$ and $\alpha$ are the same as in Lemma
\ref{Lemma_R_k_eps_exp}.
\end{corollary}

\begin{proof}
By \eqref{Res_conj}, one has
\[
(R_{\varepsilon z}U_{\varepsilon}^{-1}u)(y)=\varepsilon^{-2}(U_{\varepsilon
}^{-1}R_{z}u)(y)=\varepsilon^{-1/2}(R_{z}u)(y_{0}+\varepsilon(y-y_{0}))
\]
and so, by \eqref{R_k_eps_exp},
\[
\Xi_{\varepsilon,z}u=\gamma_{0}R_{\varepsilon z}U_{\varepsilon}^{-1}%
u-\varepsilon^{-1/2}\left(  R_{z}u\right)  (y_{0})=\varepsilon^{-1/2}%
\gamma_{0}\Psi_{\varepsilon,z}\,.
\]
Hence, whenever $B\supset\overline{\Omega}$, one gets
\[
\left\Vert \Xi_{\varepsilon,z}\right\Vert _{L_{\alpha}^{2}\left(
\mathbb{R}^{3}\right)  ,H^{3/2}\left(  \Gamma\right)  }\leq\Vert\gamma
_{0}\Vert_{H^{2}(B),H^{3/2}\left(  \Gamma\right)  }\,\varepsilon^{-1/2}%
\Vert\Psi_{\varepsilon,z}\Vert_{L_{\alpha}^{2}(B),H^{2}(B)}%
\]
and \eqref{Xi} follows from \eqref{R_k_eps_rem_est_{lap}}. By
\[
\left\langle u,U_{\varepsilon}R_{\varepsilon z}\gamma_{0}^{\ast}%
\phi\right\rangle _{L^{2}\left(  \mathbb{R}^{3}\right)  }=\left\langle
U_{\varepsilon}^{-1}\gamma_{0}R_{-\varepsilon\bar{z}}u,\phi\right\rangle
_{_{H^{3/2}(\Gamma),H^{-3/2}(\Gamma)}}%
\]
and by
\[
\left\langle u,\langle\phi\rangle\mathcal{G}_{z}\right\rangle _{L^{2}\left(
\mathbb{R}^{3}\right)  }=\langle\phi\rangle(R_{z}u)(y_{0})=\left\langle
(R_{-\bar{z}}u)(y_{0}),\phi\right\rangle _{_{H^{3/2}(\Gamma),H^{-3/2}(\Gamma
)}}\,,
\]
one gets
\[
\Phi_{z,\varepsilon}=\Xi_{-\bar{z},\varepsilon}^{\ast}\,
\]
and so \eqref{Phi} is consequence of \eqref{Xi}. \hfill
\end{proof}

\begin{lemma}
\label{Lemma_S_DN} For any $z\in\mathbb{C}$, one has
\[
S_{z}D\!N_{z}=Q_{0}\left(  \frac{1}{2}+K_{0}\right)  Q_{0}+K_{(2)}%
\,z^{2}+O(|z|^{3})\,.
\]

\end{lemma}

\begin{proof}
By \eqref{DN_k_id}, there follows $S_{z}D\!N_{z}=\left(  1/2+K_{z}\right)  $.
By \eqref{anK}, one has
\[
S_{z}D\!N_{z}=\frac{1}{2}+K_{0}+K_{(2)}\,z^{2}+O(|z|^{3})\,.
\]
The proof is then concluded by \eqref{K_0_spectral_decomp}.\hfill
\end{proof}

\end{document}